\documentclass[12pt]{article}%***
\usepackage[sectionbib]{natbib}
\usepackage{array,epsfig,fancyheadings,rotating}
\usepackage[]{hyperref}  %<----modified by Ivan
\usepackage{sectsty, secdot}
\usepackage{verbatim}
\sectionfont{\fontsize{12}{14pt plus.8pt minus .6pt}\selectfont}
\renewcommand{\theequation}{\thesection\arabic{equation}}
\subsectionfont{\fontsize{12}{14pt plus.8pt minus .6pt}\selectfont}
\RequirePackage{graphicx}%% uncomment this for including figures
\usepackage{amsmath}
\usepackage{amssymb}
\usepackage{amsfonts}
\usepackage{multirow}
\usepackage{amsthm}
\usepackage{dsfont}   
\usepackage{algorithm}
\usepackage{algorithmic}

\newtheorem{theorem}{Theorem}
\newtheorem{lemma}{Lemma}
\newtheorem{corollary}{Corollary}
\newtheorem{proposition}{Proposition}
\theoremstyle{definition}
\newtheorem{definition}{Definition}

\newtheorem{remark}{Remark}

\newcommand{\rom}[1]{\uppercase\expandafter{\romannumeral #1\relax}}
\newcommand{\beas}{\begin{eqnarray*}}
\newcommand{\enas}{\end{eqnarray*}}
\newcommand{\bea}{\begin{eqnarray}}
\newcommand{\ena}{\end{eqnarray}}
\newcommand{\bms}{\begin{multline*}}
\newcommand{\ems}{\end{multline*}}
\newcommand{\bels}{\begin{align*}}
\newcommand{\enls}{\end{align*}}
\newcommand{\bel}{\begin{align}}
\newcommand{\enl}{\end{align}}

\newcommand{\ignore}[1]{}
\newcommand{\tr}{\mbox{tr\,}}

\def\blfootnote{\xdef\@thefnmark{}\@footnotetext}

\newcommand{\expect}[1]{\mathbb{E}{\l[#1\r]}}

\newcommand{\dotp}[2]{\left\langle#1,#2\right\rangle}
\newcommand{\m}{\mathcal}
\newcommand{\mb}{\mathbb}
\newcommand\argmin{\mathop{\mbox{argmin}}}

\newcommand{\rank}{\mathrm{\,rank}}

\newcommand{\sign}{\mathrm{sign}}
\newcommand{\rk}{\mathrm{rk}}
\newcommand{\im}{\mathrm{Im}}
\newcommand{\lspan}{\mathrm{span}}
\def\r{\right}
\def\l{\left}
\newcommand{\eps}{\varepsilon}

\newcommand{\F}{\mathrm{F}}
\newcommand{\wh}{\widehat}
\newcommand{\wt}{\widetilde}
\newcommand{\prox}{\mathrm{prox}}

\newcommand{\vertiii}[1]{{\left\vert\kern-0.25ex\left\vert\kern-0.25ex\left\vert #1 
    \right\vert\kern-0.25ex\right\vert\kern-0.25ex\right\vert}}

\begin{document}
	
	%%%%%%%%%%%%%%%%%%%%%%%%%%%%%%%%%%%%%%%%%%%%%%%%%%%%%%%%%%%%%%%%%%%%%%%%%%%%%%%%%%%%%%%%%%%%%%%%%%%%%%%%%%%%%%%%%%%%%%%%%%%%
	%%%%%%%%%%%%%%%%%%%%%%%%%%%%%%%%%%%%%%%%%%%%%%%%%%%%%%%%%%%%%%%%%%%%%%%%%%%%%%%%%%%%%%%%%%%%%%%%%%%%%%%%%%%%%%%%%%%%%%%%%%%%
	
	\renewcommand{\baselinestretch}{2}
	
	\markright{ \hbox{\footnotesize\rm %Statistica Sinica
			%{\footnotesize\bf 24} (201?), 000-000
		}\hfill\\[-13pt]
		\hbox{\footnotesize\rm
			%\href{http://dx.doi.org/10.5705/ss.20??.???}{doi:http://dx.doi.org/10.5705/ss.20??.???}
		}\hfill }
	
	\markboth{\hfill{\footnotesize\rm Stanislav Minsker and Lang Wang} \hfill}
	{\hfill {\footnotesize\rm Robust covariance estimation} \hfill}
	
	\renewcommand{\thefootnote}{}
	$\ $\par
	
	%%%%%%%%%%%%%%%%%%%%%%%%%%%%%%%%%%%%%%%%%%%%%%%%%%%%%%%%%%%%%%%%%%%%%%%%%%%%%%%%%%%%%%%%%%%%%%%%%%%%%%%%%%%%%%%%%%%%%%%%%%%%
	
	\fontsize{12}{14pt plus.8pt minus .6pt}\selectfont \vspace{0.8pc}
	\centerline{\large\bf Robust Estimation of Covariance Matrices: }
	\vspace{2pt} 
	\centerline{\large\bf Adversarial Contamination and Beyond}
	\vspace{.4cm} 
	\centerline{Stanislav Minsker and Lang Wang} 
	\vspace{.4cm} 
	\centerline{\it University of Southern California}
	\vspace{.55cm} \fontsize{9}{11.5pt plus.8pt minus.6pt}\selectfont
	
	%%%%%%%%%%%%%%%%%%%%%%%%%%%%%%%%%%%%%%%%%%%%%%%%%%%%%%%%%%%%%%%%%%%%%%%%%%%%%%%%%%%%%%%%%%%%%%%%%%%%%%%%%%%%%%%%%%%%%%%%%%%%
	
	\begin{quotation}
		\noindent {\it Abstract:}
%		{\bf Contents of the Abstract.}\\
We consider the problem of estimating the covariance structure of a random vector $Y\in \mb R^d$ from a sample $Y_1,\ldots,Y_n$. We are interested in the situation when $d$ is large compared to $n$ but the covariance matrix $\Sigma$ of interest has (exactly or approximately) low rank. We assume that the given sample is (a) $\eps$-adversarially corrupted, meaning that $\eps$ fraction of the observations could have been replaced by arbitrary vectors, or that (b) the sample is i.i.d. but the underlying distribution is heavy-tailed, meaning that the norm of $Y$ possesses only finite fourth moments. We propose an estimator that is adaptive to the potential low-rank structure of the covariance matrix as well as to the proportion of contaminated data, and admits tight deviation guarantees despite rather weak assumptions on the underlying distribution. Finally, we discuss the algorithms that allow to approximate the proposed estimator in a numerically efficient way.
		
		\vspace{9pt}
		\par
	\end{quotation}\par

	\def\thefigure{\arabic{figure}}
	\def\thetable{\arabic{table}}
	
	\renewcommand{\theequation}{\thesection.\arabic{equation}}

	\fontsize{12}{14pt plus.8pt minus .6pt}\selectfont
	
%==============================
%  Section: introduction
%==============================	
\section{Introduction}
\label{sec:topic1_introduction}

In this paper, we consider the problem of estimating covariance matrices under various types of contamination: we are given independent copies $Y_1, \ldots, Y_n$ of a random vector $Y \in \mb{R}^{d}$ which follows an unknown distribution $\m D$ over $\mb R^{d}$ 
with mean $\mu$ and covariance matrix $\Sigma$, and we assume that (a) the observations $Y_1, \ldots, Y_n$ are \textit{$\eps$-adversarially corrupted}, meaning that $\eps$ fraction of them could have been replaced by arbitrary (possibly random) vectors, or that (b) the underlying distribution $\m D$ is heavy-tailed, meaning that only the fourth moment of $Y$ is finite. Our goal is to construct a robust estimator for the covariance matrix $\Sigma$ in this framework.

%with $n$ observations from \textit{$\eps$-corrupted} samples: we are given $n$ independent observations from some unknown distribution $\m D$ over $\mb R^{d}$ with mean $\mu$ and covariance matrix $\Sigma$, where an $\eps$-fraction of them are adversarially corrupted. More specifically, we assume that our observations $Y_1,\ldots, Y_n$ satisfy the equation 
%\[
%Y_j=Z_j+V_j,j=1,\ldots,n,
%\]
%where $Z_j$'s come from the target distribution $\m D$ and $V_j$'s are arbitrary (possibly random) vectors such that only an $\eps$-fraction of them are non-zeros. Our goal is to construct a robust estimator for the covariance matrix $\Sigma$ in this framework.

As attested by some early references such as the works \citet{tukey1960survey, huber1992robust}, robust estimation has a long history. During the past two decades, increasing amount of practical applications created a high demand for the tools to recover high-dimensional parameters of interest from grossly corrupted measurements. 
Robust covariance estimators in particular have been studied extensively, see e.g. \citet{huber1992robust,huber2011robust, maronna2019robust}. Although some of the proposed estimators admit theoretically optimal error bounds, they are hard to compute in general when the dimension is high because the running time is exponential in the dimension (\citet{bernholt2006robust}).

Recent work by \citet{lai2016agnostic,diakonikolas2019robust} introduced the first robust estimators for the covariance matrix $\Sigma$ that are computationally efficient in the high-dimensional case, i.e. the running time is only polynomial in the dimension, assuming that the distribution $\m D$ is Gaussian or an affine transformation of a product distribution with a bounded fourth moment. Since the publication of these initial papers, a growing body of subsequent works has appeared. For instance, \citet{,cheng2019faster} developed fast algorithms that nearly match the best-known running time to compute the empirical covariance matrix, assuming that the distribution of $Y$ is Gaussian with zero mean. \citet{chen2018robust} developed efficient algorithms under significantly weaker conditions on the unknown distribution $\m D$, i.e. $\m D$ does not have to be an affine transformation of a product distribution. However, these algorithms either require prior knowledge on the fraction of outliers, or can only achieve a theoretically suboptimal error bound in the Frobenius norm.

The present paper continues this line of research. We design a double-penalized estimator for the covariance matrix $\Sigma$, which will be shown to admit theoretically optimal error bounds when the ``effective rank'' of $\Sigma$ (to be defined later) is small, and can be efficiently calculated using traditional numerical methods. %Moreover, recent work such as \citet{prasad2019unified} introduced methods that allow us to connect the heavy-tailed data to the $\espilon$-corruption data.

The rest of the paper is organized as follows.
Section \ref{sec:prelim} explains the main notations and background material. Section \ref{section:main} and  \ref{sec:heavy-tailed} displays the main results for $\epsilon$-adversarially corrupted data and heavy-tailed data, respectively. %Section \ref{sec:heavy-tailed} displays results for heavy-tailed data. 
Section \ref{sec:numerical simulation} presents analysis and algorithms for numerical experiments.
%and Section \ref{sec:future work} summarizes potential directions of future work for continuing study. 
Finally, the numerical results and proofs are contained in the supplementary material.

%################# Preliminaries ###########
\section{Preliminaries}
\label{sec:prelim}
%#################
\begingroup
\allowdisplaybreaks
In this section, we introduce the main notations and recall some useful facts that we rely on in the subsequent exposition.
Given two real numbers $a,b \in \mb R$, we define $a\vee b:=\max\{a,b\}$, $a\wedge b:= \min\{a,b\}$. Also, given $x \in \mb{R}$, we will denote $\lfloor x \rfloor := \max\{n\in \mb Z: n\leq x\}$ to be the largest integer less than or equal to $x$. We will separately introduce important results of matrix algebra and sub-Gaussian distributions in the following two subsections.

%###########################
\subsection{Matrix algebra}
\label{sec:matrix algebra}
%###########################
Assume that $A\in \mb R^{d_1\times d_2}$ is a $d_1 \times d_2$  matrix with real-valued entries. Let $A^T$ denote the transpose of $A$. A square matrix $A \in \mb R^{d\times d}$ is called an orthogonal matrix if $AA^T=A^TA=I_d$, where $I_d$ is the identity matrix in $\mb R^{d\times d}$. 
	\begin{comment}

The following theorem is well known:
\begin{theorem}[{Singular Value Decomposition}]
\label{thm:svd}
Let $A\in \mb R^{d_1\times d_2}$ be a $d_1 \times d_2$ real matrix, then the singular value decomposition of $A$ always exists, and is of the form
\[
A=U\Sigma V^T,
\]
where
\begin{itemize}
\item U is a $d_1\times d_1$ orthogonal matrix, whose columns are the eigenvectors of $AA^T$.
\item V is a $d_2\times d_2$ orthogonal matrix, whose columns are the eigenvectors of $A^TA$.
\item $\Sigma$ is a $d_1 \times d_2$ matrix with $\Sigma_{1,1}=\sigma_1(A)$,\ldots, $\Sigma_{r,r}=\sigma_r(A)$, $\Sigma_{k,k}=0$ for all $ k>r$ and $\Sigma_{i,j}=0 \ \forall i\neq j$, where $r=\rank(A)$, and $\sigma_1(A)\geq\sigma_2(A)\geq\cdots\geq\sigma_r(A)\geq0$ are the square roots of positive eigenvalues of $A^TA$. They are called the singular values of A.
\end{itemize}
\end{theorem}

	\end{comment}
%===
Given a square matrix $A \in \mb{R}^{d\times d}$, we define the trace of $A$ to be the sum of elements on the main diagonal, namely, $\tr(A) := \sum_{i=1}^{d} A_{i,i}$, 
where $A_{i,i}$ represents the element on the $i^{th}$ row and $i^{th}$ column of $A$.
We introduce three types of matrix norms and the Frobenius (or Hilbert-Schmidt) inner product as follows:
\begin{definition}[Matrix norms]
\label{def:matrix norms}
Given $A\in \mb R^{d_1\times d_2}$ with singular values $\sigma_1(A) \geq \cdots \geq \sigma_{\rank(A)}(A) \geq 0$, we define the following three types of matrix norms.
\begin{enumerate}
\item Operator norm:
\begin{equation*}
\l\|A\r\| := \sigma_1(A) = \sqrt{\lambda_{max}(A^TA)}, 
\end{equation*}
where $\lambda_{max}(A^TA)$ stands for the largest eigenvalue of $A^TA$.
\item Frobenius norm:
\begin{equation*}
\l\|A\r\|_F := \sqrt{\sum_{i=1}^{\rank(A)}\sigma_i^2(A)} = \sqrt{\tr(A^TA)}.
\end{equation*}
\item Nuclear norm:
\begin{equation*}
\l\|A\r\|_1:=\sum_{i=1}^{\rank(A)}\sigma_i(A)=\tr(\sqrt{A^TA}), 
\end{equation*}
where $\sqrt{A^TA}$ is a nonnegative definite matrix such that $\Big(\sqrt{A^TA} \Big)^2=A^TA$.
\end{enumerate}
\end{definition}
%===
\begin{definition}%[\citet{fasino2016localization}]
\label{def:Frobenius inner product}
Given $A,B\in \mb R^{d_1 \times d_2}$, we define the Frobenius inner product as
\begin{equation*}
\dotp{A}{B} := \dotp{A}{B}_F = \tr(A^TB) =\tr(AB^T).
\end{equation*}
It is clear that $\l\|A\r\|_F^2=\dotp{A}{A}$.
\end{definition}
%====
We will now introduce matrix functions. Denote $S^d(\mb R):=\big\{A\in \mb R^{d\times d}: A^T=A\big\}$ 
to be the set of all symmetric matrices. The eigenvalues of $A$ will be denoted $\lambda_1,\ldots,\lambda_d$, all of which are real numbers. Next, we define functions acting on $S^d(\mb R)$ as follows:
\begin{definition}
\label{matrix-function}
Given a real-valued function $f$ defined on an interval $\mb T\subseteq \mb R$ and a real symmetric matrix $A\in S^d(\mb R)$ with the spectral decomposition 
$A=U\Lambda U^T$ such that $\lambda_j(A)\in \mb T,\ j=1,\ldots,d$, define $f(A)$ as 
$f(A)=Uf(\Lambda) U^T$, where 
\[
f(\Lambda)=f\l( \begin{pmatrix}
\lambda_1 & \,  & \,\\
\, & \ddots & \, \\
\, & \, & \lambda_d
\end{pmatrix} \r)
=\begin{pmatrix}
f(\lambda_1) & \,  & \,\\
\, & \ddots & \, \\
\, & \, & f(\lambda_d)
\end{pmatrix}.
\] 
\end{definition}
%===
Finally, the effective rank of a matrix $A\in S^d(\mb{R})\setminus\{0\}$ is defined as
\[
\rk(A):=\frac{\tr(A)}{\l\|A\r\|}.
\]
Note that $1\leq \rk(A) \leq \rank(A)$ is always true, and it is possible that $\rk(A) \ll \rank(A)$ for approximately low-rank matrices $A$. For instance, consider $A\in S^{d}(\mb R)$ with eigenvalues $\lambda_1 = 1, \lambda_2=\cdots=\lambda_d=1/d$, whence we have $\rk(A) = {2-1/d} \ll d = \rank(A)$.

%####################################
%=========Sec: sub-Gaussian===========
\subsection{Sub-Gaussian distributions}
\label{sec:sub-gaussian}
%####################################

Given a random variable $X$ on a probability space $(\Omega,\mathcal{A},\mb P)$, and a convex nondecreasing function $\psi: \mb R_+ \mapsto \mb R_+$ with $\psi(0)=0$%and $\psi(x) \to \infty$ as $x \to \infty$
, we define the $\psi$-norm of $X$ as (\citet[Section 2.7.1]{vershynin2018high})
\begin{equation*}
\l\|X\r\|_\psi := \inf \l\{C>0: \expect{\psi\l(\frac{|X|}{C} \r)} \leq 1 \r\} \nonumber.
\end{equation*}
In particular, in what follows we will mainly consider $\psi_1(u):=\exp{\{u\}}-1,u\geq 0$ and $\psi_2(u):=\exp{\{u^2\}}-1,u\geq0$, which correspond to the sub-exponential norm and sub-Gaussian norm respectively. 
We will say that a random variable $X$ is sub-Gaussian (or sub-exponential) if $\l\|X\r\|_{\psi_2}<\infty$ (or $\l\|X\r\|_{\psi_1}<\infty$). Also, let $\l\|X\r\|_{L_2}:=\l(\expect{|X|^2}\r)^{1/2}$ be the $L_2$ norm of a random variable $X$. The sub-Gaussian (or sub-exponential) random vector is defined as follows:
%=======Def: sub-Gaussian vector========
\begin{definition}
\label{def:sub-gaussian}
A random vector $Z$ in $\mb R^d$ with mean $\mu=\expect{Z}$ is called L-sub-Gaussian if for every $v\in \mb R^d$, there exists an absolute constant $L>0$ such that
\begin{equation}\label{L-sub-gaussian}
\l\|\dotp{Z-\mu}{v}\r\|_{\psi_2} \leq L \l\|\dotp{Z-\mu}{v}\r\|_{L_2}. 
\end{equation}
Moreover, Z is called L-sub-exponential if $\psi_2$-norm in \eqref{L-sub-gaussian} is replaced by $\psi_1$-norm.
\end{definition}
%==
It is clear that if $Z$ is L-sub-Gaussian, then $(-Z)$ is also L-sub-Gaussian. We introduce some important results for sub-Gaussian distributions.
\begin{proposition}{(\citet[pp.24]{vershynin2018high})}
\label{prop:sub-gaussian}
A mean zero random variable Z is L-sub-Gaussian if and only if there exists an absolute constant $K(L)>0$ depending only on $L$ such that 
\[
P\l(|Z|\geq t\r)\leq 2\exp\{-t^2/K(L)^2\},\quad \text{for all } t\geq 0.
\]
\end{proposition}
%==
\begin{proposition} {(\citet[Theorem 2.6.3]{vershynin2018high})}
\label{thm:hoeffding}
Let  $Z_1,\ldots,Z_n$ be i.i.d L-sub-Gaussian random variables with mean zero, and $a=\l(a_1,\ldots,a_n\r)\in \mb R^{n}$. Then for any $t\geq0$, there exists a constant $K(L)>0$ depending only on $L$ such that 
\[
P\l( \l| \sum_{i=1}^{n}a_iZ_i \r| \geq t \r) \leq 2 \exp\l\{ -\frac{t^2}{K(L)^2\l\|a\r\|_2^2} \r\},
\]
where $\l\|a\r\|_2^2 = a_1^2 + \cdots + a_n^2$.
\end{proposition}
\begin{corollary}\label{sum of sub-gaussians}
Let $Z_1,\ldots,Z_n$ be i.i.d L-sub-Gaussian random variables with common mean $\expect{Z_1}=\mu$ and sub-Gaussian norm $\l\|Z_1\r\|_{\psi_2}=K$. Let $a = (a_1,\ldots,a_n)$ be a vector in $\mb R^d$ such that $\l\|a\r\|_2\leq 1$. Then
\begin{enumerate}
\item $Y:=\sum_{i=1}^{n}a_i(Z_i-\mu)$ is still L-sub-Gaussian.
\item There exists an absolute constant $c>0$, such that $\l\|Y\r\|_{\psi_2}\leq cK$.
\end{enumerate}
\end{corollary}
\begin{proof}
This corollary immediately follows by a combination of Theorem \ref{thm:hoeffding} and Proposition \ref{prop:sub-gaussian}.
\end{proof}
\endgroup

%#################################
%=======Sec: main=========
\section{Problem formulation and main results}
\label{section:main}
%#################################

%#######################################
%\subsection{Problem formulation}
%\label{section:problem formulation}
%#######################################

Let $Z_1,\ldots,Z_n\in \mb R^d$ be i.i.d. copies of an L-sub-Gaussian random vector $Z$ such that $\mb EZ = \mu$ and $\mb E (Z-\mu)(Z-\mu)^T = \Sigma$. Assume that we observe a sequence
\begin{equation}
\label{outlier model}
Y_j = Z_j + V_j, \ j=1,\ldots,n ,
\end{equation}
where $V_j$'s are arbitrary (possibly random) vectors such that only a small portion of them are not equal to zero. Namely, we assume that there exists a set of indices $J\subseteq \{1,\ldots,n\}$ such that $|J| \ll n$ and $V_j = 0$ for $j\notin J$. In what follows, the sample points with $j\in J$ will be called \textit{outliers} and $\eps:={|J|}/{n}$ will denote  the proportion of such points. 
In this case, 
\[
Y_j Y_j^T = Z_j Z_j^T + \underbrace{V_j V_j^T + V_j Z_j^T + Z_j V_j^T}_{:=\sqrt{n} U^\ast_j} := X_j + \sqrt{n} U^\ast_j,
\] 
where $\rank(U_j^*)\leq 2$ and the $\sqrt{n}$ factor is added for technical convenience. Our main goal is to construct an estimator for the covariance matrix $\Sigma$ in the presence of outliers $V_j$. 
%General \mu settings%
In practice, we usually do not know the true mean $\mu$ of $Z$. To address this problem, we first recall the definition of U-statistics.
%======Def: U-stat=======
\begin{definition}%[\citet{hoeffding1992class}]
Let $Y_1,\ldots,Y_n$ $(n\geq2)$ be a sequence of random variables taking values in a measurable space $(\m S, \m B)$. Assume that $H: \m S^{m}\mapsto \mb S^{d}(\mb R)$ $(2\leq m\leq n)$ is an $\m S^{m}$-measurable permutation-symmetric kernel, i.e. $H(y_1,\ldots, y_m) = H(y_{\pi_1},\ldots,y_{\pi_m})$ for any $(y_1,\ldots,y_m) \in \m S^{m}$ and any permutation $\pi$. The U-statistic with kernel $H$ is defined as
\[
U_n := \frac{(n-m)!}{n!} \sum_{(i_1,\dots,i_m) \in I_n^{m}} H(Y_{i_1},\ldots, Y_{i_m}),
\]
where $I_n^{m} := \l\{(i_1,\ldots,i_m): 1\leq i_j\leq n, i_j\neq i_k \text{ if } j\neq k \r\}$.
\end{definition}
%==================
A particular example of U-statistics is the sample covariance matrix defined as
\begin{equation}\label{sample covariance matrix}
\wt{\Sigma}_s := \frac{1}{n-1}\sum_{j=1}^{n}(Y_j-\bar Y)(Y_j-\bar Y)^T ,
\end{equation}
where $\bar Y:=\frac{1}{n}\sum_{j=1}^{n}Y_j$. Indeed, it is easy to verify that 
\begin{equation}\label{sample covariance matrix u-stat}
\wt{\Sigma}_s  = \frac{1}{n(n-1)} \sum_{(i,j)\in I_n^{2}}\frac{(Y_i-Y_j)(Y_i-Y_j)^T}{2} ,
\end{equation}
hence the sample covariance matrix is a U-statistic with kernel
\begin{equation*}%\label{kernel}
H(x,y) := \frac{(x-y)(x-y)^T}{2} \text{ for any } x,y\in \mb R^{d}.
\end{equation*}
Note that $\expect{ {(Y_i-Y_j)}/{\sqrt{2}}}=0$ and $\expect{{(Y_i-Y_j)} {(Y_i-Y_j)}^T / 2}=\Sigma$ for all $(i,j)\in I_n^2$. Namely, by expressing the sample covariance matrix as a U-statistic in \eqref{sample covariance matrix u-stat}, the explicit estimation of the unknown mean $\mu$ can be avoided. Therefore, we consider the following settings:
\begin{multline}\label{U-stat form of Y}
\wt{Y}_{i,j}:=\frac{Y_i-Y_j}{\sqrt{2}},\quad\widetilde{Z}_{i,j}:=\frac{Z_i-Z_j}{\sqrt{2}},\quad\widetilde{V}_{i,j}:=\frac{V_i-V_j}{\sqrt{2}}, \quad
\forall(i,j)\in I_n^2.
\end{multline}
Then
\begin{equation*}
\widetilde{Y}_{i,j}\widetilde{Y}_{i,j}^T = \widetilde{Z}_{i,j} \widetilde{Z}_{i,j}^T + \underbrace{\widetilde{V}_{i,j} \widetilde{V}_{i,j}^T + \widetilde{V}_{i,j} \widetilde{Z}_{i,j}^T + \widetilde{Z}_{i,j} \widetilde{V}_{i,j}^T}_{:=\sqrt{n(n-1)} \widetilde{U}^\ast_{i,j}} 
:= \wt{X}_{i,j} + \sqrt{n(n-1)} \widetilde{U}^\ast_{i,j} ,
\end{equation*}
where the $n(n-1)=|I_n^2|$ factor equals the total number of $\wt{Y}_{i,j}$'s, and is added for technical convenience. The followings facts can be easily verified: %It is easy to check that the following claims hold:
\begin{enumerate}\label{facts for Y tilde}
\item $\widetilde{Y}_{i,j}=\widetilde{Z}_{i,j}+\widetilde{V}_{i,j}$, with $\expect{\widetilde{Z}_{i,j}}=0$ and $\expect{\widetilde{Z}_{i,j}\widetilde{Z}_{i,j}^T}=\Sigma$, for any $(i,j)\in I_n^2$. Moreover, $\wt{Z}_{i,j},(i,j)\in I_n^2$ has sub-Gaussian distribution according to Corollary \ref{sum of sub-gaussians}.
\item $\wt{Z}_{i,j}$'s are identically distributed, but not independent. 
\item Denote $\widetilde{J}=\l\{(i,j)\in I_n^2:\widetilde{V}_{i,j}\neq0\r\}$ to be the set of indices such that $\widetilde{V}_{i,j}=0$, $\forall (i,j)\notin \wt{J}$. Then $|\wt{J}|$ represents the number of outliers in $\l\{\wt{Y}_{i,j}: (i,j)\in I_n^2 \r\}$, and we have that
\begin{equation}\label{bound_J}
|\wt{J}| = 2|J|(n-|J|) + |J|(|J|-1) = |J|(2n-|J|-1).
\end{equation}
\item $\mathrm{\,Rank}(\wt{U}_{i,j}^*)\leq 2$. This follows from the fact that for any vector $v\in \mb {R}^d$, $\wt{U}_{i,j}^*v \in \lspan\l\{ \wt{V}_{i,j}, \wt{Z}_{i,j} \r\}$.
\end{enumerate}
In what follows, we will use the notation $\mathbf{U_{I_n^2}}:=(U_{1,2},\ldots,U_{n,n-1})$
to represent the $n(n-1)$-dimensional sequence with subscripts taking from $I_n^2$. Similarly, the notation $(S,\mathbf{U_{I_n^2}})$ will represent the $(n^2-n+1)$-dimensional sequence $(S,U_{1,2},\ldots,\allowbreak U_{n,n-1})$.
Now we are ready to define our estimator. Given $\lambda_1,\lambda_2>0$, set
\begin{multline}\label{robust covariance estimation}
(\widehat{S}_\lambda,\mathbf{\widehat{U}_{I_n^2}}) = \argmin_{S,U_{1,2},\ldots,U_{n,n-1}} \Bigg[ \frac{1}{n(n-1)}\sum_{i\neq j} \l\| \widetilde{Y}_{i,j} \widetilde{Y}_{i,j} ^T - S - \sqrt{n(n-1)}U_{i,j} \r\|^2_\F \\
+ \lambda_1 \l\|S\r\|_1 + \lambda_2\sum_{i\neq j} \l\|U_{i,j} \r\|_1 \Bigg],
\end{multline}
%===
where the minimization is over $S, U_{i,j} \in S^{d}(\mb{R})$, $\forall (i,j) \in I_n^2$. 
\begin{remark}\label{relation to penalized huber}
The double penalized least-squares estimator in \eqref{robust covariance estimation} is indeed a penalized Huber's estimator (previously observed by \citet[Section 6]{donoho2016high} in the context of linear regression). To see this, we express \eqref{robust covariance estimation} as 
\begin{multline}\label{double minimum}
(\widehat{S}_\lambda,\mathbf{\widehat{U}_{I_n^2}}) = \mathrm{arg}\min_{S} \min_{\mathbf{{U}_{I_n^2}}} \Bigg[ \frac{1}{n(n-1)}\sum_{i\neq j} \l\| \widetilde{Y}_{i,j} \widetilde{Y}_{i,j} ^T - S - \sqrt{n(n-1)}U_{i,j} \r\|^2_\F \\
+\lambda_1 \l\|S\r\|_1 + \lambda_2\sum_{i\neq j} \l\|U_{i,j} \r\|_1 \Bigg],
\end{multline}
and observe that the minimization with respect to $\mathbf{{U}_{I_n^2}}$ in \eqref{double minimum} can be done explicitly. It yields that 
\begin{equation}\label{penalized Huber}
\wh{S}_\lambda = \argmin_{S} \Bigg\{\frac{2}{n(n-1)}\tr \bigg[ \sum_{i\neq j} \rho_{\frac{\sqrt{n(n-1)}\lambda_2}{2}}(\wt{Y}_{i,j}\wt{Y}_{i,j}^T - S) \bigg] 
+ \lambda_1\l\|S\r\|_1\Bigg\}, 
\end{equation}
where 
\begin{equation}\label{huber's function}
\rho_\lambda(u) := \l\{
			    \begin{array}{ll}
			    	\frac{u^2}{2},\quad\l|u\r|\leq\lambda\\
			    	\lambda\l|u\r|-\frac{\lambda^2}{2},\quad\l|u\r|>\lambda
			    \end{array}
			    \r. \quad \forall u\in \mb R, \lambda \in \mb R^+
\end{equation}
is the Huber's loss function. Details of the derivation are presented in section \ref{sec:derivation of penalized huber} of the supplementary material.
\end{remark}

%####################################
\subsection{Main results}
\label{sec:estimation unkown mean}
%####################################
\begingroup
\allowdisplaybreaks

We are ready to state the main results related to the error bounds for the estimator in \eqref{robust covariance estimation}. 
We will compare performance of our estimator to that of the sample covariance matrix $\wt{\Sigma}_s$ defined in \eqref{sample covariance matrix}. When there are no outliers, it is well-known that $\wt{\Sigma}_s$ is a consistent estimator of $\Sigma$ with expected error at most $\mathcal{O}(d/\sqrt{n})$ in the Frobenius norm, namely, $\expect{\l\|\wt{\Sigma}_s-\Sigma \r\|_F} \leq Cd/\sqrt{n}$ %with probability $99\%$ 
for some absolute constant $C>0$ (see for example, \citet{cai2010optimal}). However, in the presence of outliers, the error for $\wt{\Sigma}_s$ can be large (see section \ref{sec:numerical results with graphs} in the supplementary material for some examples). On the other hand, recall that $\wt X_{i,j} = \wt{Z}_{i,j}\wt{Z}_{i,j}^T$, and our estimator in \eqref{robust covariance estimation} admits the following bound.
%============ Main Theorem 1 =============
\begin{theorem}
\label{thm:lasso u-stat}
Let $\delta>0$ be an absolute constant. Assume that $n \geq 2$ % ,$\lambda_1 \leq \lambda_2$ 
and $|J| \leq c_1(\delta) n$, where $c_1(\delta)$ is a constant depending only on $\delta$. Then on the event 
\begin{multline*}
\m E=\Bigg\{ \lambda_1 \geq  \frac{140\l\|\Sigma\r\|}{\sqrt{n(n-1)}}\sqrt{\rk(\Sigma)} + 4\l\| \frac{1}{n(n-1)}\sum_{(i,j)\in I_n^2}\wt{X}_{i,j} - \Sigma\r\|, \\
\lambda_2 \geq \frac{140\l\|\Sigma\r\|}{n(n-1)} \sqrt{\rk(\Sigma)} + \frac{4}{\sqrt{n(n-1)}} \max_{(i,j)\in I_n^2} \l\| \wt{X}_{i,j} - \Sigma \r\| \Bigg\},
\end{multline*}
the following inequality holds:
\begin{equation*}
\l\| \wh S_\lambda - \Sigma \r\|_\F^2 \leq \inf_{ S: \rank(S)\leq \frac{c_2 n^2\lambda_2^2}{\lambda_1^2} } \Big\{ (1+\delta) \l\| S-\Sigma \r\|_\F^2 
+ c(\delta) \big( \lambda_1^2 \rank(S) 
+ \lambda_2^2{|{J}|^2} \big) \Big\},
\end{equation*}
where $c_2$ is an absolute constant and $c(\delta)$ is a constant depending only on $\delta$.
\end{theorem}
%==
The proof of Theorem \ref{thm:lasso u-stat} is presented in section \ref{proof:lasso-ustat} of the supplementary material.
\begin{remark}\label{remark for thm:u-stat}
The bound in Theorem \ref{thm:lasso u-stat} contains two terms:
\begin{enumerate}
\item
The first term, %$\l\| S-\Sigma \r\|_\F^2 + \l( 2(\sqrt 2 + 1)^2 +49 \r)\lambda_1^2 \rank(S)$, 
$(1+\delta)\l\| S-\Sigma \r\|_\F^2 +  c(\delta)\lambda_1^2 \rank(S)$,
does not depend on the number of outliers. When there are no outliers, i.e. $|J|=0$, the bound will only contain this part. In such a scenario \citet{lounici2014high} proved that the theoretically optimal bound is 
\begin{equation*}
 \l\| \wh S_\lambda - \Sigma \r\|_\F^2 \leq \inf_{S} \bigg \{ \l\| \Sigma - S\r\|_F^2 
 + C \l\| \Sigma\r\|^2 \frac{(\rk(\Sigma)+t) }{n} \rank(S)   \bigg \}
\end{equation*}
with probability at least $1-e^{-t}$. 
By making the smallest choice of $\lambda_1$ as specified in \eqref{lambda_1_choice}, one sees that the first term of our bound coincides with the theoretically optimal bound.
\item The second term, %$4\l(28(4/3+ \sqrt 2)\r)^2 \lambda_2^2{|{J}|^2}$,
$c(\delta)\lambda_2^2 |J|^2$, 
controls the worst possible effect due to the presence of outliers. When more conditions on the outliers are imposed (for example, independence), this bound can be improved. Moreover, \citet{diakonikolas2017being} proved that when $Z$ is Gaussian with zero mean, there exists an estimator $\wh{\Sigma}$ achieving theoretically optimal bound $\l\|\wh{\Sigma}-\Sigma\r\|_F \leq \mathcal{O}(\eps)\l\|\Sigma\r\|$, which is independent of the dimension $d$. 
In our case,  by making the smallest choice of $\lambda_2$ as specified in \eqref{lambda_2_choice}, we can show that the error bound scales like $\allowbreak \mathcal{O} \Big( \big(\log(n) +  \rk(\Sigma) \big) \eps \Big)  \l\|\Sigma\r\|$. %, where $C>0$ is an absolute constant.
The additional factor $\big( \log(n) + \rk(\Sigma) \big)$ shows that our bound is sub-optimal in general. However, if $\rk(\Sigma)$ is small, 
%in the sense that $\rk(\Sigma)\leq C_0$ for some constant $C_0$, our bound becomes $C \eps \log(1/\eps)\l\|\Sigma\r\|$, which 
our bound is essentially optimal up to a logarithmic factor.
\end{enumerate}
\end{remark}
%==
Note that in Theorem \ref{thm:lasso u-stat} the regularization parameters $\lambda_1,\lambda_2$ should be chosen sufficiently large such that the event $\m E$ happens with high probability. Under the assumption that $Z_j,j=1,\ldots,n$ are independent, identically distributed L-sub-Gaussian vectors, we can prove the following result which gives an explicit lower bound on the choice of $\lambda_1$.
%==bound for the mean difference==%
\begin{theorem}\label{thm:mean_bound}
Assume that Z is L-sub-Gaussian with mean $\mu$ and covariance matrix $\Sigma$. Let $Z_1,\ldots,Z_n$ be independent copies of $Z$, and define $\wt{Z}_{i,j} := {\l( Z_i - Z_j \r)}/{\sqrt{2}}$ for all $(i,j)\in I_n^2$. Then $\wt{Z}_{i,j}, (i,j)\in I_n^2$ are mean zero L-sub-Gaussian random vectors with the same covariance matrix $\Sigma$. Moreover, for any $t\geq 1$, there exists $c(L)>0$ depending only on L such that
\begin{equation*}
\l\| \frac{1}{n(n-1)}\sum_{i\neq j}\widetilde{Z}_{i,j}\widetilde{Z}_{i,j}^T - \Sigma \r\| 
\leq c(L)\l\|\Sigma\r\| \bigg( \sqrt{\frac{\rk(\Sigma)+t}{n}} + \frac{\rk(\Sigma) + t}{n}  \bigg)
\end{equation*}
with probability at least $1-2e^{-t}$.
\end{theorem}
Theorem \ref{thm:mean_bound} along with the definition of event $\m E$ indicates that it suffices to choose $\lambda_1$ satisfying
\begin{equation}\label{lambda_1_choice}
%\lambda_1 &\geq C\l\|\Sigma\r\| \l[ \sqrt{\frac{\rk(\Sigma)+t}{n}}   + \frac{\rk(\Sigma) + t}{n}  + \sqrt{\frac{\rk(\Sigma)}{n(n-1)}} \r] \nonumber\\
%& = C'\l\|\Sigma\r\| \l( \sqrt{\frac{\rk(\Sigma)+t}{n}}   + \frac{\rk(\Sigma) + \sqrt{\rk(\Sigma)} + t}{n} \r)
\lambda_1 \geq c(L) \l\|\Sigma\r\| \sqrt{\frac{\rk(\Sigma)+t}{n}},
\end{equation}
given that $n \geq \rk(\Sigma)+t$. The next theorem provides a lower bound for the choice of $\lambda_2$:
%==bound for the max term==%
\begin{theorem}
\label{thm:max covariance bound}
Assume that $Z$ is L-sub-Gaussian with mean zero and $Z_1,\ldots,Z_n$ are samples of $Z$ (not necessarily independent). There exists $c(L)>0$ depending only on L, such that for any $t\geq1$,
\begin{equation*}
\max_{j=1,\ldots,n} \|Z_jZ_j^T - \Sigma\| \leq c(L)\l\|\Sigma\r\|\l( \rk(\Sigma) + \log(n) + t\r)
\end{equation*}
with probability at least $1-e^{-t}$.
\end{theorem}
Note that Theorem \ref{thm:max covariance bound} does not require independence of samples, so it can be applied to the mean zero, L-sub-Gaussian vectors $\wt{Z}_{i,j}, (i,j)\in I_n^2$ to deduce that
\begin{equation*}
\max_{i\neq j}\l\|\wt{Z}_{i,j}\wt{Z}_{i,j}^T-\Sigma\r\| \leq c(L) \l\|\Sigma\r\| \l[ \rk(\Sigma) +\log(n(n-1)) + t \r]
\end{equation*}
with probability at least $1-e^{-t}$. Combining this bound with the definition of event $\m E$, we conclude that it suffices to choose $\lambda_2$ satisfying
\begin{equation}\label{lambda_2_choice}
%\lambda_2 &\geq C\l\|\Sigma\r\| \l[  \frac{\rk(\Sigma) +\log\l(n(n-1)\r) + t}{\sqrt{n(n-1)}} +  \frac{\sqrt{\rk(\Sigma)}}{n(n-1)} \r] \nonumber \\
%&= C'\l\|\Sigma\r\| \l(  \frac{\rk(\Sigma) +\log\l(n\r) + t}{n}+  \frac{\sqrt{\rk(\Sigma)}}{n^2} \r)
\lambda_2 \geq c(L) \l\|\Sigma\r\| \frac{ \l( \rk(\Sigma)+\log(n)+t \r)}{n}.
\end{equation}
By choosing the smallest possible $\lambda_1,\lambda_2$ as indicated in \eqref{lambda_1_choice}\eqref{lambda_2_choice}, we deduce the following corollary:
\begin{corollary}\label{cor:3-1}
%Let $\delta \in (0,\frac{6}{5}]$, $c_1 \leq \frac{1}{5980}$, $c_2\leq\frac{\delta^2}{576}$ and assume that $n\geq 2$, $|J| \leq c_2n$, then we have:
Let $\delta>0$ be an absolute constant. Assume that $n \geq \rk(\Sigma)+\log(n)$ and $|J| \leq c_1(\delta) n$, where $c_1(\delta)$ is a constant depending only on $\delta$. Then we have that
%Assume that $n\geq 2$, $|J|\leq\frac{n}{12800}$ and $c_1<1$, then for $\delta \in (0,\frac{6}{5}]$, we have: %that for any $t\geq 1$,
\begin{multline}\label{cor:3-1 eq1}
%\l\| \wh S_\lambda - \Sigma \r\|_F^2 \leq \inf_{S: \rank(S)\leq \frac{n}{14000} } \Bigg[ 1.6\l\| S-\Sigma \r\|_\F^2 \\
% +100C(L) \l\|\Sigma\r\|^2 {\frac{\rk(\Sigma)+t}{n}} {\rank(S)}
% + 38025 C(L) \l\|\Sigma\r\|^2 \frac{\l(\rk(\Sigma)+\log(n)+t\r)^2}{n^2} {|{J}|^2} \Bigg]
\l\| \wh S_\lambda - \Sigma \r\|_F^2 \leq \inf_{ \substack {S: \rank(S) \leq %\frac{n^2-}{14000}
c_2' n \big( \rk(\Sigma) + \log(n) \big) }} \Bigg\{ (1+{\delta}) \l\| S-\Sigma \r\|_\F^2 \\
 + c(L,\delta) \l\|\Sigma\r\|^2  \Big[ {\frac{\rk(\Sigma)+\log(n)}{n}} {\rank(S)}
 +  \frac{\l(\rk(\Sigma)+\log(n)\r)^2}{n^2} {|{J}|^2} \Big] \Bigg\}
\end{multline}
with probability at least $1-{3}/{n}$, where $c_2'$ is an absolute constant and $c(L,\delta)$ is a constant depending only on $L$ and $\delta$.
\end{corollary}
Note that the last term in \eqref{cor:3-1 eq1} can be equivalently written in terms of $\eps$, the proportion of outliers, as
\begin{equation}\label{cor:3-1 eq2}
c(L,\delta) \l\|\Sigma\r\|^2 \l(\rk(\Sigma)+\log(n)\r)^2 \, \eps^2.
\end{equation}

\endgroup

%=======Sec: application to heavy-tailed data=========
\section{The case of heavy-tailed data}
\label{sec:heavy-tailed}
%#################################
In this section, we consider the application as well as possible improvements of the previously discussed results to heavy-tailed data. Let $Y\in \mb R^d$ be a random vector with mean $\expect{Y} = \mu$, covariance matrix $\Sigma = \expect{(Y-\mu)(Y-\mu)^T}$, and such that $\expect{\l\|Y-\mu\r\|_2^4}<\infty$. Assume that $Y_1,\ldots,Y_n$ are i.i.d copies of $Y$, and as before our goal is to estimate $\Sigma$. Since $\mu$ is unknown and the estimation of $\mu$ is non-trivial for the heavy tailed-data, we consider the setting $\wt{Y}_{i,j} = (Y_i-Y_j)/\sqrt{2}$ and denote, for brevity, $H_{i,j}:=\wt{Y}_{i,j}\wt{Y}_{i,j}^T$.
We have previously shown that $\expect{\wt Y_{i,j}}=0$ and $\expect{H_{i,j}}=\Sigma$, so the mean estimation is no longer needed for $\wt{Y}_{i,j}$. 
Given $\lambda_1,\lambda_2>0$, we propose the following estimator for $\Sigma$:
\begin{equation}\label{heavy-tailed estimator}
\wh{S}_\lambda = \argmin_{S} \Bigg\{\frac{1}{n(n-1)} \tr \bigg[ \sum_{i\neq j} \rho_{\frac{\sqrt{n(n-1)}\lambda_2}{2}}(\wt{Y}_{i,j}\wt{Y}_{i,j}^T - S) \bigg] 
+ \frac{\lambda_1}{2}\l\|S\r\|_1\Bigg\},
\end{equation}
which is the minimizer of the penalized Huber's loss function:
\begin{equation}\label{heavy-tailed loss}
L(S) =  \frac{1}{n(n-1)}\tr \bigg[ \sum_{i\neq j} \rho_{\frac{\sqrt{n(n-1)}\lambda_2}{2}}(\wt{Y}_{i,j}\wt{Y}_{i,j}^T - S) \bigg] + \frac{\lambda_1}{2}\l\|S\r\|_1.
\end{equation}

Recall that the estimator $\wh{S}_\lambda$ in \eqref{heavy-tailed estimator} is equivalent to the double-penalized least-squares estimator in \eqref{robust covariance estimation}. The key idea to derive an error bound for $\wh{S}_\lambda$ %in the Frobenius norm 
is motivated by \citet{prasad2019unified}, which suggests that it is possible to decompose any heavy-tailed distribution as a mixture of a ``well-behaved'' and a contamination components. The decomposition bridges the gap between the heavy-tailed model and the outlier model \eqref{outlier model}, allowing us to follow an argument similar to that in Section \ref{section:main}. To be precise, %recall that $\wt{Y}_{i,j}=(Y_i-Y_j)/\sqrt{2}$ and 
we consider the decomposition
\begin{equation}\label{eq:truncation}
\wt{Y}_{i,j} = \underbrace{\wt{Y}_{i,j}\mathds{1}\l\{\l\|\wt{Y}_{i,j}\r\|_2\leq R\r\}}_{:=\wt{Z}_{i,j}} + \underbrace{\wt{Y}_{i,j}\mathds{1}\l\{\l\|\wt{Y}_{i,j} \r\|_2> R\r\}}_{:=\wt{V}_{i,j}},
\end{equation}
where $R>0$ is the truncation level that will be specified later. 
In the following two subsections, we will separately show that the estimator $\wh{S}_\lambda$ in \eqref{heavy-tailed estimator} is close to $\Sigma$ in both the operator norm and the Frobenius norm.

%#######################################
\subsection{Bounds in the operator norm}
\label{sec: Bound in operator norm}
%#######################################
In this subsection we show that $\wh{S}_\lambda$ is close to $\Sigma$ in the operator norm with high probability. 
We will be interested in the effective rank of the matrix $\expect{(H_{1,2}-\Sigma)^2}$, and denote it as
\[
r_H := \rk(\expect{(H_{1,2}-\Sigma)^2}) = \frac{\tr(\expect{(H_{1,2}-\Sigma)^2})}{\l\|\expect{(H_{1,2}-\Sigma)^2}\r\|}.
\]
\citet[Lemma 4.1]{minsker2020robust} suggest that under the bounded kurtosis assumption (to be specified later, see \eqref{cond:bounded kurtoses}), we can upper bound $r_H$ by the effective rank of $\Sigma$, namely, $r_H \leq C \cdot \rk(\Sigma)$ with some absolute constant $C$.
We first present a lemma which shows that if the tuning parameter $\lambda_1$ is too large, the estimator $\wh{S}_\lambda$ will be a zero matrix with high probability.
\begin{lemma}
\label{lemma:large tuning}
%Let $a,b\geq 2$ be real numbers such that $2c(1-\alpha)^{-1}(\frac{1}{a}+\frac{1}{b})^\alpha + \frac{1}{8}\leq\frac{1}{4}$, where $c,\alpha$ are specified by Theorem \ref{thm:operator holder}.
Assume that $t\geq0$,  $\sigma\geq \l\|\expect{(H_{1,2}-\Sigma)^2}\r\|^{\frac{1}{2}}$ and 
\[
n\geq \max \l\{64a^2 r_H t ,\frac{4b^2t^2\l\|\Sigma\r\|^2}{\sigma^2} \r\},
\]
where $a$,$b$ are sufficiently large constants. Then for any 
$\lambda_1> (\sigma/4)\sqrt{n/t} ,$
we have that $\argmin_{S}L(S)=0$ with probability at least $1-e^{-t}$.
\end{lemma}
Lemma \ref{lemma:large tuning} immediately implies that for the choice of $\lambda_1>(\sigma/4)\sqrt{n/t}$, $\l\| \wh S_\lambda - \Sigma\r\| \allowbreak = \allowbreak \l\|\Sigma\r\|$ with high probability, which is bounded by the largest singular value of $\Sigma$.
The following theorem provides a bound  for the choice of $\lambda_1\leq (\sigma/4)\sqrt{n/t}$.
\begin{theorem}
\label{thm:heavy-tailed operator bound}
%$k_0=\lfloor n/2 \rfloor$ and 
%Assume that $t\geq 1$ is such that 
%\[
%r_H\frac{2t}{n} \l( 1+2c(1-\alpha)^{-1} \r) \leq \frac{1}{520}
%\]
%where $c>0,\alpha\in(0,1)$ are specified by Theorem \ref{thm:operator holder}.
Assume that $t\geq 1$ is such that $r_H{t} \leq c_3 n$
for some sufficiently small constant $c_3$, $\sigma\geq \l\|\expect{(H_{1,2}-\Sigma)^2}\r\|^{\frac{1}{2}}$, 
and $n\geq \max \big\{64a^2  r_H t,{4b^2t^2\l\|\Sigma\r\|^2}/{\sigma^2} \big\}$ for some sufficiently large constants $a$, $b$. Then for $\lambda_1\leq (\sigma/4)\sqrt{n/t}$ and %$\theta_2 = \theta_\sigma =\frac{1}{\sigma}\sqrt{\frac{2t}{k_0}}$, 
$\lambda_2 \geq \sigma / \sqrt{(n-1)t}$, we have that
\[
\l\|\wh{S}_\lambda-\Sigma\r\|\leq \frac{20}{39}\lambda_1 + \frac{80}{39} \sigma \sqrt{\frac{t}{n}} + \frac{40}{39} \lambda_2 t
\]
with probability at least $1-\l({8}r_H/3 + 1\r)e^{-t}$.
\end{theorem}
The proofs of Lemma \ref{lemma:large tuning} and Theorem \ref{thm:heavy-tailed operator bound} are presented in section \ref{sec:proof or heavy-tailed operator bound} of the supplementary material.
\begin{remark}
The bound in Theorem \ref{thm:heavy-tailed operator bound} is close to that in \citet[Corrolary 4.1] {minsker2020robust}, with an additional term ${20}\lambda_1 / 39$. This term comes from the penalization ${\lambda_1}\l\|S\r\|_1 / 2$, showing that by ``shrinking'' our estimator to a low rank matrix, we introduce a bias term bounded by a multiple of the tuning parameter $\lambda_1$. 
\end{remark}
\begin{remark}\label{remark:bounded kurtosis}
According to \citet[Lemma 4.1] {minsker2020robust}, the ``matrix variance'' parameter $\sigma^2$ appearing in the statement of Theorem \ref{thm:heavy-tailed operator bound} can be bounded by % $\rk(\Sigma)$, the effective rank of the covariance matrix $\Sigma$, 
$\l\| \Sigma \r\| \tr(\Sigma) = \rk(\Sigma) \l\| \Sigma \r\|^2$ under the bounded kurtosis assumption. More precisely, if we assume that the kurtoses of the linear forms $\dotp{Y}{v}$ are uniformly bounded by $K$, meaning that 
\begin{equation}\label{cond:bounded kurtoses}
\sup_{v:\l\|v\r\|_2=1}\frac{\expect{\dotp{Y-\expect{Y}}{v}}^4}{\Big[\expect{\dotp{Y-\expect{Y}}{v}}^2\Big]^2} \leq K
\end{equation}
for any $v\in\mb{R}^d$. 
Then we have that
\[
\l\|\expect{(H_{1,2}-\Sigma)^2}\r\| \leq K \rk(\Sigma)\l\|\Sigma\r\|^2
\]
and $\sigma$ can be chosen as $C\sqrt{K \rk(\Sigma)}\l\|\Sigma\r\|$ with some absolute constant $C$.
Moreover, in this case the assumptions on $n$ and $t$ in Lemma \ref{lemma:large tuning} and Theorem \ref{thm:heavy-tailed operator bound} can be reduced to a single assumption that $r_H {t} \leq c_3' n$ for some sufficiently small constant $c_3'$.
We will formally state condition \eqref{cond:bounded kurtoses} in the next subsection and derive additional results based on it.
\end{remark}

%#######################################
\subsection{Bounds in the Frobenius norm}
\label{sec: Bound in Frobenius norm}
%#######################################
In this subsection we show that $\wh S_\lambda$ is close to the covariance matrix of $Y$ in the Frobenius norm with high probability, under a slightly stronger assumption on the fourth moment of $Y$.
\begin{definition}
A random vector $Y\in\mb{R}^d$ is said to satisfy an $L_4-L_2$ norm equivalence with constant $K$ (also referred to as the bounded kurtosis assumption), if there exists a constant $K\geq 1$ such that
\begin{equation}\label{L4-L2 norm equivalence}
\l(\expect{\dotp{Y-\mu}{v}^4}\r)^{\frac{1}{4}} \leq K\l(\expect{\dotp{Y-\mu}{v}^2}\r)^{\frac{1}{2}}
\end{equation}
for any $v\in\mb R^{d}$, where $\mu =\expect{Y}$.
\end{definition}
As previously discussed in Remark \ref{remark:bounded kurtosis}, condition \eqref{L4-L2 norm equivalence} allows us to connect the matrix variance parameter $\sigma^2$ with $\rk(\Sigma_Y)$, the effective rank of the covariance matrix $\Sigma_Y$. We will assume that $Y$ satisfies \eqref{L4-L2 norm equivalence} with a constant $K$ throughout this subsection.

%The key idea to derive a bound in the Frobenius norm is motivated by \citet{prasad2019unified}, which suggests that it is possible to decompose any heavy-tailed distribution as a mixture of a well-behaved and a contamination distribution. The decomposition bridges the gap between the heavy-tailed model and the outlier model \eqref{outlier model}, allowing us to follow similar argument to that in Section \ref{section:main}. To be precise, we set $\wt{Y}_{i,j}=(Y_i-Y_j)/\sqrt{2}$ and consider the decomposition:
%where $R>0$ is the truncation level that will be specified later. 
Recall the decomposition 
\begin{equation}\label{eq:truncation_2}
\wt{Y}_{i,j} = \underbrace{\wt{Y}_{i,j}\mathds{1}\l\{\l\|\wt{Y}_{i,j}\r\|_2\leq R\r\}}_{:=\wt{Z}_{i,j}} + \underbrace{\wt{Y}_{i,j}\mathds{1}\l\{\l\|\wt{Y}_{i,j} \r\|_2> R\r\}}_{:=\wt{V}_{i,j}},
\end{equation}
where $R>0$ is the truncation level that will be specified later. Denote $\Sigma_Y:= \expect{\wt{Y}_{1,2}\wt{Y}_{1,2}^T}$, $\Sigma_Z := \expect{\wt{Z}_{1,2}\wt{Z}_{1,2}^T} $ and recall that our goal is to estimate $\Sigma_Y$.  Note that $\l\|\wt{Z}_{i,j}\r\|_2\leq R$ almost surely, so equation \eqref{eq:truncation_2} represents $\wt{Y}_{i,j}$ as a sum of a bounded vector $\wt{Z}_{i,j}$ and a ``contamination'' component $\wt{V}_{i,j}$, which is similar to \eqref{outlier model}. On the other hand, we note that the truncation level $R$ should be chosen to be neither too large (to get a better truncated distribtuion) nor too small (to reduce the bias introduced by the truncation). \citet{mendelson2020robust}  suggest that a reasonable choice is given as follows: 
\begin{equation}\label{choice of truncation level}
R = \l( \frac{\tr(\Sigma_Y)\l\|\Sigma_Y\r\|n}{\log\big(\rk(\Sigma_Y) \big)+\log(n)}\r)^{\frac14}.
\end{equation} 
Denote %$N=|I_n^2| = n(n-1)$ to be the total number of $\wt{Y}_{i,j}$'s, 
$\wt{J} = \Big\{(i,j)\in I_n^2: \l\|\wt{Y}_{i,j}\r\|_2 > R \Big\}$ to be the set of indices corresponding to the nonzero outliers (i.e. $\wt{V}_{i,j} \neq 0$), and $\eps := {|\wt{J}|}/{\big( n(n-1) \big)}$ to be the proportion of outliers.
Under the above setup, we can derive the following lemma which provides an upper bound on $\eps$ with high probability:
\begin{lemma}\label{lemma:upper bound of outliers}
Assume that $Y$ satisfies the $L_4-L_2$ norm equivalence with constant K, and $R$ is chosen as in \eqref{choice of truncation level}
. Then
 \begin{equation}\label{heavy tailed-bound of outliers}
 \eps \leq c(K) {\frac{\rk(\Sigma_Y) \l[ \log\big( \rk(\Sigma_Y) \big) + \log(n) \r]}{n}}
 \end{equation}
with probability at least $1-{1}/{n}$, where $c(K)$ is a constant only depending on $K$.
\end{lemma}
The proof of Lemma \ref{lemma:upper bound of outliers} is presented in section \ref{section:proof of upper bound of outliers} of the supplementary material. It is worth noting that the proportion of ``outliers'' (in a sense of the definition above) in the heavy-tailed model can be pretty small when the sample size $n$ is large. %This coincides with the fact that $\wt{Z}_{i,j} = 0$ if and only if  $\wt{V}_{i,j} \neq 0$, namely, there are some cancellations among outliers. 
Consequently, we can derive the following bound:
%===
\begin{theorem}\label{thm:lasso u-stat improved}
Given $A \geq 1$, assume that $Y\in\mb{R}^d$ is a random vector with mean $\expect{Y}=\mu$, covariance matrix $\Sigma_Y = \expect{(Y-\mu)(Y-\mu)^T}$, and satisfying an $L_4-L_2$ norm equivalence with constant $K$. Let $Y_1,\ldots,Y_n$ be i.i.d samples of $Y$, and let $\wt{Z}_{i,j}$ be defined as in \eqref{eq:truncation_2}. 
Assume that  
%\[
%r_H\frac{\log(n)}{n} \leq c_3
%\]
%for some sufficiently small constant $c_3$, 
$n\geq c_4(K)\rk(\Sigma_Y) \big( \log( \rk(\Sigma_Y) ) + \log(n) \big)$ for some constant $c_4(K)$ depending only on $K$, and %$\rank(\Sigma_Y) \leq c_2(K) \cdot {n} \cdot \frac{\log(n)}{\log\big(n \cdot \rk(\Sigma_Y) \big)}$ 
$\rank(\Sigma_Y) \leq c_2(K) {n}$ for some constant $c_2(K)$ depending only on $K$. Then for $\lambda_1 = c(K) \l\|\Sigma_Y\r\| {\Big[ \rk(\Sigma_Y) \big( \log( \rk(\Sigma_Y)) + \log(n) \big) \Big] ^{1/2}}{n^{-1/2}}$ and $\lambda_2 = c(K) \l\|\Sigma_Y\r\| \big( \rk(\Sigma_Y)\log(n) \big)^{1/2} (An)^{-1/2}$, we have that 
%$\lambda_1 = c(L)\l\|\Sigma_Y\r\|\sqrt{\frac{\big(\log(\rk(\Sigma_Y))+\log(n)\big)\rk(\Sigma_Y)}{n}}$ and $\lambda_2 = C\frac{\sigma\log(n)}{\sqrt{n}}$, 
%for any $A\geq1$,
\begin{multline*}
\l\|\wh{S}_\lambda- \Sigma_Y\r\|_F^2 
\leq  c(K)\l\|\Sigma_Y\r\|^2 \Bigg[ \frac{\rk(\Sigma_Y) \Big( \log\big( \rk(\Sigma_Y)\big) + \log(n) \Big) }{n} \rank(\Sigma_Y) \\
+ \frac{A \cdot \rk(\Sigma_Y)^2 \log(n)^3}{n} \Bigg]
\end{multline*}
with probability at least $1- {({8} r_H / 3 +1)}{n^{-A}} - {4}{n^{-1}}$, where $c(K)$ is a constant depending only on $K$.
\end{theorem}
%==
The proof of Theorem \ref{thm:lasso u-stat improved} is given in section \ref{sec:proof of the improved frobenius bound} of the supplementary material.
\begin{remark}
Let us compare the result in Theorem \ref{thm:lasso u-stat improved} to the bound of Corollary \ref{cor:3-1}.
\begin{enumerate}
	\item The first part of the bound,
	\[
	c(K) \l\|\Sigma_Y\r\|^2 \frac{\rk(\Sigma_Y) \Big( \log\big( \rk(\Sigma_Y)\big) + \log(n) \Big) }{n}\rank(\Sigma_Y) ,
	\]
	has the same order as in Corollary \ref{cor:3-1} (up to a logarithmic factor), under the assumption that $\Sigma_Y$ has low rank. This part of the bound is theoretically optimal according to Remark \ref{remark for thm:u-stat}. 
	\item The second part of the bound,
	\begin{equation}\label{improved bound for the outlier}
	c(K)\l\|\Sigma_Y\r\|^2 \frac{\rk(\Sigma_Y)^2  \log(n)^3 }{n},
	\end{equation}
	controls the error introduced by the outliers. 
	It is much smaller than the corresponding quantity in Corollary \ref{cor:3-1} when $\eps$, the proportion of the outliers, is only 
	assumed to be a constant. The improvement is mainly due to the special structure of the heavy-tailed data, namely, %the ``outlier'' term $\wt{V}_{i,j}$ is nonzero if and only if the ``well-behaved'' term $\wt{Z}_{i,j}$ equals zero. 
	the ``outliers'' $\wt{V}_{i,j}$ are mutually independent as long as the subscripts do not overlap, and hence there are many cancellations among them. Without this special structure, one can 
	only apply Theorem \ref{thm:lasso u-stat} directly and derive a sub-optimal bound of order
	\[
		c(K)\l\|\Sigma_Y\r\|^2 \frac{\rk(\Sigma_Y)^3 {\log(n)^3} }{n}.
	\]

%	It can be interpreted in terms of $\eps$, the proportion of outliers as 
%	\[
%	c(L)\l\|\Sigma_Y\r\|^2  \log\big(n\cdot\rk(\Sigma_Y)\big)\cdot\eps^2
%	\]
%	which is in the order of $\m{O}\big(\rk(\Sigma_Y) \eps^2  \big)$. It is clear that this order is smaller
%	than that in \eqref{cor:3-1 eq2}. Moreover, $\eps$ in Corollary \ref{cor:3-1} is assumed to be a constant, while for heavy-		tailed data, it is in the order of $\m{O}\big( \sqrt{\frac{\rk(\Sigma_Y)}{n}} \big)$. Therefore, the bound \eqref{improved bound for the 	outlier} is much smaller comparing to that in Corollary \ref{cor:3-1}. Note that the improvement is mainly due to the special 		structure of the heavy-tailed data, without which one can only apply Theorem \ref{thm:lasso u-stat} and derive a sub-optimal 		bound
%	\[
%	c(L)\l\|\Sigma_Y\r\|^2 \frac{\rk(\Sigma_Y)^3 \log\big( n\cdot\rk(\Sigma_Y) \big)^2  }{n}.
%	\]	
\end{enumerate}
\end{remark}

%#################################
%=======Sec: Numerical experiments=========
\section{Numerical experiments}
\label{sec:numerical simulation}
%#################################
In this section we present analysis and algorithms for our numerical experiments. Recall that our loss function is
\begin{multline}\label{original_simulation_loss}
\wt{L}(S,\mathbf{U_{I_n^2}}) = \frac{1}{n(n-1)}\sum_{i\neq j}\l\|\wt{Y}_{i,j}\wt{Y}_{i,j}^T-S-\sqrt{n(n-1)}U_{i,j}\r\|_F^2 \\
+ \lambda_1\l\|S\r\|_1 + \lambda_2 \sum_{i\neq j}\l\|U_{i,j}\r\|_1.
\end{multline}
We are aiming to find $(\widehat{S}_\lambda,\boldsymbol{\widehat{U}_{I_n^2}})$, the minimizer of \eqref{original_simulation_loss}, numerically. Since we are only interested in $\wh{S}_\lambda$, equation \eqref{penalized Huber} suggests that it suffices to minimize the following function:
%_____Loss Function______%
\begin{equation}\label{simulation_loss}
L(S) :=  \frac{1}{n(n-1)}\tr\sum_{i\neq j} \rho_{\frac{\sqrt{n(n-1)}\lambda_2}{2}}(\wt{Y}_{i,j}\wt{Y}_{i,j}^T - S) + \frac{\lambda_1}{2}\l\|S\r\|_1,
\end{equation}
%_______________________
%where \[
%\rho_\lambda(u) := \l\{
%			    \begin{array}{ll}
%			    	\frac{u^2}{2},\quad\l|u\r|\leq\lambda\\
%			    	\lambda\l|u\r|-\frac{\lambda^2}{2},\quad\l|u\r|>\lambda
%			    \end{array}
%			    \r. \quad \forall u\in \mb R, \lambda \in \mb R^+
%\]
where $\rho_\lambda(\cdot)$ is the Huber's loss function defined in \eqref{huber's function}. %We will introduce the proximal gradient descent method in the next subsection, which helps us find the minimizer of \eqref{simulation_loss} numerically.

%#######################################
\subsection{Numerical algorithm}
\label{sec:numerical algorithm}
%#######################################

In this section, we will state our algorithm for minimizing $L(S)$. We start with an introduction to the proximal gradient method (see for example, \citet{combettes2005signal}). Suppose that we want to minimize the function $f(x)=g(x)+h(x)$, where
\begin{itemize}
\item $g$ is convex, differentiable
\item $h$ is convex, not necessarily differentiable
\end{itemize}
We define the proximal mapping and the proximal gradient descent method as follows:
%=====Def: proximal mapping and PGD=======
\begin{definition}
The proximal mapping of a convex function $h$ at the point $x$ is defined as:
\begin{equation*}
\prox_h(x) = \argmin_{u}\l(h(u) + \frac{1}{2}\l\|u-x\r\|_2^2\r).
\end{equation*}
\end{definition}
\begin{definition}[Proximal gradient descent (PGD) method]
The proximal gradient descent method for solving the problem $\argmin_xf(x) = \argmin_xg(x)+h(x)$ starts from an initial point $x^{(0)}$, and updates as:
\begin{equation*}
x^{(k)} = \prox_{\alpha_kh}\l(x^{(k-1)} - \alpha_k\nabla g(x^{(k-1)}) \r),
\end{equation*}
where $\alpha_k>0$ is the step size. 
\end{definition}
%==========================================
We have the following convergence result.
\begin{theorem}\label{pgd convergence}
Assume that $\nabla g$ is Lipschitz continuous with constant $L>0$:
\[
\l\|\nabla g(x) - \nabla g(y) \r\| \leq L \l\|x-y\r\|
\]
and the optimal value $f^*$ is finite and achieved at the point $x^*$. Then the proximal gradient algorithm with constant step size $\alpha_k=\alpha \leq L$ will yield an $\mathcal{O}({1}/{k})$ convergence rate, i.e.
\[
f(x^{(k)}) - f^* \leq \frac{C}{k},\quad \forall k \in \{1,2,\ldots\}.
\]
\end{theorem}
%==
Theorem \ref{pgd convergence} is well known (see for example, \citet[Chapter 10]{beck2017first}), but a detailed proof in our case is given in section \ref{sec:convergence analysis of pgd} of the supplementary material for the convenience of the reader.  
Moreover, when $g(x) = \frac{1}{n}\sum_{i=1}^{n}g_i(x)$, where $g_1,\ldots,g_n$ are convex functions and $\nabla g_1,\ldots,\nabla g_n$ are Lipschitz continuous with a common constant $L>0$, the update step of PGD will require the evaluation of $n$ gradients, which is expensive for large $n$ values. A natural improvement is to consider the stochastic proximal gradient descent method (SPGD), where at each iteration $k=1,2,\ldots$, we pick an index $i_k$ randomly from $\{1,2,\ldots,n\}$, and take the following update:
\[
x^{(k)} = \prox_{\alpha_kh}\l(x^{(k-1)} - \alpha_k\nabla g_{i_k}(x^{(k-1)}) \r).
\]
The advantage of SPGD over PGD is that the computational cost of SPGD per iteration is $1/n$ that of the PGD. On the other hand, since the random sampling in SPGD introduces additional variance, we need to choose a diminishing step size $\alpha_k = \mathcal{O}({1}/{k})$. As a result, the SPGD only converges at a sub-linear rate (see \citet{nitanda2014stochastic}). To this end, we will consider the ``mini-batch" PGD, which has been previously explored and widely used in large-scale learning problems (see, e.g., \citet{shalev2011pegasos, gimpel2010distributed, dekel2012optimal, khirirat2017mini}). This method picks a small batch of indices rather than one at each iteration to calculate the gradient, and in such a way we are able balance the computational cost of PGD and the additional variance of SPGD. The algorithm is summarized in Algorithm \ref{alg: SPGD}.
\begingroup
\allowdisplaybreaks
\begin{algorithm}[H]
	\caption{Stochastic proximal gradient descent (SPGD)}
	\label{alg: SPGD}
\begin{flushleft}
	\quad\textbf{Input:} number of iterations $T$, step size $\eta_t$, batch size $b$, tuning parameters $\lambda_1$ and $\lambda_2$, initial estimation $S^0$, sample size $n$, dimension $d$.
\end{flushleft}
	\begin{algorithmic}[1]
		\FOR {$t=1,2,\ldots,T$}
			\STATE (1) Randomly pick $i_t,j_t\in \{1,2,\ldots,n\}$ without replacement.
			\STATE (2) Compute $G_t = -\nabla g_{i,j}(S^t) = -\rho'_{\frac{\sqrt{n(n-1)}\lambda_2}{2}}(\wt{Y}_{i,j}\wt{Y}_{i,j}^T-S^t)$.
			\STATE (3) If $b>1$, then repeat (1)(2) for b times and save the average gradient in $G_t$.
			\STATE (4) (\textbf{gradient update}) 
					$T^{t+1} = S^{t} - G_t.$
			\STATE (5) (\textbf{proximal update}) 
				\[
				S^{t+1} = \argmin_S\Big\{\frac{1}{2}\l\|S-T^{t+1}\r\|_F^2 +\frac{\lambda_1}{2}\l\|S\r\|_1 \Big\} = \gamma_{\frac{\lambda_1}{2}}(T^{t+1}),
				\]
				where $\gamma_\lambda(u) = \sign(u)(|u|-\lambda)_+$.
			
		\ENDFOR
	\end{algorithmic}
\begin{flushleft}
	\quad\textbf{Output:} $S^{T+1}$
\end{flushleft}
\end{algorithm}
\endgroup

%#######################################
\subsection{Rank-one update of the spectral decomposition}
\label{sec:rank-one update}
%#######################################

Note that at each iteration of our algorithm, we need to compute the spectral decomposition of $(\wt{Y}_{i,j}\wt{Y}_{i,j}^T-S^t)$, which is computationally expensive. However, since $\wt{Y}_{i,j}\wt{Y}_{i,j}^T$ is a rank-one matrix and $S^t$ was already saved in the spectral decomposition form after previous iteration, the problem of computing the spectral decomposition of ($\wt{Y}_{i,j}\wt{Y}_{i,j}^T-S^t$) can be viewed as a rank-one update of the spectral decomposition, which has been extensively studied (see for example, \citet{bunch1978rank}, and \citet{stange2008efficient}). In this subsection we will show how to use this idea to improve our algorithm.

Consider $\wt{B}=B + \rho bb^T$, where the spectral decomposition $B=QDQ^T$ is known, $\rho\in\mb{R}$ and $b\in\mb R^d$. Our target is to compute the spectral decomposition of $\wt{B}$. Note that 
\begin{equation}
\wt{B} = B +\rho bb^T = Q(D+\rho zz^T)Q^T,
\end{equation}
where $b=Qz$, so it suffices to compute the spectral decomposition of $D+\rho zz^T$. We denote $z=(\zeta_1,\ldots,\zeta_d)^T$, and without loss of generality, we can assume that $\l\|z\r\|_2=1$. The following theorem is fundamental for our algorithm.
%==
\begin{theorem}{(\citet[Theorem 1]{bunch1978rank})}
\label{rank-1 update}
Let $C=D+\rho zz^T$, where $D$ is diagonal, $\l\|z\r\|_2=1$. Let $d_1\leq d_2\leq\ldots\leq d_d$ be the eigenvalues of D, and let $\wt{d}_1\leq\wt{d}_2\leq\ldots\wt{d}_d$ be the eigenvalues of C. Then $\wt{d}_i=d_i+\rho\mu_i$, $1\leq i \leq d$ where $\sum_{i=1}^n \mu_i =1$ and $0\leq\mu_i\leq1$. Moreover, $d_1\leq\wt{d}_1\leq d_2\leq\wt{d}_2\leq\ldots\leq d_d \leq \wt{d}_d$ if $\rho>0$ and $\wt d_1\leq {d}_1\leq \wt d_2\leq d_2\leq\ldots\leq \wt d_d \leq {d}_d$ if $\rho<0$. Finally, if $d_i$'s are distinct and all the elements of z are nonzero, then the eigenvalues of $C$ strictly separate those of $D$. 
\end{theorem}
%==
There are several cases where we can deflate the problem (i.e. reduce the size of the problem):
\begin{enumerate}
	\item If $\zeta_i=0$ for some $i$, then $\wt{d}_i = d_i$ and the corresponding eigenvector remains unchanged. This is because $(D+\rho zz^T)e_i = d_i e_i$ as $\zeta_i=0$.
	\item If $|\zeta_i|=1$ for some $i$, then $\wt{d}_i = d_i +\rho$ and the corresponding eigenvector remains unchanged. Moreover, in this case $\zeta_j=0$ for all $j\neq i$, so $\wt{d}_j = d_j$ and their eigenvectors are the same, so the problem is done.
	\item If $d_i$ has a multiplicity $r\geq 2$, we can reduce the size of the problem via the following steps:
	\begin{enumerate}
		\item Let $Q_1 = [q_{i_1},\ldots,q_{i_r}]\in\mb R^{d\times r}$, where $\{q_{i_1},\ldots,q_{i_r}\}$ are the eigenvectors corresponding to $d_i$. Also, set $z_1 = Q_1^Tz$, i.e. $z_1$ contains rows corresponding to $d_i$.
		\item Construct an Householder transformation $H\in \mb R^{r\times r}$ such that $Hz_1 = -\l\|z_1\r\|_2 e_1$, and define $\bar{Q}_1=Q_1H^T$.
		\item Replace $q_{i_1},\ldots,q_{i_r}$ by the columns of $\bar{Q}_1$, and $z_1$ by $\bar{Q}_1^Tz_1=-\l\|z_1\r\|_2e_1$. This introduces $(r-1)$ more zero entries to $z$ and possibly an entry with absolute value equals one. An application of (1)(2) gives us $(r-1)$ (or $r$) more eigen-pairs of $D+\rho zz^T$.
	\end{enumerate}
\end{enumerate}
After the deflation step, it remains to work with a $k\times k$ problem ($k\leq d$), in which the eigenvalues $d_i$ are distinct and $\zeta_i\neq 0$ for all $i$. We will compute the eigenvalues and eigenvectors separately. 

First, \citet{golub1973some} showed that the eigenvalues of $C = D+\rho zz^T$ are the zeros of $\omega(\lambda)$, where
\[
\omega(\lambda) = 1+\rho\sum_{j=1}^{k}\frac{\zeta_j^2}{d_j-\lambda}.
\]
Alternatively, since $\wt{d}_1 < \ldots < \wt{d}_k$ and $\wt{d}_i = d_i +\rho\mu_i$, for each $i=1,\ldots,k$ we can compute $\mu_i$ by solving $\omega_i(\mu_i) = 0$, where
\begin{equation}\label{compute_eigenvalue}
\omega_i(\mu) = 1 + \sum_{j=1}^{k}\frac{\zeta_j^2}{\delta_j-\mu}.
\end{equation}
and $\delta_j = {(d_j-d_i)}/{\rho}$. \citet{bunch1978rank} proved that we can solve $\omega_i(\mu)=0$ with a numerical method that converges quadratically. Details of the numerical method are presented in section \ref{sec:eigenvalues} of the supplementary material.

Second, after computing the eigenvalues $\wt{d}_1,\ldots,\wt{d}_k$, we can calculate the corresponding eigenvectors of $C = D+\rho zz^T$ by solving $C\wt{q}_i = \wt{d}_i\wt{q}_i$, $i=1,\ldots,k$. Theorem 5 in \citet{bunch1978rank} shows that $\wt{q}_i$ can be computed via
\begin{equation}\label{compute_eigenvector}
\wt{q}_i = \frac{D_i^{-1}z}{\l\|D_i^{-1}z\r\|_2},
\end{equation}
where $D_i := D - \wt{d}_i I$. Finally, once we obtained the spectral decomposition of $D+\rho zz^T = \bar{Q}\wt{D}\bar{Q}^T$, we can easily get the decomposition of $B +\rho bb^T= (Q\bar{Q})\wt{D}(Q\bar{Q})^T$. Note that computing $k$ eigenvectors via \eqref{compute_eigenvector} costs $\mathcal{O}(k^3)$ and the matrix multiplication $Q\bar{Q}$ in the last step costs $\mathcal{O}(d^3)$, so the overall complexity of the algorithm is still $\mathcal{O}(d^3)$. This can be further improved by exploiting the special structure of  $\bar{Q}$, which is given by the product (see for example, \citet{stange2008efficient} and \citet{gandhi2017updating}):
\begin{equation}\label{eq:special structure of Q bar}
\bar{Q} = 
\underbrace{
\begin{bmatrix}
\zeta_1 &  \\
   & \ddots  \\
   & 		&\zeta_d
\end{bmatrix}
} _{:= A}
\underbrace{
\begin{bmatrix}
\frac{1}{d_1-\wt{d}_1} & \cdots & \frac{1}{d_1-\wt{d}_d} \\
\vdots & & \vdots \\
\frac{1}{d_d - \wt{d}_1} & \cdots & \frac{1}{d_d - \wt{d}_d}
\end{bmatrix}
} _{:= C}
\begin{bmatrix}
\l\|\bar{c}_{\cdot 1}\r\|_2 &  \\
   & \ddots  \\
   & 		&\l\|\bar{c}_{\cdot d}\r\|_2
\end{bmatrix}^{-1}
\end{equation}
where $\bar{C} := AC = [\bar{c}_{\cdot 1},\ldots,\bar{c}_{\cdot d}]$, $\bar{c}_{\cdot i}$ represents the $i^{th}$ column of $\bar{C}$, and $\l\|\bar{c}_{\cdot i}\r\|_2$ is the Euclidean norm of $\bar{c}_{\cdot i}$. 
Using \eqref{eq:special structure of Q bar}, we can evaluate the matrix multiplication $Q\bar{Q}$ through the following steps:
\begin{enumerate}
\item Compute 
$QA := U = [u_{\cdot 1}, \ldots, u_{\cdot d}]$,
where $u_{\cdot i} = \zeta_i q_{\cdot i}$ and $q_{\cdot i}$ is the $i^{th}$ column of $Q$. This step is straightforward and requires $\mathcal{O}(d^2)$ computational time. 
\item Let $u_{i \cdot}$ be the $i^{th}$ row of $U$. Define
\[
\wt{U} = UC = 
\begin{bmatrix}
u_{1\cdot} C \\
\vdots \\
u_{d\cdot} C
\end{bmatrix},
\]
which requires to evaluate the product of a vector $u_{i\cdot}$ and a Cauchy matrix $C$ $d$ times. The problem of multiplying a Cauchy matrix with a vector is called Trummer's problem, and \citet{gandhi2017updating} provide an algorithm which efficiently computes such matrix-vector product in $\mathcal{O}(d\log^2d)$ time. Consequently, the complexity of this step is  $\mathcal{O}(d^2\log^2d)$. 
\item Compute the matrix product
\[
\wt{U}
\begin{bmatrix}
\l\|\bar{c}_{\cdot 1}\r\|_2 &  \\
   & \ddots  \\
   & 		&\l\|\bar{c}_{\cdot d}\r\|_2
\end{bmatrix}^{-1}.
\]
This step is again straightforward and can be done in $\mathcal{O}(d^2)$ time.
\end{enumerate}
The overall complexity for the computation of $Q\bar{Q}$ is now reduced to $\mathcal{O}(d^2\log^2d)$, which is much smaller than $\mathcal{O}(d^3)$ when $d$ is large.
We summarized our rank-one update algorithm in Algorithm \ref{alg: rank-one update}. Numerical experiments were performed and the results are presented in section \ref{sec:numerical results with graphs} of the supplementary material.
\begingroup
\allowdisplaybreaks
\begin{algorithm}\label{rank-one update}[H]
	\caption{Rank-one update of the spectral decomposition of $B + \rho bb^T$}
	\label{alg: rank-one update}
\begin{flushleft}
	\quad\textbf{Input:} orthogonal matrix Q and vector d such that $B=Q\text{diag}(d)Q^T$, constant $\rho$, vector $b$
\end{flushleft}
	\begin{algorithmic}[1]
		\STATE Set $Qb = z$. If $\l\|z\r\|_2 \neq 1$, then further set $\rho = \rho\l\|z\r\|_2^2$ and $z = z/\l\|z\r\|_2$.
		\STATE Handle deflation cases, and record indices that have not done as a vector $d_{sub}$.
		\STATE Compute eigenvalues of the $d_{sub}\times d_{sub}$ sub-problem by solving \eqref{compute_eigenvalue} numerically.
		\STATE Compute eigenvectors of the $d_{sub}\times d_{sub}$ sub-problem with \eqref{compute_eigenvector}.
		\STATE Combine the resulting eigenvalues in $\wt{d}$ and eigenvectors in $\bar{Q}$.
		\STATE Compute $\wt{Q}=Q\bar{Q}$.
	\end{algorithmic}
\begin{flushleft}
	\quad\textbf{Output:} orthogonal matrix $\wt{Q}$ and vector $\wt{d}$.
\end{flushleft}
\end{algorithm}
\endgroup

\section*{Acknowledgements}
\it{Authors acknowledge support by the National Science Foundation grants CIF-1908905 and DMS CAREER-2045068.}
\par
%%%%%%%%%%%%%%%%%%%%%%%%%%%%%%%%%%%%%%%%%%%%%%%%%%%%%%%%%%%%%%%%%%%%%%%%%%%%%%%%%%%%%%%%%%

{
%\iffalse
\bibhang=1.7pc
\bibsep=2pt
\fontsize{9}{14pt plus.8pt minus .6pt}\selectfont
\renewcommand\bibname{\large \bf References}
%\begin{thebibliography}{11}
\expandafter\ifx\csname
natexlab\endcsname\relax\def\natexlab#1{#1}\fi
\expandafter\ifx\csname url\endcsname\relax
  \def\url#1{\texttt{#1}}\fi
\expandafter\ifx\csname urlprefix\endcsname\relax\def\urlprefix{URL}\fi
%\fi

%\newpage
\bibliographystyle{chicago} % Style BST file (imsart-number.bst or imsart-nameyear.bst)
\bibliography{bibliography}       % Bibliography file (usually '*.bib')

%%%%%%%%%%%%%%%%%%%%%%%%%%%%%%%%%%%%%%%%%%%%%%%%%%%%%%%%%%%%%%%%%%%%%%%%%%%%%%%%%%%%%%%%%%%%%%%%%%%%%%%%%%%%%%%%%%%%%%%%%%%%
\vskip .65cm
\noindent
Department of Mathematics, University of Southern California, Los Angeles, CA, 90089, U.S.A.
\vskip 2pt
\noindent
E-mail: minsker@usc.edu
\vskip 2pt

\noindent
Department of Mathematics, University of Southern California, Los Angeles, CA, 90089, U.S.A.
\vskip 2pt
\noindent
E-mail: langwang@usc.edu

}

\begin{appendix}
\begin{center}
\large \textbf{Supplementary material}
\end{center}

%#######################################
\section{Numerical results}
\label{sec:numerical results with graphs}
%#######################################
In this section we present some numerical results with different parameter settings. First, note that if we start with $S^0=0_{d\times d}$, we can easily compute the gradient in the first step of proximal gradient descent via
\begin{equation}\label{full gradient}
G = \rho'_{\frac{\sqrt{N}\lambda_2}{2}}(\wt{Y}_{i,j}\wt{Y}_{i,j}^T) = \frac{ \rho'_{\frac{\sqrt{N}\lambda_2}{2}}(\l\|\wt{Y}_{i,j}\r\|_2^2)}{\l\|\wt{Y}_{i,j}\r\|_2^2} \wt{Y}_{i,j}\wt{Y}_{i,j}^T.
\end{equation}
Here, no explicit spectral decomposition was required. \citet[Remark 4.1]{minsker2020robust} provide details supporting the claim that the full gradient update at the first step helps to improve the initial guess of the estimator. Therefore, we will start with $S^0=0_{d\times d}$, run one step of PGD with the full data set, and use the output as the initial estimate of the solution. 

Now consider the following parameter settings: $d=200$, $n=100$, $|J|=3$, $\mu=(0,\ldots,0)^T$, $\Sigma=\text{diag}(10,1 \allowbreak , \allowbreak 0.1,\ldots,0.1)$. The samples are generated as follows: generate $n=100$ independent samples $Z_j$ from the Gaussian distribution $\m N(\mu,\Sigma)$, and then replace $|J|$ of them (randomly chosen) with $Z_j + V_j$, where $V_j,j\in J$ are ``outliers'' to be specified later. The final results after replacement, denoted as $Y_j$'s, are the samples we observe and that will be used as inputs for the SPGD algorithm. Next, we calculate $\wt{Y}_{i,j}={(Y_i-Y_j)}/{\sqrt{2}}, i\neq j$ and perform our algorithm with $K=500$ steps and the diminishing step size $\alpha_k = 1/k$. The initial value $S^0$ is determined by a one-step full gradient update %, i.e. start with $0_{d\times d}$ and update one step using 
\eqref{full gradient}. 
To analyze the performance of estimators, we define 
\[
\text{RelErr}(S, \text{Frob}) := \frac{\l\| S- \Sigma \r\|_F}{\l\| \Sigma \r\|}
\]
to be the relative error of the estimator $S$ in the Frobenius norm, and 
\[
 \text{RelErr}(S,\text{op}) := \frac{\l\| S- \Sigma \r\|}{\l\| \Sigma \r\|}
\]
to be the relative error of the estimator S in the operator norm, where $S$ is an arbitrary estimator.
We will compare the performance of the estimator $S^*$ produced by our algorithm with the performance of the sample covariance matrix $\wt\Sigma_s$ introduced in \eqref{sample covariance matrix}. Here are some results corresponding to different types of outliers:
%===enumerate===
\begin{enumerate}
%====item 1======
\item\textbf{Constant outliers.}
Consider the outliers $V_j = (100,\ldots,100)^T, j\in J$. %, which constantly shift $Z_j$ on all directions. 
We performed 200 repetitions of the experiment with $\lambda_1 =3$, $\lambda_2=1$, and recorded $S^*$, $\wt{\Sigma}_s$ for each run. Histograms illustrating the distributions of relative errors in the Frobenius norm are shown in Figure \ref{fig-1} and \ref{fig-2}.

\begin{figure}[H]
  \centering
  \begin{minipage}[b]{0.4\textwidth}
     \includegraphics[width=1.1\textwidth]{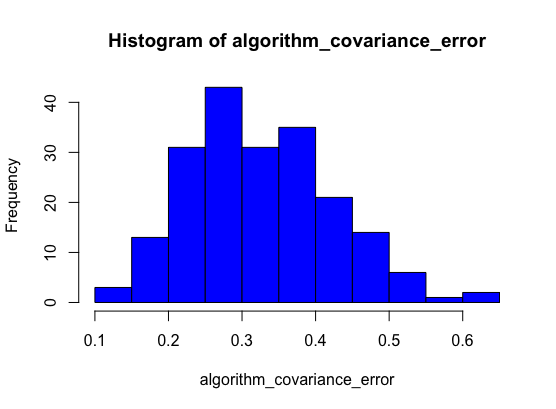}
     \caption{Distribution of RelErr($S^*$, Frob). }
     \label{fig-1}
  \end{minipage}
  \hfill
  \begin{minipage}[b]{0.4\textwidth}
    \includegraphics[width=1.1\textwidth]{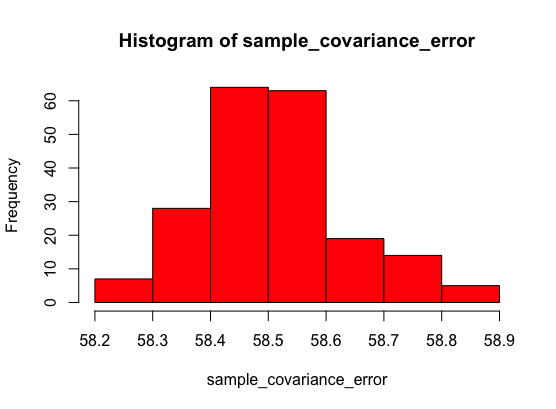}
    \caption{Distribution of RelErr($\wt{\Sigma}_s$, Frob).}
    \label{fig-2}
  \end{minipage}
\end{figure}

%\begin{figure}[H]
%\center
%\includegraphics[width=8cm]{figures/Outliers1_plot1.png}
%\caption{Distribution of the relative error of $S^*$ in the Frobenius norm }
%\label{fig-1}
%\end{figure}
%
%\begin{figure}[H]
%\center
%\includegraphics[width=8cm]{figures/Outliers1_plot2.png}
%\caption{Distribution of the relative error of $\wt \Sigma_s$ in the Frobenius norm}
%\label{fig-2}
%\end{figure}

The histograms show that  $\wt{\Sigma}_s$ always produces a relative error in the Frobenius norm around $58.5$,  while $S^*$ always produces a relative error in the Frobenius norm around $0.3$. The average and maximum (over 200 repetitions) relative errors of $S^*$ were $0.3246$ and $0.6144$ respectively, with the standard deviation of $0.0979$. The corresponding values for $\wt{\Sigma}_s$ were $58.5124$, $58.8629$ and $0.1227$. It is clear that in the considered scenario, estimator $S^*$ performed noticeably better than the sample covariance $\wt{\Sigma}_s$.
\newline In the meanwhile, the following histograms (Figure \ref{fig-3} and \ref{fig-4}) show that $S^*$ produces smaller relative errors in the operator norm as well. The average and maximum relative errors of $S^*$ in the operator norm were $0.2700$ and $0.5324$ respectively, with the standard deviation of $0.0983$. The corresponding values for $\wt{\Sigma}_s$ were $58.8055$, $59.1579$ and $0.1233$.
\begin{figure}[H]
  \centering
  \begin{minipage}[b]{0.4\textwidth}
    \includegraphics[width=1.1\textwidth]{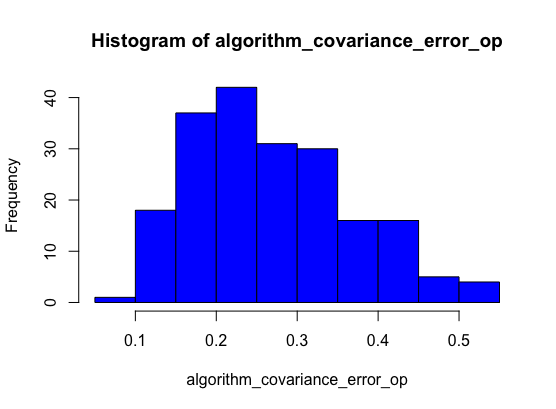}
    \caption{Distribution of RelErr($S^*$, op).}
    \label{fig-3}
  \end{minipage}
  \hfill
  \begin{minipage}[b]{0.4\textwidth}
    \includegraphics[width=1.1\textwidth]{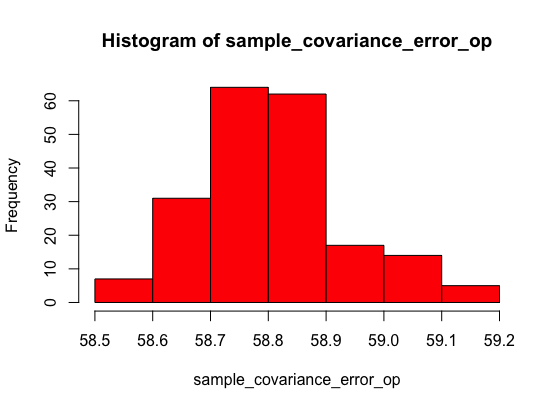}
    \caption{Distribution of RelErr($\wt{\Sigma}_s$, op).}
    \label{fig-4}
  \end{minipage}
\end{figure}
%======item 2=======
\item\textbf{Spherical Gaussian outliers.} Consider the case that the outliers $V_j$ are drawn independently from a spherical Gaussian distribution $\m N(\mu_V,\Sigma_V)$, where $\mu_V=(0,\ldots,0)^T$, $\Sigma_V=\text{diag}(100,\ldots,100)$. In this case, the outliers affect $Z_j$ uniformly in all directions. We performed 200 repetitions of the experiment with $\lambda_1=3$, $\lambda_2=1$, and recorded $S^*$, $\wt{\Sigma}_s$ for each run. Histograms illustrating the distributions of relative errors are shown in Figure \ref{fig-5} and \ref{fig-6}:

\begin{figure}[H]
  \centering
  \begin{minipage}[b]{0.4\textwidth}
    \includegraphics[width=1.1\textwidth]{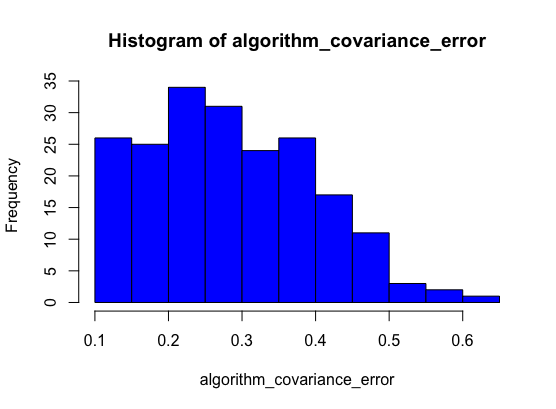}
    \caption{Distribution of RelErr($S^*$, Frob).}
    \label{fig-5}
  \end{minipage}
  \hfill
  \begin{minipage}[b]{0.4\textwidth}
    \includegraphics[width=1.1\textwidth]{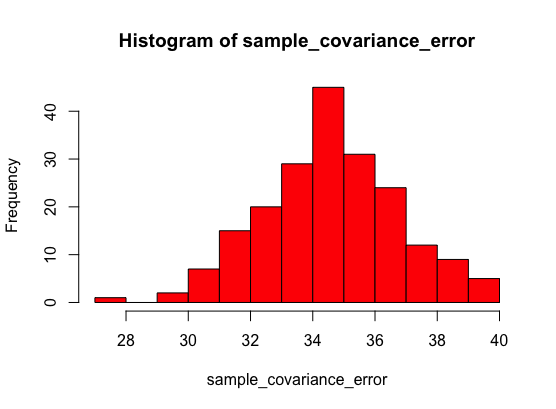}
    \caption{Distribution of RelErr($\wt{\Sigma}_s$, Frob).}
    \label{fig-6}
  \end{minipage}
\end{figure}
%
%
%\begin{figure}[H]
%\center
%\includegraphics[width=8cm]{figures/Outliers2_plot1.png}
%\caption{Distribution of the relative error of $S^*$ in the Frobenius norm }
%\label{fig-5}
%\end{figure}
%
%\begin{figure}[H]
%\center
%\includegraphics[width=8cm]{figures/Outliers2_plot2.png}
%\caption{Distribution of the relative error of $\wt \Sigma_s$ in the Frobenius norm}
%\label{fig-6}
%\end{figure}
The histograms show that  $\wt{\Sigma}_s$ always produces a relative error in the Frobenius norm around $34$,  while $S^*$ always produces a relative error in the Frobenius norm around $0.3$. The average and maximum (over 200 repetitions) relative errors of $S^*$ were $0.2842$ and $0.6346$ respectively, with the standard deviation of $0.1108$. The corresponding values for $\wt{\Sigma}_s$ were $34.5880$, $39.6758$ and $2.1501$. It is clear that in the considered scenario, estimator $S^*$ performed noticeably better than the sample covariance $\wt{\Sigma}_s$.
\newline In the meanwhile, the following histograms (Figure \ref{fig-7} and \ref{fig-8}) show that $S^*$ produces smaller relative errors in the operator norm as well. The average and maximum relative errors of $S^*$ in the operator norm were $0.2676$ and $0.6290$ respectively, with the standard deviation of $0.1148$. The corresponding values for $\wt{\Sigma}_s$ were $22.9255$, $28.2328$ and $1.8791$.
\begin{figure}[H]
  \centering
  \begin{minipage}[b]{0.4\textwidth}
    \includegraphics[width=1.1\textwidth]{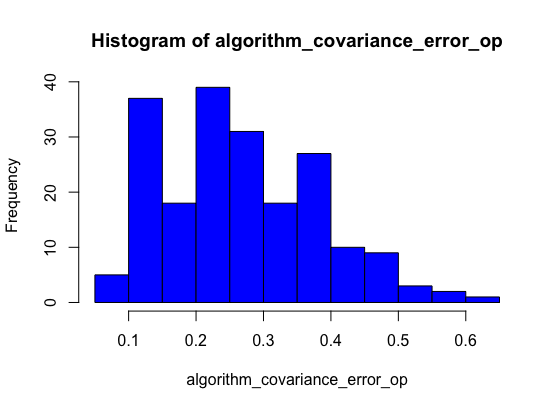}
    \caption{Distribution of RelErr($S^*$, op).}
    \label{fig-7}
  \end{minipage}
  \hfill
  \begin{minipage}[b]{0.4\textwidth}
    \includegraphics[width=1.1\textwidth]{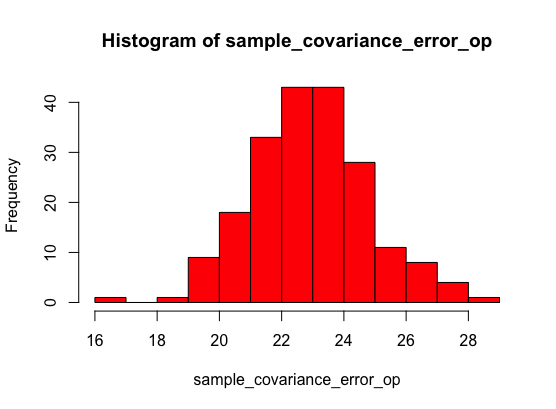}
    \caption{Distribution of RelErr($\wt{\Sigma}_s$, op).}
    \label{fig-8}
  \end{minipage}
\end{figure}
%=====item 3======
\item\textbf{Outliers that ``erase'' some observations.}
Consider the case that the outliers are given as $V_j = \beta Z_j$ for $ j\in J, \beta\in\mb R$. In this case, the outliers erase (when $\beta=-1$), amplify (when $\beta>0$) or negatively amplify (when $\beta< -1$) some sample points $Z_j$. We performed 200 repetitions of the experiment with $\lambda_1=\lambda_2=0.4$, $\beta = -50$ and recorded $S^*$, $\wt{\Sigma}_s$ for each run. Histograms illustrating the distributions of relative errors are shown in Figure \ref{fig-9} and \ref{fig-10}:

\begin{figure}[H]
  \centering
  \begin{minipage}[b]{0.4\textwidth}
    \includegraphics[width=1.1\textwidth]{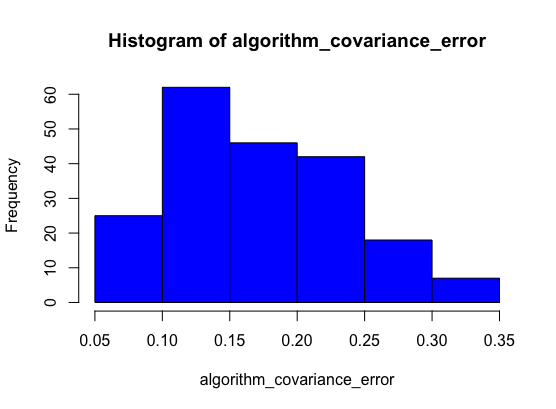}
    \caption{Distribution of RelErr($S^*$, Frob).}
    \label{fig-9}
  \end{minipage}
  \hfill
  \begin{minipage}[b]{0.4\textwidth}
    \includegraphics[width=1.1\textwidth]{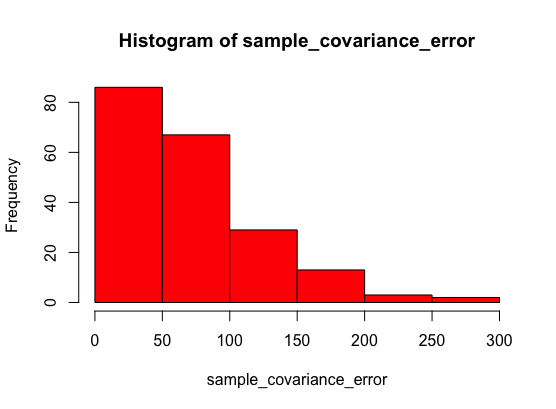}
    \caption{Distribution of RelErr($\wt{\Sigma}_s$, Frob).}
    \label{fig-10}
  \end{minipage}
\end{figure}

%\begin{figure}[H]
%\center
%\includegraphics[width=8cm]{figures/Outliers3_plot1.png}
%\caption{Distribution of the relative error of $S^*$ in the Frobenius norm }
%\label{fig-9}
%\end{figure}
%
%\begin{figure}[H]
%\center
%\includegraphics[width=8cm]{figures/Outliers3_plot2.png}
%\caption{Distribution of the relative error of $\wt \Sigma_s$ in the Frobenius norm}
%\label{fig-10}
%\end{figure}
The histograms show that $S^*$ always produces a relative error in the Frobenius norm around $0.15$, while $\wt{\Sigma}_s$ produces a relative error in the Frobenius norm around $50$. Note that unlike previous examples, the performance of $\wt{\Sigma}_s$ is unstable in the current settings, with relative errors raising to $300$ occasionally. The average and maximum (over 200 repetitions) relative errors of $S^*$ were $0.1716$ and $0.3383$ respectively, with the standard deviation of $0.0644$. The corresponding values for $\wt{\Sigma}_s$ were $72.9263$, $263.0925$ and $50.3815$. It is clear that in the considered scenario, estimator $S^*$ performed noticeably better than the sample covariance $\wt{\Sigma}_s$.
\newline In the meanwhile, the following histograms (Figure \ref{fig-11} and \ref{fig-12}) show that $S^*$ produces smaller and more stable relative errors in the operator norm as well. The average and maximum relative errors of $S^*$ in the operator norm were $0.1652$ and $0.3393$ respectively, with the standard deviation of $0.0680$. The corresponding values for $\wt{\Sigma}_s$ were $72.7100$, $263.4040$ and $50.8983$.

\begin{figure}[H]
  \centering
  \begin{minipage}[b]{0.4\textwidth}
    \includegraphics[width=1.1\textwidth]{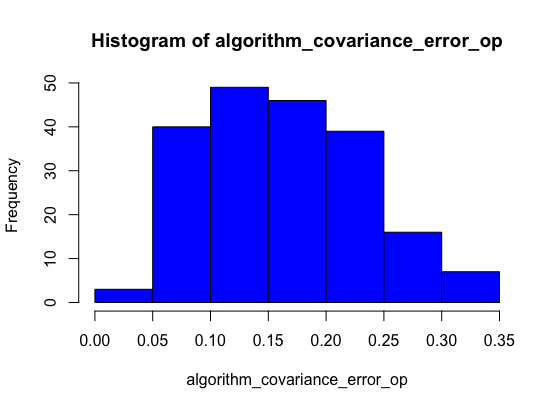}
    \caption{Distribution of RelErr($S^*$, op).}
    \label{fig-11}
  \end{minipage}
  \hfill
  \begin{minipage}[b]{0.4\textwidth}
    \includegraphics[width=1.1\textwidth]{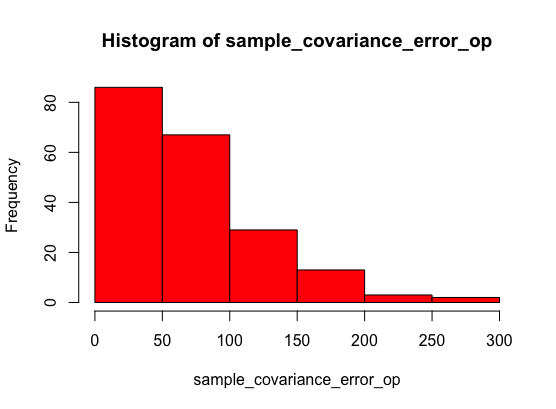}
    \caption{Distribution of RelErr($\wt{\Sigma}_s$, op).}
    \label{fig-12}
  \end{minipage}
\end{figure}

%======item 4=======
\item\textbf{Outliers in a particular direction.}
Finally we consider the case that the outliers are all orthogonal (or parallel) to the subspace spanned by the first $M$ principal components of  $Z$, where Z is an $n\times d$ matrix with $Z_j^T$ on each row. We performed 200 repetitions of the experiment with $\lambda_1=3$, $\lambda_2=1$, $M=1$ (orthogonal case) and recorded $S^*$, $\wt{\Sigma}_s$ for each run. Histograms illustrating the distributions of relative errors are shown in Figure \ref{fig-13} and \ref{fig-14}:

\begin{figure}[H]
  \centering
  \begin{minipage}[b]{0.4\textwidth}
    \includegraphics[width=1.1\textwidth]{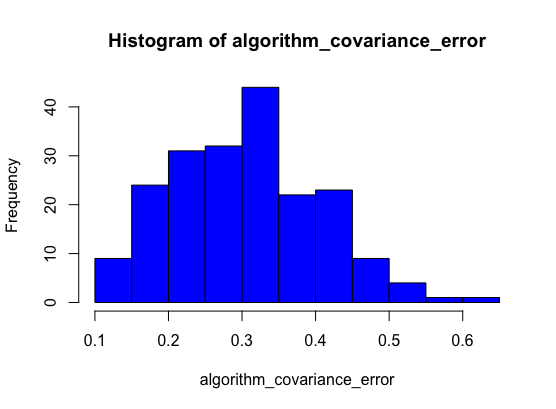}
    \caption{Distribution of RelErr($S^*$, Frob).}
    \label{fig-13}
  \end{minipage}
  \hfill
  \begin{minipage}[b]{0.4\textwidth}
    \includegraphics[width=1.1\textwidth]{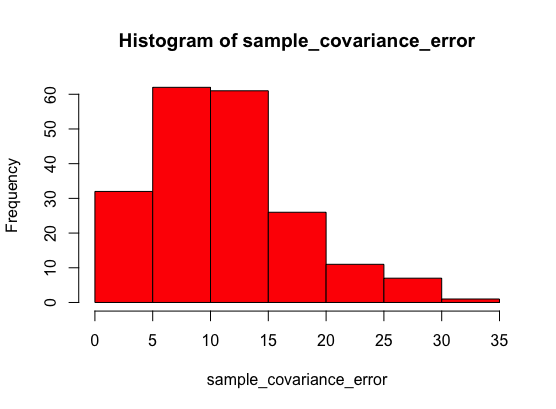}
    \caption{Distribution of RelErr($\wt{\Sigma}_s$, Frob).}
    \label{fig-14}
  \end{minipage}
\end{figure}

%
%\begin{figure}[H]
%\center
%\includegraphics[width=8cm]{figures/Outliers4_plot1.png}
%\caption{Distribution of the relative error of $S^*$ in the Frobenius norm }
%\label{fig-13}
%\end{figure}
%\begin{figure}[H]
%\center
%\includegraphics[width=8cm]{figures/Outliers4_plot2.png}
%\caption{Distribution of the relative error of $\wt \Sigma_s$ in the Frobenius norm}
%\label{fig-14}
%\end{figure}

The histograms show that  $\wt{\Sigma}_s$ mainly produces a relative error in the Frobenius norm around $10$,  while $S^*$ always produces a relative error in the Frobenius norm around $0.3$. The average and maximum (over 200 repetitions) relative errors of $S^*$ were $0.3038$ and $0.6296$ respectively, with the standard deviation of $0.1021$. The corresponding values for $\wt{\Sigma}_s$ were $11.1444$, $30.5072$ and $6.2206$. Note that the smallest error produced by $\wt{\Sigma}_s$ was $0.2930$, which is comparable to the error produced by $S^*$. However, the histograms show that the small error produced by $\wt{\Sigma}_s$ only occurs occasionally, while $S^*$ was producing small errors consistently. Therefore, in the considered scenario, we can still conclude that estimator $S^*$ performed better than the sample covariance $\wt{\Sigma}_s$.
\newline In the meanwhile, the following histograms (Figure \ref{fig-15} and \ref{fig-16}) show that $S^*$ produces smaller relative errors in the operator norm as well. The average and maximum relative errors of $S^*$ in the operator norm were $0.2573$ and $0.5888$ respectively, with the standard deviation of $0.1038$. The corresponding values for $\wt{\Sigma}_s$ were $11.1985$, $30.6591$ and $6.2528$.
\begin{figure}[H]
  \centering
  \begin{minipage}[b]{0.4\textwidth}
    \includegraphics[width=1.1\textwidth]{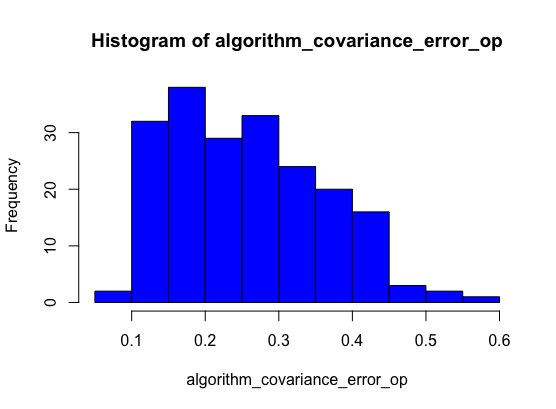}
    \caption{Distribution of RelErr($S^*$, op).}
    \label{fig-15}
  \end{minipage}
  \hfill
  \begin{minipage}[b]{0.4\textwidth}
    \includegraphics[width=1.1\textwidth]{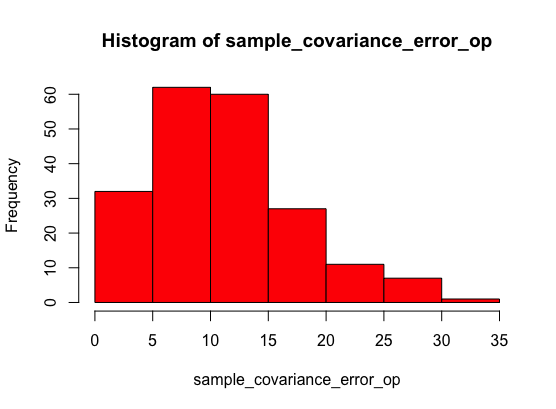}
    \caption{Distribution of RelErr($\wt{\Sigma}_s$, op).}
    \label{fig-16}
  \end{minipage}
\end{figure}
\end{enumerate}
%====end enumerate=======

\section{Proofs ommited from the main exposition}%Section \ref{sec:estimation unkown mean}}
\label{sec:proofs for section 2.3}
In this section, we present the proofs that were omitted from the main exposition in Section \ref{sec:estimation unkown mean}. We start by introducing some technical tools that will be useful for the proof.

%#####################
\subsection{Technical tools}
%#####################
First, we have the following useful trace duality inequalities for the Frobenius inner product.
\begin{proposition}
\label{lem:trace duality}
For any $A,B\in \mb R^{d_1 \times d_2}$,
\begin{eqnarray}
|\dotp{A}{B}| &\leq& \l\|A\r\|_F \l\|B\r\|_F \nonumber ,\\
|\dotp{A}{B}| &\leq& \l\|A\r\|_1 \l\| B \r\|  \nonumber.
\end{eqnarray}
\end{proposition}
Next, let $L$ be a linear subspace of $\mb R^{d}$ and  $L^{\perp}$ be its orthogonal complement, namely, $L^{\perp} = \big\{  v \in \mb{R}^{d}: \dotp{v}{u} = 0, \forall \, u\in L \big\}$. In what follows, $P_L$ will stand for the orthogonal projection onto $L$, meaning that $P_L\in \mb R^{d\times d}$ is such that $P_L^2=P_L=P_L^T$ and $\im(P_L)\subseteq L$, where $\im(P_L)$ represents the image of $P_L$. 
Given the spectral decomposition of a real symmetric matrix, we have the following proposition:
\begin{proposition}\label{prop:projection}
Let $S\in S^{d}(\mb R)$ be a real symmetric matrix with spectral decomposition $S=\sum_{j=1}^{d} \lambda_j u_j u_j^T$, where the eigenvalues satisfy $|\lambda_1|\geq \cdots \geq |\lambda_d| \geq 0$. Denote $L=\im(S)=\lspan\l\{u_j: \lambda_j\neq 0\r\}$. Then $P_L = \sum_{j:\lambda_j\neq0}u_ju_j^T$ and $P_{L^\perp} = \sum_{j:\lambda_j=0}u_ju_j^T$.
\end{proposition}
Moreover, we will be interested in a linear operator $\m P_L: \mb R^{d\times d} \mapsto \mb R^{d\times d}$ defined as 
\begin{equation}\label{def:matrix_proj_operator}
\m P_L(A) := A - P_{L^\perp}AP_{L^\perp}.
\end{equation}
The following lemma provides some results on $\m P_L(\cdot)$ that will be useful in our proof.
\begin{lemma}\label{lemma:matrix_algebra}
Let $L$ be a linear subspace of $\mb R^{d}$ and $A\in S^{d}(\mb R)$ be an arbitrary real symmetric matrix, then
\begin{enumerate}
\item $\l\|\m P_L(A)\r\| \leq \l\| A \r\|$.
\item $\rank(\m P_L(A)) \leq 2 \dim(L)$.
\end{enumerate}
\end{lemma}
\begin{proof}
The proof of this lemma follows straightforward from the definition of $\m P_L$ and hence omitted here.
\end{proof}
%=====
The following proposition characterizes the subdifferential of the convex function $A \mapsto \l\|A\r\|_1$.
%========Prop: subdifferential==============
\begin{proposition}[\citet{watson1992characterization}]
Let $A\in S^{d}(\mb R)$ be a symmetric matrix and $A = \sum_{j=1}^{\rank(A)}\sigma_ju_jv_j^T$  be the singular value decomposition. Denote $L=\lspan\{u_1,\ldots,u_r\}$, then
\begin{equation*}
\partial\l\|A\r\|_1 = \l\{ \sum_{j:\sigma_j>0}u_jv_j^T + P_{L^\perp} W P_{L^\perp}: \l\|W\r\|\leq1  \r\},
\end{equation*}
where $P_{L^\perp}$ represents the orthogonal projection onto $L^\perp$.
\end{proposition}
%=====
Next, we state some results for the best rank-k approximation. We say that the function $\vertiii{\cdot}: \mb{R}^{d_1 \times d_2} \mapsto \mb{R}$ is a matrix norm if for any scalar $\alpha \in \mb{R}$ and any matrices $A,B \in \mb{R}^{d_1 \times d_2}$, the following properties are satisfied:
\begin{itemize}
	\item $\vertiii{\alpha A} = | \alpha | \vertiii{A}$;
	\item $\vertiii{A + B} \leq \vertiii{A} + \vertiii{B}$; 
	\item $\vertiii{A} \geq 0 $, and $\vertiii{A} = 0$ if and only if $A = 0_{d_1 \times d_2}$.
\end{itemize}
The operator norm $\l\| \cdot \r\|$, the Frobenius norm $\l\| \cdot \r\|_F$ and the nuclear norm$\l\| \cdot \r\|_1$ introduced in Definition \ref{def:matrix norms} are concrete examples of matrix norms. 
Given a nonnegative definite matrix $\Sigma$, we say that $\Sigma(k)$ is the best rank-k approximation of $\Sigma$ with respect to the matrix norm $\vertiii{\cdot}$, if %in the operator norm (or the Frobenius norm) if 
%\[
%\Sigma(k) = \argmin_{S:\rank(S)\leq k}\l\|S-\Sigma\r\|
%\]
\[
\Sigma(k) = \argmin_{S:\rank(S)\leq k} \vertiii{S-\Sigma}.
\]
%(or $\Sigma(k) = \argmin_{S:\rank(S)\leq k}\l\|S-\Sigma\r\|_F$). 
The following theorem characterized the best rank-k approximation.
%==thoerem:Eckart-Young==%
\begin{theorem}[\citet{kishore2017literature}]
\label{thm:Eckart-Young}
Let $\Sigma \in \mb R^{d\times d}$ be a nonnegative definite matrix with spectral decomposition $\Sigma=\sum_{j=1}^{d}\lambda_j u_j u_j^T$, where the eigenvalues satisfy $\lambda_1\geq\cdots\geq \lambda_d\geq0$. Then the matrix $A:=\sum_{j=1}^{k}\lambda_ju_ju_j^T$ is the best rank-k approximation of $\Sigma$ in both Frobenius norm and operator norm. Consequently, we have that $$\min_{S:\rank(S)\leq k}\l\|S-\Sigma\r\| = \lambda_{k+1}$$ and $$\min_{S:\rank(S)\leq k}\l\|S-\Sigma\r\|_F = \sqrt{\sum_{j=k+1}^{d} \lambda_{j}^2}.$$
\end{theorem}
The following two corollaries will be used in our proof.
%==lemma: rank-k approximation==%
\begin{corollary}\label{rank-k approximation}
Let $\Sigma(k)$ be the best rank-k approximation of $\Sigma$ in the operator norm. Then $\l\|\Sigma(k)-\Sigma\r\|\leq \l\|\Sigma\r\| \l( ({\rk(\Sigma)}/{k}) \wedge \sqrt{{\rk(\Sigma)}/{k}} \r)$ and $\l\|\Sigma(k)-\Sigma\r\|_F\leq {\tr(\Sigma)^2}/{k}$.
\end{corollary}
\begin{proof}
Let $\lambda_j(A)$ be the j-th largest eigenvalue of a nonnegative definite matrix, then by Theorem \ref{thm:Eckart-Young},$\l\|\Sigma(k)-\Sigma\r\| \allowbreak = \lambda_{k+1}(\Sigma)$. Moreover, we have that
\begin{equation}\label{lambda_k+1}
\lambda_{k+1}(\Sigma)\leq\frac{\sum_{j=1}^{k+1}\lambda_j(\Sigma)}{k+1}\leq\frac{\tr(\Sigma)}{k+1}\leq\frac{\tr(\Sigma)}{k} = \l\|\Sigma\r\| \frac{\rk(\Sigma)}{k}.
\end{equation}
Note that $\lambda_{k+1}(\Sigma)\leq\sqrt{\l\|\Sigma\r\|}\sqrt{\lambda_{k+1}(\Sigma)}$. Combining this with the previous display, we get another inequality 
\[
\lambda_{k+1}(\Sigma)\leq\sqrt{\frac{\l\|\Sigma\r\|tr(\Sigma)}{k}} = \l\|\Sigma\r\| \sqrt{\frac{\rk(\Sigma)}{k}}.
\]
So we have that  $\l\|\Sigma(k)-\Sigma\r\|\leq \l\|\Sigma\r\| \l( ({\rk(\Sigma)}/{k}) \wedge \sqrt{{\rk(\Sigma)}/{k}} \r)$. To obtain the bound in the Frobenius norm, we note that
\begin{equation*}
\l\|\Sigma(k)-\Sigma\r\|_F^2 = \sum_{j\geq k+1}\lambda_j(\Sigma)^2 \leq (\tr(\Sigma))^2\sum_{j\geq k+1}j^{-2}\leq\frac{\tr(\Sigma)^2}{k},
\end{equation*}
where the first inequality follows from \eqref{lambda_k+1} and the second inequality follows from %$\sum_{j\geq k+1}j^{-2}=\sum_{j\geq k}\int_{j}^{j+1}\frac{1}{(j+1)^2}dt \leq \int_{k}^{\infty}\frac{1}{t^2}dt=\frac{1}{k}$. 
$\sum_{j\geq k+1}j^{-2} = \sum_{j\geq k+1} j^{-1}(j+1) ^{-1} \allowbreak = 1/(k+1) .$
\end{proof}
\begin{remark}
It is easy to see that ${\rk(\Sigma)}/{k}\leq\sqrt{{\rk(\Sigma)}/{k}}$ if and only if $\rk(\Sigma)\leq k$, so when $\Sigma$ has low effective rank, i.e. $\rk(\Sigma)\leq k$, the upper bound becomes $$\l\|\Sigma(k) - \Sigma\r\| \leq \frac{\l\|\Sigma\r\|\rk(\Sigma)}{k}.$$
\end{remark}
\begin{corollary}
Let $\Sigma(k)$ be the best rank-k approximation of $\Sigma$ in the operator norm defined in Theorem \ref{thm:Eckart-Young}, and $L(k):=\im(\Sigma(k))$. Then $\m P_{L(k)}(\Sigma) = \Sigma(k)$.
\end{corollary}
\begin{proof}
By Theorem \ref{thm:Eckart-Young}, $\Sigma(k)$ has spectral decomposition $\Sigma(k)=\sum_{j=1}^{d}\lambda_ju_ju_j^T$ with $\lambda_{k+1}=\cdots=\lambda_d=0$. Then Proposition \ref{prop:projection} implies that $P_{L(k)^\perp}=\sum_{j=k+1}^{d}u_ju_j^T$. Therefore,
\begin{multline*}
\m P_{L(k)}(\Sigma) = \Sigma - P_{L(k)^\perp}\Sigma P_{L(k)^\perp} 
=\sum_{j=1}^{d}\lambda_ju_ju_j^T - \sum_{j=k+1}^{d}\lambda_ju_ju_j^T \\
=\sum_{j=1}^{k}\lambda_ju_ju_j^T = \Sigma(k).
\end{multline*}
\end{proof}

%#########################################
\subsection{Proof of Theorem \ref{thm:lasso u-stat}}
\label{proof:lasso-ustat}
%#########################################
%=====begin proof=====%
%\begin{proof}
%allow breaking long equations among pages
\begingroup
\allowdisplaybreaks
To simplify the notations, we denote $N:=n(n-1)$. Let
\begin{multline}\label{loss function}
F(S,\mathbf{U_{I_n^2}}) =  \frac{1}{N}\sum_{i\neq j} \l\| \widetilde{Y}_{i,j} \widetilde{Y}_{i,j}^T - S - \sqrt{N}U_{i,j} \r\|^2_\F \\
+ \lambda_1 \l\|S\r\|_1 + \lambda_2 \sum_{i\neq j} \l\|U_{i,j} \r\|_1.
\end{multline}
The function $F$ is convex, and we have that
\begin{align}\label{gradient of loss function}
\partial F(S,\boldsymbol{U_{I_n^2}})=
\begin{pmatrix}
-\frac{2}{N}\sum_{i\neq j} (\widetilde{Y}_{i,j}\widetilde{Y}_{i,j}^T-S-\sqrt{N}U_{i,j}) \\
-\frac{2}{N}\sqrt{N}(\widetilde{Y}_{1,2}\widetilde{Y}_{1,2}^T-S-\sqrt{N}U_{1,2}) \\
\vdots \\
-\frac{2}{N}\sqrt{N}(\widetilde{Y}_{n,n-1}\widetilde{Y}_{n,n-1}^T-S-\sqrt{N}U_{n,n-1}) 
\end{pmatrix} 
+
\begin{pmatrix}
\lambda_1\partial\l\|S\r\|_1 \\
\lambda_2\partial\l\|U_{1,2}\r\|_1 \\
\vdots \\
\lambda_2\partial\l\|U_{n,n-1}\r\|_1
\end{pmatrix}
\end{align} 
where $\partial\l\|A\r\|_1$ represents the subdifferential of $\l\|\cdot\r\|_1$ at $A$.
Note that for any symmetric matrices $S,U_{1,2},\ldots,U_{n,n-1}$, the directional derivative of $F$ at the point $ \big( \wh{S}_\lambda, \wh{U}_{1,2}, \ldots \allowbreak,\allowbreak \wh{U}_{n,n-1} \big)$ in the direction 
$\big( S - \wh S_{\lambda},U_{1,2}- \wh U_{1,2},\ldots \allowbreak, \allowbreak U_{n,n-1} - \wh U_{n,n-1}\big)$ is nonnegative. In particular, we consider an arbitrary S and $U_{1,2}:=\widetilde{U}_{1,2}^*,\ldots,U_{n,n-1}:=\widetilde{U}_{n,n-1}^*$. By the necessary condition of the minima, there exist $\wh V \in \partial \|\wh S_\lambda\|_1, \ \wh W_{1,2}\in \partial \| \wh U_{1,2}\|_1,\ldots,\wh W_{n,n-1}\in \partial \| \wh U_{n,n-1}\|_1$ such that
\begin{multline*}
\dotp{\partial F(\wh{S},\mathbf{\wh{U}_{I_n^2}})}{(S-\wh{S}_\lambda; \mathbf{\wt{U}_{I_n^2}^{*}} -\mathbf{\wh{U}_{I_n^2}} )} =\\
-\frac{2}{N}\sum_{i\neq j} \dotp{\widetilde{Y}_{i,j} \widetilde{Y}_{i,j}^T - \wh S_\lambda -\sqrt{N}\wh U_{i,j}}{S - \wh S_\lambda + \sqrt{N}(\widetilde{U}_{i,j}^* - \wh U_{i,j})} \\
+\lambda_1 \dotp{\wh V}{S - \wh S_\lambda} + \lambda_2 \sum_{i\neq j} \dotp{\wh W_{i,j}}{\widetilde{U}_{i,j}^* - \wh U_{i,j}} \geq 0.
\end{multline*}
For any choice of $V \in \partial \l\|S\r\|_1, \ W_{1,2}\in \partial \l\|\widetilde{U}_{1,2}^*\r\|_1,\ldots \allowbreak , \allowbreak W_{n,n-1}\in \partial \l\|\widetilde{U}_{n,n-1}^*\r\|_1$, by the monotonicity of subgradients we deduce that
\begin{eqnarray*}
&&\dotp{V-\wh{V}}{S-\wh{S}_\lambda}\geq0, \\
&&\dotp{W_{i,j}-\wh{W}_{i,j}}{\widetilde{U}_{i,j}^*-\wh{U}_{i,j}}\geq0,\quad (i,j)\in I_n^2.
\end{eqnarray*}
Hence the previous display implies that
\begin{multline*}
\frac{2}{N}\sum_{i\neq j} \dotp{\widetilde{Y}_{i,j} \widetilde{Y}_{i,j}^T - \wh S_\lambda -\sqrt{N}\wh U_{i,j}}{S - \wh S_\lambda + \sqrt{N}(\widetilde{U}_{i,j}^* - \wh U_{i,j})}\\
\leq\lambda_1 \dotp{V}{S - \wh S_\lambda} + \lambda_2 \sum_{i\neq j} \dotp{W_{i,j}}{\widetilde{U}_{i,j}^* - \wh U_{i,j}},
\end{multline*}
which is equivalent to
\begin{multline}\label{ineq-1}
\frac{2}{N}\sum_{i\neq j} \dotp{\Sigma - \wh S_\lambda +\sqrt{N}(\widetilde{U}_{i,j}^*-\wh U_{i,j})}{S - \wh S_\lambda + \sqrt{N}(\widetilde{U}_{i,j}^* - \wh U_{i,j})} \\
\leq-\lambda_1 \dotp{V}{\wh S_\lambda - S} - \lambda_2 \sum_{i\neq j} \dotp{W_{i,j}}{\wh U_{i,j} - \widetilde{U}_{i,j}^*} \\
-2\dotp{\frac{1}{N}\sum_{i\neq j}\widetilde{X}_{i,j}-\Sigma}{S-\wh{S}_\lambda}-\frac{2}{\sqrt{N}}\sum_{i\neq j}\dotp{\widetilde{X}_{i,j}-\Sigma}{\widetilde{U}_{i,j}^*-\wh{U}_{i,j}},
\end{multline}
where $\wt{X}_{i,j} = \wt{Z}_{i,j} \wt{Z}_{i,j}^T$. We will bound \eqref{ineq-1} in two cases.
%====Case1 when the inner product is negative======%
\newline\textbf{Case 1:} Assume that 
\begin{equation*}
\frac{2}{N}\sum_{i\neq j} \Big\langle {\Sigma - \wh S_\lambda +\sqrt{N}(\widetilde{U}_{i,j}^*-\wh U_{i,j})}, 
{S - \wh S_\lambda + \sqrt{N}(\widetilde{U}_{i,j}^* - \wh U_{i,j})} \geq 0.
\end{equation*}
Applying the law of cosines, $2\dotp{A}{B}=\l\|A\r\|_F^2+\l\|B\r\|_F^2-\l\|A-B\r\|_F^2$, $\forall A,B\in\mb{R}^{d\times d}$, to the left hand side of \eqref{ineq-1}, we get that
\begin{multline}\label{ineq-2.1}
\frac{1}{N}\sum_{i\neq j}\l\|\Sigma-\wh{S}_\lambda+\sqrt{N}(\widetilde{U}_{i,j}^*-\wh U_{i,j})\r\|_F^2 
+ \frac{1}{N}\sum_{i\neq j}\l\|S - \wh S_\lambda + \sqrt{N}(\widetilde{U}_{i,j}^* - \wh U_{i,j})\r\|_F^2 \\
\leq \l\|\Sigma-S\r\|_F^2+2\dotp{\frac{1}{N}\sum_{i\neq j}\widetilde{X}_{i,j}-\Sigma}{\wh{S}_\lambda-S} 
+\frac{2}{\sqrt{N}}\sum_{i\neq j}\dotp{\widetilde{X}_{i,j}-\Sigma}{\wh{U}_{i,j}-\widetilde{U}_{i,j}^*} \\
- \lambda_1 \dotp{V}{\wh S_\lambda - S} - \lambda_2 \sum_{i\neq j} \dotp{W_{i,j}}{\wh U_{i,j} - \widetilde{U}_{i,j}^*}. 
\end{multline}
We will now analyze the terms on the right-hand side of equation \eqref{ineq-2.1} one by one. 
%==bound 4 parts of RHS one by one===%
%==part 1===%
First, let $S=\sum_{j=1}^{\rank(S)}\sigma_{j}(S)u_jv_j^T$ be the singular value decomposition of $S$, where $\sigma_j(S)$ is the j-th largest singular value of $S$. Then we can represent any $V\in\partial\l\|S\r\|_1$ by $V=\sum_{j=1}^{\rank(S)}u_jv_j^T+P_{L^\perp}WP_{L^\perp}$ for some $\l\|W\r\|\leq1$, where $L=\text{span}\{u_1,\ldots,u_{\rank(S)}\}$. From this representation, we have  that $\m P_L(V)=V-P_{L^\perp}VP_{L^\perp}=\sum_{j=1}^{\rank(S)}u_jv_j^T$, and
\begin{align}\label{bound-1}
-\dotp{V}{\wh S_\lambda - S}&=-\dotp{\m P_L(V)}{\wh S_\lambda - S}-\dotp{P_{L^\perp}VP_{L^\perp}}{\wh S_\lambda - S}\nonumber\\
&=-\dotp{\m P_L(V)}{\wh S_\lambda - S}-\dotp{W}{P_{L^\perp}\wh{S}_\lambda P_{L^\perp}}\nonumber\\
&\leq\l|\dotp{\m P_L(V)}{\wh S_\lambda - S}\r|-\l\|P_{L^\perp}\wh{S}_\lambda P_{L^\perp}\r\|_1\nonumber\\
&\leq\l\|\m P_L(V)\r\|_F\l\|\wh{S}_\lambda-S\r\|_F-\l\|P_{L^\perp}\wh{S}_\lambda P_{L^\perp}\r\|_1\nonumber\\
&=\sqrt{\rank(S)}\l\|\wh{S}_\lambda-S\r\|_F-\l\|P_{L^\perp}\wh{S}_\lambda P_{L^\perp}\r\|_1,
\end{align}
where we chose $W$ such that $\dotp{W}{P_{L^\perp}\wh{S}_\lambda P_{L^\perp}} \allowbreak =  \l\|P_{L^\perp}\wh{S}_\lambda P_{L^\perp}\r\|_1$.
%==part 2===%
Similarly, let $L_{i,j}$ be the image of $\wt{U}^*_{i,j} \allowbreak , \allowbreak (i,j)\in I_n^2$, then for properly chosen $W_{1,2}\in\partial\l\|\wt{U}^*_{1,2}\r\|_1 ,\ldots \allowbreak , \allowbreak W_{n,n-1}\in\partial\l\|\wt{U}^*_{n,n-1}\r\|_1$, we have that
\begin{multline}\label{bound-2}
 -\sum_{i\neq j}\dotp{W_{i,j}}{\wh U_{i,j} - \widetilde{U}_{i,j}^*} \\
\leq-\sum_{i\neq j}\l\|P_{L_{i,j}^\perp}\wh{U}_{i,j}P_{L_{i,j}^\perp}\r\|_1+\sum_{i\neq j}\l|\dotp{\m P_{L_{i,j}}(W_{i,j})}{\wh{U}_{i,j}-\wt{U}^*_{i,j}}\r| \\
\leq\sum_{i\neq j}\sqrt{\rank(\wt{U}_{i,j}^*)}\l\|\wh{U}_{i,j}-\wt{U}^*_{i,j}\r\|_F-\sum_{i\neq j}\l\|P_{L_{i,j}^\perp}\wh{U}_{i,j}P_{L_{i,j}^\perp}\r\|_1 \\
\leq\sum_{(i,j)\in\wt{J}}\sqrt{2}\l\|\wh{U}_{i,j}-\wt{U}^*_{i,j}\r\|_F-\sum_{i\neq j}\l\|P_{L_{i,j}^\perp}\wh{U}_{i,j}P_{L_{i,j}^\perp}\r\|_1, 
\end{multline}
where we used the fact that $\rank(\wt{U}^*_{i,j})\leq2$, and $\wt{U}^*_{i,j}=0$, $L_{i,j}=\{0\}$ for $(i,j)\notin\wt{J}$.
%==part 3===%
Next, we denote $\Delta:=\frac{1}{N}\sum_{i\neq j}\wt{X}_{i,j}-\Sigma$ and recall the linear operator defined in \eqref{def:matrix_proj_operator}: 
\[
\m P(A) = A - P_{L^\perp} A P_{L^\perp}.
\]
It is easy to check that $\m P_L(\Delta)=P_{L^\perp}\Delta P_{L}+P_{L}\Delta$, which implies $\rank(\m P_{L}(\Delta))\leq2\rank(S)$. Therefore, 
\begin{multline}\label{bound-3}
\dotp{\Delta}{\wh{S}_\lambda-S} 
=\dotp{\m P_L(\Delta)}{\wh{S}_\lambda-S}+\dotp{P_{L^\perp}\Delta P_{L^\perp}}{\wh{S}_\lambda-S} \\
=\dotp{\m P_L(\Delta)}{\wh{S}_\lambda-S}+\dotp{\Delta}{P_{L^\perp}(\wh{S}_\lambda-S)P_{L^\perp}} \\
\leq\l\|\m P_L(\Delta)\r\|_F\l\|\wh{S}_\lambda-S\r\|_F+\l\|\Delta\r\| \l\|P_{L^\perp}\wh{S}_\lambda P_{L^\perp}\r\|_1 \\
\leq\sqrt{\rank(\m P_L(\Delta))}\l\|\m P_L(\Delta)\r\| \l\|\wh{S}_\lambda-S\r\|_F+\l\|\Delta\r\| \l\|P_{L^\perp}\wh{S}_\lambda P_{L^\perp}\r\|_1 \\
\leq\sqrt{2\rank(S)}\l\|\Delta\r\|  \l\|\wh{S}_\lambda-S\r\|_F+\l\|\Delta\r\| \l\|P_{L^\perp}\wh{S}_\lambda P_{L^\perp}\r\|_1,
\end{multline}
where the last inequality follows from the bound $\l\|\m P_L(\Delta)\r\| \leq \l\|\Delta\r\|$.
%==part 4===%
Finally, it is easy to see that
\begin{multline}\label{bound-4}
\sum_{i\neq j}\dotp{\wt{X}_{i,j}-\Sigma}{\wh{U}_{i,j}-\wt{U}^*_{i,j}} = \sum_{i\neq j}\dotp{\m P_{L_{i,j}}(\wt{X}_{i,j}-\Sigma)}{\wh{U}_{i,j}-\wt{U}^*_{i,j}} \\
+ \sum_{i\neq j}\dotp{P_{L_{i,j}^\perp}(\wt{X}_{i,j}-\Sigma) P_{L_{i,j}^\perp}}{\wh{U}_{i,j}-\wt{U}^*_{i,j}} \\
\leq\sum_{i\neq j}\l\|\m P_{L_{i,j}}(\wt{X}_{i,j}-\Sigma)\r\|_F \l\|\wh{U}_{i,j}-\wt{U}^*_{i,j}\r\|_F 
+ \sum_{i\neq j}\l\|\wt{X}_{i,j}-\Sigma\r\| \l\|P_{L_{i,j}^\perp} \wh{U}_{i,j} P_{L_{i,j}^\perp}\r\|_1 \\
\leq \sum_{(i,j)\in\wt{J}} \sqrt{2\rank(\wt{U}^*_{i,j})}\l\|\wt{X}_{i,j}-\Sigma\r\| \l\|\wh{U}_{i,j}-\wt{U}^*_{i,j}\r\|_F 
+ \sum_{i\neq j}\l\|\wt{X}_{i,j}-\Sigma\r\| \l\|P_{L_{i,j}^\perp} \wh{U}_{i,j} P_{L_{i,j}^\perp}\r\|_1\\
 \leq\sum_{(i,j)\in\wt{J}} \sqrt{4}\l\|\wt{X}_{i,j}-\Sigma\r\| \l\|\wh{U}_{i,j}-\wt{U}^*_{i,j}\r\|_F 
 + \sum_{i\neq j}\l\|\wt{X}_{i,j}-\Sigma\r\| \l\|P_{L_{i,j}^\perp} \wh{U}_{i,j} P_{L_{i,j}^\perp}\r\|_1.
\end{multline}
Combining inequalities (\ref{bound-1}, \ref{bound-2}, \ref{bound-3}, \ref{bound-4}) with \eqref{ineq-2.1}, we deduce that
\begin{multline*}
\frac{1}{N}\sum_{i\neq j}\l\|\Sigma-\wh{S}_\lambda+\sqrt{N}(\widetilde{U}_{i,j}^*-\wh U_{i,j})\r\|_F^2 
+ \frac{1}{N}\sum_{i\neq j}\l\|S - \wh S_\lambda + \sqrt{N}(\widetilde{U}_{i,j}^* - \wh U_{i,j})\r\|_F^2 \\
\leq \l\|\Sigma-S\r\|_F^2+2\l(\sqrt{2\rank(S)}\l\|\Delta\r\|  \l\|\wh{S}_\lambda-S\r\|_F+\l\|\Delta\r\| \l\|P_{L^\perp}\wh{S}_\lambda P_{L^\perp}\r\|_1\r) \\
+\frac{2}{\sqrt{N}} \l(\sum_{(i,j)\in\wt{J}} \sqrt{4}\l\|\wt{X}_{i,j}-\Sigma\r\| \l\|\wh{U}_{i,j}-\wt{U}^*_{i,j}\r\|_F 
+ \sum_{i\neq j}\l\|\wt{X}_{i,j}-\Sigma\r\| \l\|P_{L_{i,j}^\perp} \wh{U}_{i,j} P_{L_{i,j}^\perp}\r\|_1\r) \\
+ \lambda_1\Big(\sqrt{\rank(S)}\l\|\wh{S}_\lambda-S\r\|_F 
- \l\|P_{L^\perp}\wh{S}_\lambda P_{L^\perp}\r\|_1\Big)\\
+ \lambda_2\l(\sum_{(i,j)\in\wt{J}}\sqrt{2}\l\|\wh{U}_{i,j}-\wt{U}^*_{i,j}\r\|_F-\sum_{i\neq j}\l\|P_{L_{i,j}^\perp}\wh{U}_{i,j}P_{L_{i,j}^\perp}\r\|_1\r),
\end{multline*}
which is equivalent to
%==Ineq-2.2==%
\begin{multline}\label{ineq-2.2}
\frac{1}{N}\sum_{i\neq j}\l\|\Sigma-\wh{S}_\lambda+\sqrt{N}(\widetilde{U}_{i,j}^*-\wh U_{i,j})\r\|_F^2  
+\frac{1}{N}\sum_{i\neq j}\l\|S - \wh S_\lambda + \sqrt{N}(\widetilde{U}_{i,j}^* - \wh U_{i,j})\r\|_F^2 \\
+\l(\lambda_1-2\l\|\Delta\r\|\r)\l\|P_{L^\perp}\wh{S}_\lambda P_{L^\perp}\r\|_1
+\l( \lambda_2-\frac{2}{\sqrt{N}}\max_{i\neq j}\l\|\wt{X}_{i,j}-\Sigma\r\| \r)\l\|P_{L_{i,j}^\perp}\wh{U}_{i,j} P_{L_{i,j}^\perp}\r\|_1 \\
\leq \l\|\Sigma-S\r\|_F^2 + \l( 2\sqrt{2\rank(S)}\l\|\Delta\r\| + \lambda_1\sqrt{\rank(S)} \r)\l\|\wh{S}_\lambda-S\r\|_F \\
+ \sum_{(i,j)\in\wt{J}} \l( \frac{4}{\sqrt{N}} \l\|\wt{X}_{i,j}-\Sigma\r\| + \lambda_2\sqrt{2} \r) \l\|\wh{U}_{i,j}-\wt{U}^*_{i,j}\r\|_F.
\end{multline}
Now consider the event $$\m E_1:=\l\{ \lambda_1 \geq 2\l\|\Delta\r\|, \lambda_2\geq\frac{3}{\sqrt{N}}\max_{i\neq j}\l\|\wt{X}_{i,j}-\Sigma\r\|  \r\}.$$
We will derive a bound for $\l\|\wh{S}_\lambda-\Sigma\r\|_F^2+\sum_{i\neq j}\l\|\wh{U}_{i,j}-\wt{U}^*_{i,j}\r\|_F^2$ on $\m E_1$.
Applying the identity $\l\|A+B\r\|_F^2 = \l\|A\r\|_F^2 + \l\|B\r\|_F^2 + 2\dotp{A}{B}$ to the the left-hand side of \eqref{ineq-2.2}, we get that on the event $\m E_1$, 
%==Ineq-2.3==%
\begin{multline}\label{ineq-2.3}
 \l\|\Sigma-\wh{S}_\lambda\r\|_F^2 + \l\|S-\wh{S}_\lambda\r\|_F^2 + 2\sum_{i\neq j}\l\|\wt{U}^*_{i,j}-\wh{U}_{i,j}\r\|_F^2 \\
\leq\l\|\Sigma-S\r\|_F^2 + \l( 2\sqrt{2\rank(S)}\l\|\Delta\r\| + \lambda_1\sqrt{\rank(S)} \r)\l\|\wh{S}_\lambda-S\r\|_F \\
+ \frac{2}{\sqrt{N}}\sum_{i\neq j}\dotp{\Sigma-\wh{S}_\lambda}{\wt{U}^*_{i,j}-\wh{U}_{i,j}} 
+ \frac{2}{\sqrt{N}}\sum_{i\neq j}\dotp{S-\wh{S}_\lambda}{\wt{U}^*_{i,j}-\wh{U}_{i,j}} \\
+ \sum_{(i,j)\in\wt{J}} \l( \frac{4}{\sqrt{N}} \l\|\wt{X}_{ij}-\Sigma\r\| + \lambda_2\sqrt{2} \r) \l\|\wh{U}_{i,j}-\wt{U}^*_{i,j}\r\|_F .
\end{multline}
We now bound the inner product terms on the right-hand side.
First, combining inequalities (\ref{bound-1}, \ref{bound-2}, \ref{bound-3}, \ref{bound-4}) with \eqref{ineq-1}, we deduce the following bound:
%==Ineq-2.4==%
\begin{multline}\label{ineq-2.4}
\frac{2}{N}\sum_{i\neq j} \dotp{\Sigma - \wh S_\lambda +\sqrt{N}(\widetilde{U}_{i,j}^*-\wh U_{i,j})}{S - \wh S_\lambda + \sqrt{N}(\widetilde{U}_{i,j}^* - \wh U_{i,j})} \\
+\l(\lambda_1-2\l\|\Delta\r\|\r)\l\|P_{L^\perp}\wh{S}_\lambda P_{L^\perp}\r\|_1 
+\l( \lambda_2-\frac{2}{\sqrt{N}}\max_{i\neq j}\l\|\wt{X}_{i,j}-\Sigma\r\| \r)\l\|P_{L_{i,j}^\perp}\wh{U}_{i,j} P_{L_{i,j}^\perp}\r\|_1 \\
\leq \l( 2\sqrt{2\rank(S)}\l\|\Delta\r\| + \lambda_1\sqrt{\rank(S)} \r)\l\|\wh{S}_\lambda-S\r\|_F \\
+ \sum_{(i,j)\in\wt{J}} \l( \frac{4}{\sqrt{N}} \l\|\wt{X}_{i,j}-\Sigma\r\| + \lambda_2\sqrt{2} \r) \l\|\wh{U}_{i,j}-\wt{U}^*_{i,j}\r\|_F.
\end{multline}
On the event $\m E_1$ along with the assumption that 
\begin{equation*}
\frac{2}{N}\sum_{i\neq j} \Big\langle {\Sigma - \wh S_\lambda +\sqrt{N}(\widetilde{U}_{i,j}^*-\wh U_{i,j})}, \\
{S - \wh S_\lambda + \sqrt{N}(\widetilde{U}_{i,j}^* - \wh U_{i,j})} \geq 0.
\end{equation*}
\eqref{ineq-2.4} implies that 
\begin{multline*}
\frac{1}{3}\lambda_2\l\|P_{L_{i,j}^\perp}\wh{U}_{i,j} P_{L_{i,j}^\perp}\r\|_1 
\leq \l( 2\sqrt{2\rank(S)}\l\|\Delta\r\| + \lambda_1\sqrt{\rank(S)} \r)\l\|\wh{S}_\lambda-S\r\|_F \\
+ \sum_{(i,j)\in\wt{J}} \l( \frac{4}{\sqrt{N}} \l\|\wt{X}_{i,j}-\Sigma\r\| + \lambda_2\sqrt{2} \r) \l\|\wh{U}_{i,j}-\wt{U}^*_{i,j}\r\|_F.
\end{multline*}
Recall that $L_{i,j}=\{0\}$ for any $(i,j)\notin\wt{J}$, hence
\begin{multline}\label{Ineq-2.5}
\lambda_2\sum_{(i,j)\notin\wt{J}}\l\|P_{L_{i,j}^\perp}\wh{U}_{i,j} P_{L_{i,j}^\perp}\r\|_1\\
=\lambda_2\sum_{(i,j)\notin\wt{J}}\l\|P_{L_{i,j}^\perp}\l(\wh{U}_{i,j}-\wt{U}_{i,j}^*\r) P_{L_{i,j}^\perp}\r\|_1 
= \lambda_2\sum_{(i,j)\notin\wt{J}}\l\|\wh{U}_{i,j}-\wt{U}^*_{i,j}\r\|_1 \\
\leq 3\l( 2\sqrt{2\rank(S)}\l\|\Delta\r\| + \lambda_1\sqrt{\rank(S)} \r)\l\|\wh{S}_\lambda-S\r\|_F \\
+ 3\sum_{(i,j)\in\wt{J}} \l( \frac{4}{\sqrt{N}} \l\|\wt{X}_{i,j}-\Sigma\r\| + \lambda_2\sqrt{2} \r) \l\|\wh{U}_{i,j}-\wt{U}^*_{i,j}\r\|_F.
\end{multline}
Next, we can estimate $\sum_{(i,j)\notin\wt{J}}\l|\dotp{\wh{S}_\lambda-S}{\wh{U}_{i,j}-\wt{U}^*_{i,j}}\r|$ as follows:
%==Ineq-2.6==%
\begin{align}\label{ineq-2.6}
&\quad\sum_{(i,j)\notin\wt{J}}\l|\dotp{\wh{S}_\lambda-S}{\wh{U}_{i,j}-\wt{U}^*_{i,j}}\r| \leq \l\|\wh{S}_\lambda-S\r\| \sum_{(i,j)\notin\wt{J}}\l\|\wh{U}_{i,j}-\wt{U}^*_{i,j}\r\|_1\nonumber\\
&\leq \frac{3\l\|\wh{S}_\lambda-S\r\|}{\lambda_2}\Bigg[ \l( 2\sqrt{2\rank(S)}\l\|\Delta\r\| + \lambda_1\sqrt{\rank(S)} \r)\l\|\wh{S}_\lambda-S\r\|_F\nonumber\\
&\qquad\qquad\qquad+ \sum_{(i,j)\in\wt{J}} \l( \frac{4}{\sqrt{N}} \l\|\wt{X}_{i,j}-\Sigma\r\| + \lambda_2\sqrt{2} \r) \l\|\wh{U}_{i,j}-\wt{U}^*_{i,j}\r\|_F \Bigg]\nonumber\\
&\leq3\l\|\wh{S}_\lambda-S\r\| \l[ (\sqrt{2}+1)\sqrt{\rank(S)}\frac{\lambda_1}{\lambda_2} \l\|\wh{S}_\lambda-S\r\|_F + \Big( \frac{4}{3}+\sqrt{2} \Big)\sum_{(i,j)\in\wt{J}} \l\|\wh{U}_{i,j}-\wt{U}^*_{i,j}\r\|_F \r],
\end{align}
where the last inequality holds on event $\m E_1$. This implies that 
%==Ineq-2.7==%
\begin{multline}\label{ineq-2.7}
\frac{2}{\sqrt{N}}\sum_{i\neq j}\dotp{S-\wh{S}_\lambda}{\wt{U}^*_{i,j}-\wh{U}_{i,j}}
\leq \frac{2}{\sqrt{N}}\sum_{(i,j)\in\wt{J}}\dotp{S-\wh{S}_\lambda}{\wt{U}^*_{i,j}-\wh{U}_{i,j}} +\frac{6}{\sqrt{N}}\l\|\wh{S}_\lambda-S\r\| \\
\times \Bigg[ (\sqrt{2}+1)\sqrt{\rank(S)}\frac{\lambda_1}{\lambda_2} \l\|\wh{S}_\lambda-S\r\|_F + \Big( \frac{4}{3}+\sqrt{2} \Big)\sum_{(i,j)\in\wt{J}} \l\|\wh{U}_{i,j}-\wt{U}^*_{i,j}\r\|_F \Bigg]\\
\leq 2 \cdot \frac{\l\|\wh{S}_\lambda-S\r\|_F}{{2}} \cdot 2\frac{1}{\sqrt{N}} \sqrt{\l|\wt{J}\r|} \sqrt{\sum_{(i,j)\in\wt{J}} \l\|\wh{U}_{i,j}-\wt{U}^*_{i,j}\r\|_F^2}
+ (6\sqrt{2}+6)\frac{\lambda_1}{\lambda_2} \sqrt{\frac{\rank(S)}{N}}\l\|\wh{S}_\lambda-S\r\|_F^2\\
+ 2 \cdot \frac{\l\|\wh{S}_\lambda-S\r\|_F}{{2}} \cdot (8+6\sqrt{2})\frac{1}{\sqrt{N}} \sqrt{\l|\wt{J}\r|}  \sqrt{\sum_{(i,j)\in\wt{J}} \l\|\wh{U}_{i,j}-\wt{U}^*_{i,j}\r\|_F^2} \\
\leq\l[(6\sqrt{2}+6) \sqrt{\frac{\rank(S)}{N}} \frac{\lambda_1}{\lambda_2}+ \frac{1}{2}\r]\l\|\wh{S}_\lambda-S\r\|_F^2 
+\l[4+(6\sqrt{2}+8)^2\r]\frac{|\wt{J}|}{N}\sum_{(i,j)\in\wt{J}} \l\|\wh{U}_{i,j}-\wt{U}^*_{i,j}\r\|_F^2,
\end{multline}
where the second inequality follows from the fact that $\l\|A\r\|\leq\l\|A\r\|_F$ for any symmetric matrix A, and the last inequality follows from the fact that $2ab\leq a^2+b^2$ for any real numbers $a,b$.
\newline Similarly, we deduce that
%==Ineq-2.8==%
\begin{multline}\label{ineq-2.8}
\quad\frac{2}{\sqrt{N}}\sum_{i\neq j}\dotp{\Sigma-\wh{S}_\lambda}{\wt{U}^*_{i,j}-\wh{U}_{i,j}}
\leq (6\sqrt{2}+6) \sqrt{\frac{\rank(S)}{N}} \frac{\lambda_1}{\lambda_2} \l\|\wh{S}_\lambda-S\r\|_F^2 \\
+ \l[(6\sqrt{2}+6) \sqrt{\frac{\rank(S)}{N}} \frac{\lambda_1}{\lambda_2}+ \frac{1}{2}\r]\l\|\wh{S}_\lambda-\Sigma\r\|_F^2 \\
+\l[4+(6\sqrt{2}+8)^2\r]\frac{|\wt{J}|}{N}\sum_{(i,j)\in\wt{J}} \l\|\wh{U}_{i,j}-\wt{U}^*_{i,j}\r\|_F^2.
\end{multline}
Combining (\ref{ineq-2.7}, \ref{ineq-2.8}) with \eqref{ineq-2.3}, one sees that on event $\m E_1$, %and for $U_{i,j}=\wt{U}^*_{i,j},(i,j)\in I_n^2$,
%==Ineq-2.9==%
\begin{multline}\label{ineq_2_9_temporary}
\l\|\Sigma-\wh{S}_\lambda\r\|_F^2 + \l\|S-\wh{S}_\lambda\r\|_F^2 + 2\sum_{i\neq j}\l\|\wt{U}^{*}_{i,j}-\wh{U}_{i,j}\r\|_F^2 
\\
\leq \l\|\Sigma-S\r\|_F^2 + \l( 2\sqrt{2\rank(S)}\l\|\Delta\r\| + \lambda_1\sqrt{\rank(S)} \r)\l\|\wh{S}_\lambda-S\r\|_F 
\\
+ \sum_{(i,j)\in\wt{J}} \l( \frac{4}{\sqrt{N}} \l\|\wt{X}_{i,j}-\Sigma\r\| + \lambda_2\sqrt{2} \r) \l\|\wh{U}_{i,j}-\wt{U}^*_{i,j}\r\|_F + 2\l[4+(6\sqrt{2}+8)^2\r]\frac{|\wt{J}|}{N}\sum_{(i,j)\in\wt{J}} \l\|\wh{U}_{i,j}-\wt{U}^*_{i,j}\r\|_F^2
\\
+ \l[2(6\sqrt{2}+6) \sqrt{\frac{\rank(S)}{N}}\frac{\lambda_1}{\lambda_2} + \frac{1}{2}\r] \l\|\wh{S}_\lambda-S\r\|_F^2 + \l[(6\sqrt{2}+6) \sqrt{\frac{\rank(S)}{N}}\frac{\lambda_1}{\lambda_2} + \frac{1}{2}\r] \l\|\wh{S}_\lambda-\Sigma\r\|_F^2 
\\
\leq \l\|\Sigma-S\r\|_F^2 + 2 \lambda_1\l( \sqrt{2}+1\r)\sqrt{\rank(S)} \cdot \frac{1}{2}\l\|\wh{S}_\lambda-S\r\|_F  
\\
+ 2 \frac{1}{\sqrt{2}}\lambda_2(\frac{4}{3}+\sqrt{2})\sqrt{|\wt{J}|} \cdot \frac{1}{\sqrt{2}}\sqrt{\sum_{(i,j)\in\wt{J}} \l\|\wh{U}_{i,j}-\wt{U}^*_{i,j}\r\|_F^2} 
\\
+\l[2(6\sqrt{2}+6) \sqrt{\frac{\rank(S)}{N}}\frac{\lambda_1}{\lambda_2} + \frac{1}{2}\r] \l\|\wh{S}_\lambda-S\r\|_F^2 + \l[(6\sqrt{2}+6) \sqrt{\frac{\rank(S)}{N}}\frac{\lambda_1}{\lambda_2} + \frac{1}{2}\r] \l\|\wh{S}_\lambda-\Sigma\r\|_F^2 \\
+ 2 \l[4+(6\sqrt{2}+8)^2\r]\frac{|\wt{J}|}{N}\sum_{(i,j)\in\wt{J}} \l\|\wh{U}_{i,j}-\wt{U}^*_{i,j}\r\|_F^2 \\
\leq \l\|\Sigma-S\r\|_F^2 + \frac{1}{4}\l\|\wh{S}_\lambda-S\r\|_F^2 + \frac{1}{2}\sum_{i\neq j}\l\|\wt{U}^*_{i,j}-\wh{U}_{i,j}\r\|_F^2 + \lambda_1^2\l( \sqrt{2}+1\r)^2\rank(S) + \frac{1}{2}\lambda_2^2(\frac{4}{3}+\sqrt{2})^2 |\wt{J}| \\
+  \l[2(6\sqrt{2}+6) \sqrt{\frac{\rank(S)}{N}}\frac{\lambda_1}{\lambda_2} + \frac{1}{2}\r] \l\|\wh{S}_\lambda-S\r\|_F^2 + \l[(6\sqrt{2}+6) \sqrt{\frac{\rank(S)}{N}}\frac{\lambda_1}{\lambda_2} + \frac{1}{2}\r] \l\|\wh{S}_\lambda-\Sigma\r\|_F^2 \\
+ 2 \l[4+(6\sqrt{2}+8)^2\r]\frac{|\wt{J}|}{N}\sum_{(i,j)\in\wt{J}} \l\|\wh{U}_{i,j}-\wt{U}^*_{i,j}\r\|_F^2.
\end{multline}
Assuming that %$\lambda_1\leq\lambda_2$, 
$2(6+6\sqrt{2})\sqrt{\frac{\rank(S)}{N}} \frac{\lambda_1}{\lambda_2}  \leq \frac{1}{8}$ and $2\l[ 4+(6\sqrt{2}+8)^2\r]\frac{|\wt{J}|}{N}\leq\frac{11}{8}$, we conclude that
%==bound case 1==%
\begin{multline}
\label{bound_case_1}
\frac{1}{8}\l( \l\|\Sigma-\wh{S}_\lambda\r\|_F^2 + \sum_{i\neq j}\l\|\wt{U}^*_{i,j}-\wh{U}_{i,j}\r\|_F^2\r)
\\ 
\leq \l\|\Sigma-S\r\|_F^2+\rank(S)\lambda_1^2\l( \sqrt{2}+1\r)^2 + \lambda_2^2\frac{(4/3+\sqrt{2})^2}{2} |\wt{J}|.
\end{multline}
The assumptions above are valid provided $\rank(S)\leq \frac{1}{56000} \cdot \frac{n^2\lambda_2^2}{\lambda_1^2}$ and $|\wt{J}|\leq\frac{N}{402}$. Note that if we apply the inequality $2ab\leq a^2+b^2$ in the derivation above with different choices of constants, we can reduce the conditions on $\rank(S)$ and $|\wt{J}|$ to 
\begin{equation}\label{condition on rank of S}
\rank(S) \leq c_1 \frac{n^2\lambda_2^2}{\lambda_1^2},\qquad \forall c_1\leq\frac{1}{5980}
\end{equation}
and
\[
|\wt{J}| \leq c_2N, \qquad \forall c_2 \leq \frac{1}{295}.
\]
%====Case2 when the inner product is positive======%
\newline\textbf{Case 2:} Assume that 
\[
\frac{2}{N}\sum_{i\neq j} \dotp{\Sigma - \wh S_\lambda +\sqrt{N}(\widetilde{U}_{i,j}^*-\wh U_{i,j})}{S - \wh S_\lambda + \sqrt{N}(\widetilde{U}_{i,j}^* - \wh U_{i,j})}\allowbreak < 0.
\] 
We start with several lemmas.
%==lemma 3.2==%
\begin{lemma}\label{lemma 3.2}
On the event 
\[\m E_2:=\l\{\lambda_1\geq4 \l\|\frac{1}{N}\sum_{i\neq j}\wt{X}_{i,j}-\Sigma(k)\r\|, \lambda_2 \geq \frac{4}{\sqrt{N}}\max_{i\neq j}\l\|\wt{X}_{i,j}-\Sigma(k)\r\| \r\},
\]
the following inequality holds
\begin{multline*}
\lambda_1 \l\|P_{L(k)^\perp}\wh{S}_\lambda P_{L(k)^\perp}\r\|_1 + \lambda_2 \sum_{i\neq j}\l\|P_{L_{i,j}^\perp} \wh{U}_{i,j} P_{L_{i,j}^\perp}\r\|_1 \\
\leq 3\l( \lambda_1 \l\|\m P_{L(k)}(\wh{S}_\lambda-\Sigma(k)) \r\|_1 + \lambda_2 \sum_{(i,j)\in\wt{J}}\l\|\m P_{L_{i,j}}(\wh{U}_{i,j}-\wt{U}_{i,j}^*) \r\|_1 \r),
\end{multline*}
where $L(k) = \im\l(\Sigma(k)\r)$, $L_{i,j}=\im\l( \wt{U}_{i,j}^*\r)$, and $P_{L(k)}$, $P_{L_{i,j}}$ are the orthogonal projections onto the corresponding subspaces.
\end{lemma}
\begin{proof}[Proof of Lemma \ref{lemma 3.2}]
Denote
\begin{align*}
Q(S,U_{1,2},\ldots,U_{n,n-1}):=\frac{1}{N}\sum_{i\neq j}\l\|\wt{Y}_{i,j}\wt{Y}_{i,j}^T-S-\sqrt{N}U_{i,j} \r\|_F^2.
\end{align*}
By definition of $\wh{S}_\lambda$, 
\begin{align}\label{3.2.0}
&\quad Q(\wh{S}_\lambda,\wh{U}_{1,2},\ldots, \wh{U}_{n,n-1}) - Q(\Sigma(k),\wt{U}^*_{1,2},\ldots,\wt{U}^*_{n,n-1}) \nonumber\\
&\leq \lambda_1 \l( \l\|\Sigma(k)\r\|_1 -  \|\wh{S}_\lambda\|_1 \r)  + \lambda_2\sum_{i\neq j}(\|\wt{U}^*_{i,j}\|_1 - \|\wh{U}_{i,j}\|_1) .
\end{align}
By convexity of the $\l\|\cdot\r\|_1$ norm, for any $V \in \partial\l\|\Sigma(k)\r\|_1$, $\l\|\Sigma(k)\r\|_1 - \|\wh{S}_\lambda\|_1 \leq \dotp{V}{\Sigma(k)-\wh{S}_\lambda}$. Let $r=\rank\l(\Sigma(k)\r) \leq k$, we have the representation $V=\sum_{j=1}^{r}v_jv_j^T + P_{L(k)^\perp} W P_{L(k)^\perp}$, where $\l\|W\r\| \leq 1$. By duality between the spectral and nuclear norm (Proposition \ref{lem:trace duality}), we deduce that with an appropriate choice of W,
\begin{multline}\label{3.2.1}
\l\|\Sigma(k)\r\|_1 - \|\wh{S}_\lambda\|_1 \leq \dotp{V}{\Sigma(k)-\wh{S}_\lambda}
\\
= \dotp{\m P_{L(k)}(V)}{\Sigma(k) - \wh{S}_\lambda} + \dotp{P_{L(k)^\perp} W P_{L(k)^\perp}}{\Sigma(k) - \wh{S}_\lambda} 
\\
\leq \l\|\m P_{L(k)} (\Sigma(k)-\wh{S}_\lambda) \r\|_1 - \l\| P_{L(k)^\perp} \wh{S}_\lambda P_{L(k)^\perp} \r\|_1.
\end{multline}
Similalry,
\begin{equation}\label{3.2.2}
\sum_{i\neq j} (\|\wt{U}^*_{i,j}\|_1 - \|\wh{U}_{i,j}\|_1) \leq \sum_{i\neq j} \l( \l\|\m P_{L_{i,j}} (\wt{U}^*_{i,j}-\wh{U}_{i,j}) \r\|_1 - \l\| P_{L_{i,j}^\perp} \wh{U}_{i,j} P_{L_{i,j}^\perp} \r\|_1\r),
\end{equation}
where $L_{i,j}$ is the image of $\wt{U}^*_{i,j}$, $\forall (i,j)\in I_n^2$. 
\newline On the other hand, recall that $\wt{Y}_{i,j}\wt{Y}_{i,j}^T = \wt{X}_{i,j} + \sqrt{N}\wt{U}^*_{i,j}$ and $\nabla Q$ is given by the first term in equation \eqref{gradient of loss function}. Convexity of Q implies that
\begin{align}\label{3.2.3}
&\quad Q(\wh{S}_\lambda,\wh{U}_{1,2},\ldots, \wh{U}_{n,n-1}) - Q(\Sigma(k),\wt{U}^*_{1,2},\ldots,\wt{U}^*_{n,n-1}) \nonumber\\
&\geq \dotp{\nabla Q\l(\Sigma(k),\wt{U}^*_{1,2},\ldots,\wt{U}^*_{n,n-1}\r)}{(\wh{S}_\lambda - \Sigma(k), \wh{U}_{1,2} - \wt{U}^*_{1,2}, \ldots, \wh{U}_{n,n-1} - \wt{U}^*_{n,n-1})} \nonumber\\
&=-\frac{2}{N}\sum_{i\neq j} \dotp{\wt{X}_{i,j}-\Sigma(k)}{\wh{S}_\lambda-\Sigma(k)} - \frac{2}{\sqrt{N}}\sum_{i\neq j}\dotp{\wt{X}_{i,j}-\Sigma(k)}{\wh{U}_{i,j}-\wt{U}^*_{i,j}} \nonumber\\
&=2\dotp{\Sigma(k)-\frac{1}{N}\sum_{i\neq j}\wt{X}_{i,j}}{\wh{S}_\lambda-\Sigma(k)} + \frac{2}{\sqrt{N}} \sum_{i\neq j}\dotp{\Sigma(k) - \wt{X}_{i,j}}{\wh{U}_{i,j}-\wt{U}^*_{i,j}} \nonumber\\
&\geq -2\l\|\Sigma(k)-\frac{1}{N}\sum_{i\neq j}\wt{X}_{i,j} \r\| \l\|\wh{S}_\lambda - \Sigma(k)\r\|_1 - \frac{2}{\sqrt{N}} \max_{i\neq j}\l\|\Sigma(k)-\wt{X}_{i,j}\r\| \sum_{i\neq j}\l\|\wh{U}_{i,j}-\wt{U}^*_{i,j}\r\|_1.
\end{align}
On the event 
\[\m E_2:=\l\{\lambda_1\geq4 \l\|\frac{1}{N}\sum_{i\neq j}\wt{X}_{i,j}-\Sigma(k)\r\|, \lambda_2 \geq \frac{4}{\sqrt{N}}\max_{i\neq j}\l\|\wt{X}_{i,j}-\Sigma(k)\r\| \r\},
\]
the inequality \eqref{3.2.3} implies that
\begin{multline}\label{3.2.4}
\quad Q(\wh{S}_\lambda,\wh{U}_{1,2},\ldots, \wh{U}_{n,n-1}) - Q(\Sigma(k),\wt{U}^*_{1,2},\ldots,\wt{U}^*_{n,n-1}) 
\\
\geq -\frac{1}{2}\l( \lambda_1\l\|\wh{S}_\lambda - \Sigma(k)\r\|_1 + \lambda_2\sum_{i\neq j}\l\|\wh{U}_{i,j}-\wt{U}^*_{i,j}\r\|_1 \r).
\end{multline}
Moreover, note that
\[
\l\|\wh{S}_\lambda - \Sigma(k)\r\|_1 \leq \l\| \m P_{L(k)} (\wh{S}_\lambda - \Sigma(k)) \r\|_1 + \l\| P_{L(k)^\perp}\wh{S}_\lambda P_{L(k)^\perp}\r\|_1
\]
and
\[
\l\|\wh{U}_{i,j}-\wt{U}^*_{i,j}\r\|_1 \leq \l\| \m P_{L_{i,j}} (\wh{U}_{i,j}-\wt{U}^*_{i,j}) \r\|_1 + \l\| P_{L_{i,j}^\perp}\wh{U}_{i,j} P_{L_{i,j}^\perp}\r\|_1.
\]
Combining these inequalities with \eqref{3.2.4}, we get the lower bound
\begin{multline}
\label{3.2.5}
\quad Q(\wh{S}_\lambda,\wh{U}_{1,2},\ldots, \wh{U}_{n,n-1}) - Q(\Sigma(k),\wt{U}^*_{1,2},\ldots,\wt{U}^*_{n,n-1}) \nonumber\\
\geq -\frac{1}{2}\Bigg[ \lambda_1 \l(\l\| \m P_{L(k)} (\wh{S}_\lambda - \Sigma(k)) \r\|_1 + \l\| P_{L(k)^\perp}\wh{S}_\lambda P_{L(k)^\perp}\r\|_1\r) 
\\
+ \lambda_2 \l( \l\| \m P_{L_{i,j}} (\wh{U}_{i,j}-\wt{U}^*_{i,j}) \r\|_1 + \l\| P_{L_{i,j}^\perp}\wh{U}_{i,j} P_{L_{i,j}^\perp}\r\|_1\r) \Bigg].
\end{multline}
Combining %\eqref{3.2.0}\eqref{3.2.1}\eqref{3.2.2}
(\ref{3.2.0}, \ref{3.2.1}, \ref{3.2.2}) with the lower bound \eqref{3.2.5}, we deduce the ``sparsity inequality''
\begin{multline*}
\lambda_1 \l\|P_{L(k)^\perp}\wh{S}_\lambda P_{L(k)^\perp}\r\|_1 + \lambda_2 \sum_{i\neq j}\l\|P_{L_{i,j}^\perp} \wh{U}_{i,j} P_{L_{i,j}^\perp}\r\|_1 \\
\leq 3\l( \lambda_1 \l\|\m P_{L(k)}(\wh{S}_\lambda-\Sigma(k)) \r\|_1 + \lambda_2 \sum_{(i,j)\in\wt{J}}\l\|\m P_{L_{i,j}}(\wh{U}_{i,j}-\wt{U}_{i,j}^*) \r\|_1 \r).
\end{multline*}
\end{proof}
%==lemma-3.3==%
\begin{lemma}\label{lemma-3.3}
Assume that %$\lambda_1\leq \lambda_2$ and that 
$\max\l\{6\sqrt{2} \cdot \frac{\lambda_1}{\lambda_2} \sqrt{\frac{k}{N}}, 7\sqrt{\frac{|\wt{J}|}{N}} \r\}\leq\frac{1}{4}$. Then on the event $\m E_2$ of Lemma \ref{lemma 3.2}, the following inequality holds:
\begin{align*}
\l\|\wh{S}_\lambda-\Sigma\r\|_F^2 + \sum_{i\neq j}\l\|\wh{U}_{i,j}-\wt{U}_{i,j}^*\r\|_F^2 \leq \frac{2}{N}\sum_{i\neq j}\l\|\wh{S}_\lambda-\Sigma + \sqrt{N}(\wh{U}_{i,j}-\wt{U}_{i,j}^*)\r\|_F^2.
\end{align*}
\end{lemma}
\begin{proof}[Proof of Lemma \ref{lemma-3.3}]
First, we consider the decomposition
\begin{multline}
\frac{2}{\sqrt{N}}\sum_{i\neq j}\dotp{\wh{S}_\lambda - \Sigma}{\wh{U}_{i,j} - \wt{U}_{i,j}^*} 
\\
= \underbrace{\frac{2}{\sqrt{N}}\sum_{(i,j)\in\wt{J}}\dotp{\wh{S}_\lambda - \Sigma}{\wh{U}_{i,j} - \wt{U}_{i,j}^*}}_{\mathrm{I}} 
+ \underbrace{\frac{2}{\sqrt{N}}\sum_{(i,j)\notin\wt{J}}\dotp{\wh{S}_\lambda - \Sigma}{\wh{U}_{i,j} - \wt{U}_{i,j}^*}}_{\mathrm{II}}.
\end{multline}
For the term $\mathrm{I}$, we have that
\begin{equation}\label{3.3.1}
\mathrm{I}\leq\frac{2\l\|\wh{S}_\lambda - \Sigma\r\|_F}{\sqrt{N}} \sum_{(i,j)\in\wt{J}} \l\|\wh{U}_{i,j} - \wt{U}_{i,j}^*\r\|_F
\leq 2\l\|\wh{S}_\lambda - \Sigma\r\|_F \sqrt{\frac{|\wt{J}|}{N}} \sqrt{\sum_{(i,j)\in\wt{J}} \l\|\wh{U}_{i,j} - \wt{U}_{i,j}^*\r\|_F^2}.
\end{equation}
To estimate the term $\mathrm{II}$, note that $\m P_{L(k)}(\Sigma) = \Sigma(k) = \m P_{L(k)}(\Sigma(k))$ as $L(k)=\im(\Sigma(k))$. Moreover, $\sum_{(i,j)\notin\wt{J}}\l\|P_{L_{i,j}^\perp} \wh{U}_{i,j} P_{L_{i,j}^\perp} \r\|_1 = \sum_{(i,j)\notin\wt{J}}\l\|\wh{U}_{i,j} - \wt{U}_{i,j}^*\r\|_1 $, hence Lemma \ref{lemma 3.2} yields that on the event $\m E_2$,
\begin{align}\label{3.3.2}
\mathrm{II} &\leq \frac{2\l\|\wh{S}_\lambda - \Sigma\r\|}{\sqrt{N}} \sum_{(i,j)\notin\wt{J}}\l\|\wh{U}_{i,j} - \wt{U}_{i,j}^*\r\|_1 \nonumber\\
&\leq \frac{2\l\|\wh{S}_\lambda - \Sigma\r\|}{\sqrt{N}} \cdot 3\l( \frac{\lambda_1}{\lambda_2} \l\|\m P_{L(k)}(\wh{S}_\lambda-\Sigma(k)) \r\|_1 + \sum_{(i,j)\in\wt{J}}\l\|\m P_{L_{i,j}}(\wh{U}_{i,j}-\wt{U}_{i,j}^*) \r\|_1 \r) \nonumber \\
&\leq \frac{6\l\|\wh{S}_\lambda - \Sigma\r\|}{\sqrt{N}} \Bigg( \frac{\lambda_1}{\lambda_2} \sqrt{2\rank(\Sigma(k))} \l\|\m P_{L(k)}(\wh{S}_\lambda-\Sigma) \r\|_F \nonumber\\
&\hspace{30mm} + \sum_{(i,j)\in\wt{J}} \sqrt{2\rank(\wt{U}_{i,j}^*)} \l\|\m P_{L_{i,j}}(\wh{U}_{i,j}-\wt{U}_{i,j}^*) \r\|_F \Bigg) \nonumber \\
&\leq \frac{6\l\|\wh{S}_\lambda - \Sigma\r\|}{\sqrt{N}} \l( \frac{\lambda_1}{\lambda_2} \sqrt{2k} \l\|\wh{S}_\lambda-\Sigma \r\|_F + \sum_{(i,j)\in\wt{J}} \sqrt{4} \l\| \wh{U}_{i,j}-\wt{U}_{i,j}^* \r\|_F \r) \nonumber \\
&\leq 6\l\|\wh{S}_\lambda - \Sigma\r\|_F \l( \frac{\lambda_1}{\lambda_2} \sqrt{\frac{2k}{N}} \l\|\wh{S}_\lambda-\Sigma \r\|_F + 2 \sqrt{\frac{|\wt{J}|}{N}} \sqrt{\sum_{(i,j)\in\wt{J}} \l\|\wh{U}_{i,j} - \wt{U}_{i,j}^*\r\|_F^2} \r).
\end{align}
%where we used the assumption $\lambda_1 \leq \lambda_2$ in the last inequality. 
Therefore,
\begin{align}\label{3.3.3}
&\quad\frac{2}{\sqrt{N}}\sum_{i\neq j}\dotp{\wh{S}_\lambda - \Sigma}{\wh{U}_{i,j} - \wt{U}_{i,j}^*} \nonumber\\
&\geq - 6\l\|\wh{S}_\lambda - \Sigma\r\|_F^2 \frac{\lambda_1}{\lambda_2} \sqrt{\frac{2k}{N}} - \sqrt{\frac{|\wt{J}|}{N}} \cdot 2 \cdot \l(\sqrt{7}\l\|\wh{S}_\lambda - \Sigma\r\|_F\r) \cdot \l(\sqrt{7}\sqrt{\sum_{(i,j)\in\wt{J}} \l\|\wh{U}_{i,j} - \wt{U}_{i,j}^*\r\|_F^2}\r) \nonumber\\
&\geq - 6\l\|\wh{S}_\lambda - \Sigma\r\|_F^2  \frac{\lambda_1}{\lambda_2} \sqrt{\frac{2k}{N}}  - \sqrt{\frac{|\wt{J}|}{N}} \l( 7\l\|\wh{S}_\lambda - \Sigma\r\|_F^2 + 7\sum_{(i,j)\in\wt{J}} \l\|\wh{U}_{i,j} - \wt{U}_{i,j}^*\r\|_F^2 \r) \nonumber\\
&\geq -\l( 6\sqrt{2} \frac{\lambda_1}{\lambda_2} \sqrt{\frac{k}{N}} + 7\sqrt{\frac{|\wt{J}|}{N}}  \r)\l\|\wh{S}_\lambda - \Sigma\r\|_F^2 - 7\sqrt{\frac{|\wt{J}|}{N}}\sum_{(i,j)\in\wt{J}} \l\|\wh{U}_{i,j} - \wt{U}_{i,j}^*\r\|_F^2,
\end{align}
where we used $2ab\leq a^2+b^2$ in the second inequality. Finally, given the assumption that $$\max \Big\{6\sqrt{2} \cdot \frac{\lambda_1}{\lambda_2}\sqrt{\frac{k}{N}}, 7\sqrt{\frac{\wt{J}}{N}} \Big\}\leq\frac{1}{4},$$ on the event $\m E_2$ we have that
\begin{align}\label{3.3.4}
&\quad\frac{2}{N}\sum_{i\neq j}\l\|\wh{S}_\lambda-\Sigma + \sqrt{N}(\wh{U}_{i,j}-\wt{U}_{i,j}^*)\r\|_F^2 \nonumber\\
&=2\l( \l\|\wh{S}_\lambda-\Sigma\r\|_F^2 + \sum_{i\neq j}\l\|\wh{U}_{i,j}-\wt{U}_{i,j}^*\r\|_F^2 + \frac{2}{\sqrt{N}}\sum_{i\neq j}\dotp{\wh{S}_\lambda - \Sigma}{\wh{U}_{i,j} - \wt{U}_{i,j}^*}\r) \nonumber \\
&\geq 2\l( \l\|\wh{S}_\lambda-\Sigma\r\|_F^2 + \sum_{i\neq j}\l\|\wh{U}_{i,j}-\wt{U}_{i,j}^*\r\|_F^2 - \frac{1}{2}\l\|\wh{S}_\lambda-\Sigma\r\|_F^2 - \frac{1}{4}\sum_{i\neq j}\l\|\wh{U}_{i,j}-\wt{U}_{i,j}^*\r\|_F^2\r)\nonumber\\
&\geq \l\|\wh{S}_\lambda-\Sigma\r\|_F^2 + \sum_{i\neq j}\l\|\wh{U}_{i,j}-\wt{U}_{i,j}^*\r\|_F^2.
\end{align}
\end{proof}
%==Argument to combine two events==%
\begin{remark}
We now consider the intersection of events $\m E_1$ and $\m E_2$.
Consider $k=\lfloor {\frac{N\lambda_2^2}{1200\lambda_1^2}} \rfloor$, $\big| \widetilde{J} \, \big|\leq\frac{N}{6400}$ (implying that $\big| \widetilde{J} \, \big| \leq\frac{N}{800}$). Corollary \ref{rank-k approximation} guarantees that $\l\|\Sigma(k) - \Sigma\r\| \leq \l\|\Sigma\r\| \sqrt{\frac{\rk(\Sigma)}{k}}$, so
\begin{multline*}
4\l\|\frac{1}{N}\sum_{i\neq j}\wt{X}_{i,j}-\Sigma(k)\r\| \leq 4\l( \l\|\frac{1}{N}\sum_{i\neq j}\wt{X}_{ij}-\Sigma\r\| + \l\|\Sigma - \Sigma(k)\r\| \r) \\
 \leq 4\l\|\frac{1}{N}\sum_{i\neq j}\wt{X}_{i,j} -\Sigma\r\| + 4\l\|\Sigma\r\| \sqrt{\frac{\rk(\Sigma)}{k}} 
 \leq 4\l\|\frac{1}{N}\sum_{i\neq j}\wt{X}_{i,j} -\Sigma\r\| + 140\l\|\Sigma\r\| \sqrt{\frac{\rk(\Sigma)}{N}}.
\end{multline*}
Similarly, for the second term we have that
\begin{multline*}
\quad\frac{4}{\sqrt{N}}\max_{i\neq j}\l\|\wt{X}_{i,j}-\Sigma(k)\r\|
\leq \frac{4}{\sqrt{N}}\max_{i\neq j} \l( \l\|\wt{X}_{i,j}-\Sigma\r\| + \l\|\Sigma(k)-\Sigma\r\|\r) \\
\leq  \frac{4}{\sqrt{N}}\max_{i\neq j}\l\|\wt{X}_{i,j}-\Sigma\r\|  + \frac{4}{\sqrt{N}}\l\|\Sigma(k)-\Sigma\r\| \\
\leq 4\frac{1}{\sqrt{N}}\max_{i\neq j}\l\|\wt{X}_{i,j}-\Sigma\r\| + \frac{4}{\sqrt{N}} \l\|\Sigma\r\| \sqrt{\frac{\rk(\Sigma)}{k}} \\
\leq 4\frac{1}{\sqrt{N}}\max_{i\neq j}\l\|\wt{X}_{i,j}-\Sigma\r\| + 140\frac{\l\|\Sigma\r\|\sqrt{\rk(\Sigma)}}{N}.
\end{multline*}
Therefore, the event
\begin{align*}
\m E:=\Bigg\{ &\lambda_1 \geq \frac{140\l\|\Sigma\r\|}{\sqrt{N}}\sqrt{\rk(\Sigma)} + 4\l\|  \frac{1}{N}\sum_{i\neq j}\wt{X}_{i,j}-\Sigma \r\|, \\
&\lambda_2 \geq \frac{140\l\|\Sigma\r\|}{N} \sqrt{\rk(\Sigma)} + 4\frac{1}{\sqrt{N}}\max_{i\neq j}\l\|\wt{X}_{i,j}-\Sigma\r\| \Bigg\}
\end{align*}
is a subset of both event $\m E_1$ and event $\m E_2$, and all previous results hold on the event $\m E$ naturally.
\end{remark}
%==Argument for bound in case 2}
Now applying law of cosines to the left-hand side of $$\frac{2}{N}\sum_{i\neq j} \dotp{\Sigma - \wh S_\lambda +\sqrt{N}(\widetilde{U}_{i,j}^*-\wh U_{i,j})}{S - \wh S_\lambda + \sqrt{N}(\widetilde{U}_{i,j}^* - \wh U_{i,j})}<0,$$ we get that
\begin{multline*}
\frac{1}{N}\sum_{i\neq j} \l\|\Sigma - \wh S_\lambda +\sqrt{N}(\widetilde{U}_{i,j}^*-\wh U_{i,j})\r\|_F^2 \\
+ \frac{1}{N}\sum_{i\neq j} \l\|S - \wh S_\lambda + \sqrt{N}(\widetilde{U}_{i,j}^* - \wh U_{i,j})\r\|_F^2 
< \l\|\Sigma-S\r\|_F^2.
\end{multline*}
This implies the inequality
%==ineq-3.1==%
\begin{align}\label{ineq-3.1}
\frac{1}{N}\sum_{i\neq j} \l\|\Sigma - \wh S_\lambda +\sqrt{N}(\widetilde{U}_{i,j}^*-\wh U_{i,j})\r\|_F^2 < \l\|\Sigma-S\r\|_F^2.
\end{align}
On the event $\m E$ with $k=\lfloor {\frac{N\lambda_2^2}{1200\lambda_1^2}} \rfloor$, $\big| \widetilde{J} \, \big|  \leq\frac{N}{800}$, we can combine the result of Lemma \ref{lemma-3.3} with the equation \eqref{ineq-3.1} to get that
%==bound for case 2==%
\begin{align}\label{bound_case_2}
\frac{1}{2} \l(\l\|\wh{S}_\lambda-\Sigma\r\|_F^2 + \sum_{i\neq j}\l\|\wh{U}_{i,j}-\wt{U}_{i,j}^*\r\|_F^2 \r) \leq \l\|\Sigma-S\r\|_F^2.
\end{align}
%==upper bound for Sigma only==%
This bound is consistent with \eqref{bound_case_1}, which provides upper bounds for both the estimation of $\Sigma$ and $\wt{U}_{i,j}^*,(i,j)\in I_n^2$. To complete the proof, we repeat part of the previous argument to derive an improved bound for the estimation of $\Sigma$ only, while treating $\wt{U}_{i,j}^*$ as ``nuisance parameters''. Let
\begin{align}
\label{loss:only S}
G(S):=\frac{1}{N} \sum_{i\neq j} \l\|\wt{Y}_{i,j}\wt{Y}_{i,j}^T - S - \sqrt{N} \wh{U}_{i,j} \r\|_F^2 + \lambda_1\l\|S\r\|_1
\end{align}
as before, and we note that the directional derivative of G at the point $\wh{S}_\lambda$ in the direction $S-\wh{S}_\lambda$ is nonnegative for any symmetric matrix S, implying that there exists $\wh{V}\in\partial\l\|\wh{S}_\lambda\r\|_1$ such that
\begin{align*}
-\frac{2}{N} \sum_{i\neq j} \dotp{\wt{Y}_{i,j}\wt{Y}_{i,j}^T-\wh{S}_\lambda-\sqrt{N}\wh{U}_{i,j}}{S-\wh{S}_\lambda} + \lambda_1 \dotp{\wh{V}} {S-\wh{S}_\lambda} \geq 0.
\end{align*}
Proceeding as before, we see that there exists $V\in\partial\l\|S\r\|_1$ such that 
\begin{multline*}
\frac{2}{N}\sum_{i\neq j}\dotp{\Sigma-\wh{S}_\lambda}{S-\wh{S}_\lambda} \leq \lambda_1\dotp{V}{S-\wh{S}_\lambda} + \frac{2}{N}\sum_{i\neq j}\dotp{\wt{X}_{i,j}-\Sigma}{\wh{S}_\lambda-S} \\
+\frac{2}{\sqrt{N}} \sum_{i\neq j}\dotp{\wt{U}_{i,j}^*-\wh{U}_{i,j}}{\wh{S}_\lambda-S}.
\end{multline*}
Combining (\ref{bound-1}, \ref{bound-3}) with the inequality above and applying the law of cosines, we deduce that
%==bd-1==%
\begin{multline}
\label{bd-1}
\quad\l\|\Sigma-\wh{S}_\lambda\r\|_F^2 + \l\|S-\wh{S}_\lambda\r\|_F^2 + (\lambda_1-2\l\|\Delta\r\|) \l\|P_{L^\perp}\wh{S}_\lambda P_{L^\perp}\r\|_1 
\\
\leq \frac{2}{\sqrt{N}} \sum_{i\neq j}\dotp{\wt{U}_{i,j}^*-\wh{U}_{i,j}}{\wh{S}_\lambda-S}  +  \l\|\Sigma-S\r\|_F^2
\\
+ \l(2\sqrt{2\rank(S)} \l\|\Delta\r\| + \lambda_1 \sqrt{\rank(S)} \r)\l\|\wh{S}_\lambda-S\r\|_F,
\end{multline}
where, as before, $\Delta = \frac{1}{N}\sum_{i\neq j} \wt{X}_{i,j} - \Sigma$. On the event $\m E$, we have that $\lambda_1\geq 2\l\|\Delta\r\|$, so using the inequality $2ab\leq a^2+b^2$, we get that
%==bd-2==%
\begin{align}\label{bd-2}
&\quad\l(2\sqrt{2\rank(S)} \l\|\Delta\r\| + \lambda_1 \sqrt{\rank(S)} \r)\l\|\wh{S}_\lambda-S\r\|_F \nonumber\\
&=2\sqrt{2}\l(2\sqrt{2\rank(S)} \l\|\Delta\r\| + \lambda_1 \sqrt{\rank(S)} \r)\l( \frac{\l\|\wh{S}_\lambda-S\r\|_F}{2\sqrt{2}} \r)\nonumber\\
&\leq \frac{1}{8}\l\|\wh{S}_\lambda - S\r\|_F^2 + 2\lambda_1^2 \rank(S) (\sqrt{2}+1)^2.
\end{align}
On the other hand, we can repeat the reasoning in \eqref{3.3.2} and apply Lemma \ref{lemma 3.2} to deduce that
\begin{align}
&\quad\sum_{(i,j)\notin\wt{J}} \l\|\wh{U}_{i,j} - \wt{U}_{i,j}^*\r\|_1 = \sum_{\wt{J}}\l\|P_{L_{i,j}^\perp}\wh{U}_{i,j}P_{L_{i,j}^\perp}\r\|_1 \nonumber\\
&\leq 3\l( \frac{\lambda_1}{\lambda_2}  \l\|\m P_{L(k)} \l(\wh{S}_\lambda - \Sigma(k)\r) \r\|_1 + \sum_{\wt{J}}\l\|\m P_{L_{i,j}}\l(\wh{U}_{i,j} - \wt{U}_{i,j}^* \r) \r\|_1 \r) \nonumber\\
&\leq 3\l( \frac{\lambda_1}{\lambda_2} \cdot \sqrt{2k}\l\|\wh{S}_\lambda-\Sigma\r\|_F + 2\sum_{\wt{J}} \l\|\wh{U}_{i,j}-\wt{U}_{i,j}^* \r\|_F    \r).\nonumber
\end{align}
%==bd-3==%
Therefore, %under the assumption that $\lambda_1 \leq \lambda_2$, we have that 
\begin{align}\label{bd-3}
&\quad\frac{2}{\sqrt{N}} \sum_{i\neq j}\dotp{\wt{U}_{i,j}^*-\wh{U}_{i,j}}{\wh{S}_\lambda-S} \nonumber\\
&\leq \frac{2\l\|S-\wh{S}_\lambda\r\|_F}{\sqrt{N}} \cdot \sum_{(i,j)\in\wt{J}} \l\|\wt{U}_{i,j}^*-\wh{U}_{i,j}\r\|_F + \frac{2\l\|S-\wh{S}_\lambda\r\|}{\sqrt{N}} \cdot \sum_{(i,j)\notin\wt{J}} \l\|\wt{U}_{i,j}^*-\wh{U}_{i,j}\r\|_1 \nonumber\\
&\leq\frac{2\l\|S-\wh{S}_\lambda\r\|_F}{\sqrt{N}}\l( 7\sum_{\wt{J}} \l\|\wt{U}_{i,j}^*-\wh{U}_{i,j}\r\|_F +  \frac{\lambda_1}{\lambda_2} \cdot 3\sqrt{2k} \l\|\wh{S}_\lambda - \Sigma \r\|_F\r) \nonumber\\
&\leq \frac{\lambda_1}{\lambda_2} \cdot 3\sqrt{2}\sqrt{\frac{k}{N}} \l( \l\|S-\wh{S}_\lambda\r\|_F^2 + \l\|\wh{S}_\lambda-\Sigma\r\|_F^2 \r) + 14\l\|S-\wh{S}_\lambda\r\|_F \sqrt{\frac{|\wt{J}|}{N}}\sqrt{\sum_{\wt{J}}\l\| \wh{U}_{i,j}-\wt{U}_{i,j}^*\r\|_F^2 }.
\end{align}
To estimate $\sqrt{\sum_{\wt{J}}\l\| \wh{U}_{i,j}-\wt{U}_{i,j}^*\r\|_F^2 }$, we apply the inequality \eqref{bound_case_1} which entails that
\begin{multline*}
\sqrt{\sum_{\wt{J}}\l\| \wh{U}_{i,j}-\wt{U}_{i,j}^*\r\|_F^2 } \leq 2\sqrt{2} \Bigg(\l\|\Sigma-S\r\|_F + \sqrt{\rank(S)}\lambda_1(\sqrt{2}+1) 
+  \lambda_2 \frac{(4/3+\sqrt{2})}{\sqrt{2}}\sqrt{|\wt{J}|} \Bigg),
\end{multline*}
given that $k=\lfloor\frac{N\lambda_2^2}{1200\lambda_1^2}\rfloor$, $\rank(S) \leq \frac{1}{56000} \cdot \frac{n^2\lambda_2^2}{\lambda_1^2}$, $|\wt{J}|\leq\frac{N}{6400}$.
Therefore, by applying the bound $2ab\leq a^2 + b^2$ several times, we deduce that
\begin{align*}
&\quad14\l\|S-\wh{S}_\lambda\r\|_F \sqrt{\frac{|\wt{J}|}{N}}\sqrt{\sum_{\wt{J}}\l\| \wh{U}_{i,j}-\wt{U}_{i,j}^*\r\|_F^2 } \\
&\leq 2\cdot 14\sqrt{2}\sqrt{\frac{|\wt{J}|}{N}}\l\|S-\wh{S}_\lambda\r\|_F\cdot\l\|\Sigma-S\r\|_F +   2\cdot \frac{\l\|S-\wh{S}_\lambda\r\|_F}{2} \cdot 28(4/3+\sqrt{2})\lambda_2\sqrt{\frac{|\wt{J}|^2}{N}} \\
& \hspace{30mm}+ 2\cdot 2\sqrt{2}(\sqrt{2}+1)\sqrt{\frac{|\wt{J}|}{N}}\l\|S-\wh{S}_\lambda\r\|_F \cdot7\lambda_1\sqrt{\rank(S)} \\
& \leq 14\sqrt{2}\sqrt{\frac{|\wt{J}|}{N}}\l( \l\|S-\wh{S}_\lambda\r\|_F^2 + \l\|\Sigma-S\r\|_F^2 \r) + \l(8(\sqrt{2}+1)^2\frac{|\wt{J}|}{N} + \frac{1}{4}\r)\l\|S-\wh{S}_\lambda\r\|_F^2\\
&\hspace{30mm}+ 49\lambda_1^2\rank(S) + \l(28(4/3+\sqrt{2})\r)^2 \frac{|\wt{J}|^2}{N} \lambda_2^2.
\end{align*}
Combining this with (\ref{bd-1},\ref{bd-2},\ref{bd-3}), we obtain that 
\begin{multline}\label{bd-large coeff}
\l\|\Sigma-\wh{S}_\lambda\r\|_F^2 \leq \frac{11}{5} \l\|\Sigma-S\r\|_F^2 + \frac{8}{5} \l(2(\sqrt{2}+1)^2+49\r)\lambda_1^2\rank(S)\\
+ \frac{8}{5}  \l(28(4/3+\sqrt{2})\r)^2 \frac{|\wt{J}|^2}{N} \lambda_2^2
\end{multline}
under the assumptions that $\rank(S)\leq \frac{1}{56000} \cdot \frac{n^2\lambda_2^2}{\lambda_1^2}$, $3\sqrt{2} \frac{\lambda_1}{\lambda_2} \sqrt{\frac{k}{N}}+14\sqrt{2}\sqrt{\frac{|\wt{J}|}{N}}\leq \frac{3}{8}$ and $8(\sqrt{2}+1)^2\frac{|\wt{J}|}{N} +\frac{1}{4}\leq \frac{1}{2}$. The assumptions hold for $\rank(S)\leq \frac{1}{56000} \cdot \frac{n^2\lambda_2^2}{\lambda_1^2}$, $k=\lfloor\frac{N\lambda_2^2}{1200\lambda_1^2}\rfloor$ and $|\wt{J}|\leq\frac{N}{6400}$. Note that the coefficient $\frac{11}{5}$ can be made smaller. Given $\delta\in(0,\frac{3}{8})$, we assume that $3\sqrt{2} \frac{\lambda_1}{\lambda_2} \sqrt{\frac{k}{N}}+14\sqrt{2}\sqrt{\frac{|\wt{J}|}{N}}\leq \delta$, which holds with the choices of $k\leq \frac{N\lambda_2^2}{72\lambda_1^2}\cdot\delta^2$ and $|\wt{J}| \leq \frac{N}{1568}\cdot\delta^2$ respectively. Also, we assume that $\rank(S) \leq c_1 \frac{n^2\lambda_2^2}{\lambda_1^2} $ for some constant $c_1\leq\frac{1}{5980}$ according to \eqref{condition on rank of S}.
Then \eqref{bd-large coeff} becomes
\begin{multline}\label{bd-coeff close to one}
\l\|\Sigma-\wh{S}_\lambda\r\|_F^2 \leq \frac{1+\delta}{1-\delta} \l\|\Sigma-S\r\|_F^2 + \frac{1}{1-\delta} \l(2(\sqrt{2}+1)^2+49\r)\lambda_1^2\rank(S)\\
+ \frac{1}{1-\delta}  \l(28(4/3+\sqrt{2})\r)^2 \frac{|\wt{J}|^2}{N} \lambda_2^2 ,
\end{multline}
where $\frac{1+\delta}{1-\delta} \in (1,\frac{11}{5}]$ is a number close to $1$. Finally, by \eqref{bound_J}, we see that $\frac{|\wt{J}|}{N} \leq 2\frac{|J|}{n}$, so we can write the last term of the inequality \eqref{bd-coeff close to one} as
\begin{multline*}
\l(28(4/3+\sqrt{2})\r)^2 \frac{|\wt{J}|^2}{N} \lambda_2^2 =\l(28(4/3+\sqrt{2})\r)^2 \lambda_2^2 |J|^2 \frac{(2n-|J|-1)^2}{n(n-1)} \\
\leq 4\l(28(4/3+\sqrt{2})\r)^2 \lambda_2^2 |J|^2
\end{multline*}
under the assumption that  $|J|\leq\frac{n\delta^2}{3136}$.
This completes the proof.
\endgroup
%\end{proof}

%#########################################
\subsection{Proof of Theorem \ref{thm:mean_bound}} 
\label{section: proof of thm:mean_bound}
%#########################################
In this section we prove Theorem \ref{thm:mean_bound}, which provides the lower bound for the choice of $\lambda_1$. We start with a well-known theorem on the concentration of sample covariance matrix.
%===Kolchinskii2017 covariance concentration====%
\begin{theorem}[{\citet[Theorem 9]{koltchinskii2017concentration}}]
\label{thm:sample covariance bound}
Assume that $Z$ is L-sub-Gaussian with mean zero and sample covariance matrix $\Sigma$. Let $Z_1,\ldots,Z_n$ be independent samples of $Z$, then there exists $c(L)>0$ depending only on L, such that
\begin{equation*}
\l\| \frac{1}{n}\sum_{j=1}^{n}Z_jZ_j^T - \Sigma\r\| \leq c(L) \l\|\Sigma\r\| \l( \sqrt{\frac{\rk(\Sigma)}{n}} \vee \frac{\rk(\Sigma)}{n}  \vee \sqrt{\frac{t}{n}} \vee \frac{t}{n}\r)
\end{equation*} 
with probability at least $1-e^{-t}$.
\end{theorem}
\begin{remark}
Assuming that $\rk(\Sigma) \leq n$ and $t \leq n$, the bound can be reduced to 
\[
c(L) \l\|\Sigma\r\| \l( \sqrt{\frac{\rk(\Sigma)}{n}}  \vee \sqrt{\frac{t}{n}} \r).
\]
\end{remark}
Now we prove Theorem \ref{thm:mean_bound}.
\begin{proof}[Proof of Theorem \ref{thm:mean_bound}]
\begingroup
\allowdisplaybreaks
First, it is well-known that  
\begin{align}\label{meanthm_1}
\frac{1}{n(n-1)}\sum_{i\neq j}\frac{(Z_i - Z_j)(Z_i - Z_j)^T}{2} = \frac{1}{n-1}\sum_{i=1}^{n}(Z_i-\bar{Z})(Z_i-\bar{Z})^T,
\end{align}
where $\bar{Z}:=\frac{1}{n}\sum_{i=1}^{n}Z_i$. 
Therefore,
\begin{equation*}
\Delta := \frac{1}{n(n-1)}\sum_{i\neq j} \wt{Z}_{ij} \wt{Z}_{ij}^T - \Sigma 
= \frac{1}{n(n-1)}\sum_{i\neq j}\frac{(Z_i - Z_j)(Z_i - Z_j)^T}{2} - \Sigma
= \wt{\Sigma}_s - \Sigma.
\end{equation*}
Recall $\expect{Z_j}=\mu, j=1,\ldots,n$, and note that 
\[
\wt{\Sigma}_s = \frac{1}{n-1}\l( \sum_{i=1}^{n}(Z_i-\mu)(Z_i-\mu)^T - n(\bar Z - \mu)(\bar Z - \mu)^T \r),
\]
hence we have the decomposition
\begin{align} \label{5.3.1}
(n-1)\l\| \Delta \r\| &= \l\| (n-1)\wt{\Sigma}_s - n\Sigma + \Sigma \r\|  \nonumber\\
& \leq \l\| \sum_{i=1}^{n}(Z_i-\mu)(Z_i-\mu)^T - n\Sigma \r\| + \l\| \Sigma - n(\bar Z - \mu)(\bar Z - \mu)^T  \r\|.
\end{align}
We will bound the two terms on the right-hand side of \eqref{5.3.1} one by one. First, note that $Z_j - \mu, j=1,\ldots,n$ are i.i.d L-sub-Gaussian random vectors with mean zero and covariance matrix $\Sigma$, hence Theorem \ref{thm:sample covariance bound} immediately gives that
\begin{align} \label{5.3.2}
\l\| \sum_{i=1}^{n}(Z_i-\mu)(Z_i-\mu)^T - n\Sigma \r\| \leq nc(L) \l\|\Sigma\r\|  \l( \sqrt{\frac{\rk(\Sigma)}{n}} \vee \frac{\rk(\Sigma)}{n}  \vee \sqrt{\frac{t}{n}} \vee \frac{t}{n}\r)
\end{align}
with probability at least $1-e^{-t}$. To bound the second term, consider the random variable $Y:=\sqrt{n} \l( \bar{Z} - \mu \r) = \frac{1}{\sqrt{n}}\sum_{i=1}^{n}(Z_i-\mu)$. Clearly, $\expect{Y}=0$ and
\begin{align*}
\expect{YY^T}  = n\expect{(\bar Z -\mu)(\bar Z - \mu)^T} 
&= \frac{1}{n} \expect{\sum_{i=1}^{n}\sum_{j=1}^{n}(Z_i-\mu)(Z_j-\mu)^T } \\
& = \frac{1}{n} \expect{\sum_{i=1}^{n}(Z_i-\mu)(Z_i-\mu)^T} = \Sigma,
\end{align*}
where we used the independence of $Z_i, i=1,\ldots,n$ in the third equality.  Moreover, Corollary \ref{sum of sub-gaussians} guarantees that $Y$ is L-sub-Gaussian. Therefore, $Y$ satisfies the conditions in Theorem \ref{thm:sample covariance bound}, and a direct application of the theorem implies that
\begin{equation}\label{5.3.3}
\l\| \Sigma - n(\bar Z - \mu)(\bar Z - \mu)^T  \r\| = \l\| YY^T - \Sigma \r\| \leq c(L)\l\|\Sigma\r\| \l( \rk(\Sigma) \vee t \r)
\end{equation}
with probability at least $1-e^{-t}$, given that $t \geq 1$. Combining (\ref{5.3.1}, \ref{5.3.2}, \ref{5.3.3}), we deduce that for any $t\geq 1$,
\begin{multline*}
\l\|\Delta\r\| \leq c(L) \Bigg[ \frac{n}{n-1} \l\|\Sigma\r\| \l( \sqrt{\frac{\rk(\Sigma)}{n}} \vee \frac{\rk(\Sigma)}{n}  \vee \sqrt{\frac{t}{n}} \vee \frac{t}{n}\r) 
+ \frac{1}{n-1}\l\|\Sigma\r\| \l( \rk(\Sigma) \vee t \r) \Bigg] \\
\leq c(L) \l\|\Sigma\r\| \Bigg[  \sqrt{\frac{\rk(\Sigma)+t}{n}}  + \frac{\rk(\Sigma) + t }{n} \Bigg]
\end{multline*}
with probability at least $1-2e^{-t}$, where $c(L)$ is an absolute constant that only depends on $L$ but could vary from step to step. 

\endgroup
\end{proof}

%###########################
\subsection{Proof of Theorem \ref{thm:max covariance bound}}
\label{section: proof of thm:max_bound}
%###########################
In this section we prove Theorem \ref{thm:max covariance bound}, which provides the lower bound for the choice of $\lambda_2$. 
\begin{proof}[Proof of Theorem \ref{thm:max covariance bound}]
\begingroup
\allowdisplaybreaks
Fix $i\in\{1,\ldots,n\}$, we apply Theorem \ref{thm:sample covariance bound} to $Z_i$ and deduce that for any $u\geq1$,
\begin{align*}
\l\|Z_iZ_i^T-\Sigma\r\|&\leq c(L)\l\|\Sigma\r\| \l(\rk(\Sigma)+u\r) =c(L)\l(\tr(\Sigma)+\l\|\Sigma\r\|u \r)
\end{align*}
with probability at least $1-e^{-u}$.
\newline Therefore, by union bound we have that for $t\geq1$ and $n\geq1$,
\begin{align*}
&P\l(\max_{i}{\l\|Z_iZ_i^T-\Sigma\r\|}\geq c(L)\l[tr(\Sigma)+\log(n)\l\|\Sigma\r\|+\l\|\Sigma\r\|t \r]\r) \\
&\leq \sum_{i=1}^{n}P\l(\l\|Z_iZ_i^T-\Sigma\r\|\geq c(L) \l[tr(\Sigma)+\log(n)\l\|\Sigma\r\|+\l\|\Sigma\r\|t \r]\r) \\
&= nP\l(\l\|Z_iZ_i^T-\Sigma\r\|\geq c(L)\l[tr(\Sigma)+\l\|\Sigma\r\|(\log(n)+t) \r]\r) \\
&\leq ne^{-log(n)-t} =e^{-t}.
\end{align*}
In other words, for $t\geq1$, we have that with probability at least $1-e^{-t}$,
\begin{align*}
\max_{i}{\l\|Z_iZ_i^T-\Sigma\r\|}&\leq c(L) \l[ \tr(\Sigma)+\l\|\Sigma\r\|(\log(n)+t) \r] \\
&=c(L) \l\|\Sigma\r\| \l(\rk(\Sigma) + \log(n) + t\r),
\end{align*}
as desired.
\endgroup
\end{proof}

%=======================================
% Section: Proof of results in Section 2.4
%=======================================
\subsection{Proof of Lemma \ref{lemma:large tuning} and Theorem \ref{thm:heavy-tailed operator bound}}
\label{sec:proof or heavy-tailed operator bound}
%=========================================================%
In this subsection we present the proof of Lemma \ref{lemma:large tuning} and Theorem \ref{thm:heavy-tailed operator bound}, which provide error bounds of the estimator in \eqref{heavy-tailed estimator} in the operator norm. To simplify the expressions, we introduce the following notations, which are valid in this subsection only:
\begin{itemize}
	\item Denote
	\[
	\rho(u) := \rho_1(u) = \l\{
			 	   \begin{array}{ll}
			    		\frac{u^2}{2},\quad\l|u\r|\leq1\\
			    		\l|u\r|-\frac{1}{2},\quad\l|u\r|>1
			  	  \end{array}
			  	  \r. \quad \forall u\in \mb R.
	\]
	\item Denote $k_0 = \lfloor n/2 \rfloor$ and $N=n(n-1)$.
	\item Denote $\theta_2 := \frac{2}{\lambda_2\sqrt{n(n-1)}}$ and 
	\[
	\theta_\sigma := \frac{1}{\sigma}\sqrt{\frac{2t}{k_0}}
	\]
	where $\sigma>0$, $t>0$ are constants to be specified later.
\end{itemize}
It is easy to check that $\rho_\lambda(u)=\lambda^2\rho({u}/{\lambda})$ and $\rho_\lambda'(u)=\lambda\rho({u}/{\lambda})$, so with the above notations, we can rewrite the loss function in \eqref{heavy-tailed loss} as
\begin{equation}
L(S) = \frac{1}{\theta_2^2}\frac{1}{N}\sum_{i\neq j} \rho\l( \theta_2(H_{i,j}-S) \r) + \frac{\lambda_1}{2}\l\|S\r\|_1.
\end{equation}
The gradient of the loss function is 
\begin{equation}
\nabla L(S) = -\frac{1}{N\theta_2} \sum_{i\neq j} \rho' \l( \theta_2(H_{i,j} -S) \r) + \frac{\lambda_1}{2}\partial\l\|S\r\|_1.
\end{equation}
Given $\alpha\in(0,1)$, one can easily verify that $\rho'(\cdot)$ is H\"older continuous on $\mb{R}$, namely, $|\rho'(x)-\rho'(y)|\leq 2|x-y|^\alpha,\forall x,y\in\mb R$. %(see Appendix \ref{sec:holder continuity} for details). 
The following theorem shows that $\rho'(\cdot)$ is H\"older continuous in the operator norm, which is crucial for the next part of the proof:
\begin{theorem}
\label{thm:operator holder} 
(\citet[Theorem 1.7.2]{aleksandrov2016operator}) Assume that $f(x)$ is H\"older continuous on $\mb{R}$ with $\alpha \in(0,1)$, i.e. $|f(x)-f(y)|\leq C_0|x-y|^\alpha$, $\forall x,y \in\mb{R}$. Then there exists an absolute constant $c$ such that
\[
\l\|f(A)-f(B)\r\|\leq c(1-\alpha)^{-1}C_0\l\|A-B\r\|^\alpha
\]
for any symmetric matrices $A$ and $B$. 
\end{theorem}
We now present the proofs of Lemma \ref{lemma:large tuning} and Theorem \ref{thm:heavy-tailed operator bound}. It is worth noting that the proof follows the argument in \citet[Section 5]{minsker2020robust}.

%###########################
% \subsubsection{Proof of Lemma \ref{lemma:large tuning}}
%###########################

\begin{proof}[Proof of Lemma \ref{lemma:large tuning}]
\begingroup
\allowdisplaybreaks
Recall the loss function and its gradient:
\begin{equation*}
L(S) = \frac{1}{\theta_2^2}\frac{1}{N}\sum_{i\neq j}\rho(\theta_2(H_{i,j}-S)) + \frac{\lambda_1}{2}\l\|S\r\|_1,
\end{equation*}
\begin{equation*}
\nabla L(S) = -\frac{1}{N\theta_2}\sum_{i\neq j}\rho'(\theta_2(H_{i,j}-S)) + \frac{\lambda_1}{2}\partial\l\|S\r\|_1,
\end{equation*}
where $H_{i,j} = \wt{Y}_{i,j}\wt{Y}_{i,j}^T$. Consider the choice $\lambda_1 > ({2\theta_\sigma})^{-1}$, where $\theta_\sigma:= \sigma^{-1} \sqrt{{2t}/{k_0}}$. We assume that the minimizer $\wt{S}=V\neq0$. Since $L(S)$ is convex, we have
\[
L(V) - L(0) \geq \dotp{\nabla L(0)}{V-0}.
\]
Plugging in the explicit form of $\nabla L(0)$, we get that for any $W\in \partial { \l\|S\r\|_1} _{|_{S=0}}$,
\begin{equation*}
L(V) - L(0) \geq \dotp{-\frac{1}{N\theta_2}\sum_{i\neq j}\rho'(\theta_2 H_{i,j})+\frac{\lambda_1}{2}W}{V},
\end{equation*}
hence
\begin{equation}
L(V) - L(0) \geq \sup_{W\in \partial \l\|S\r\|_1|_{S=0}}\dotp{-\frac{1}{N\theta_2}\sum_{i\neq j}\rho'(\theta_2 H_{i,j})+\frac{\lambda_1}{2}W}{V}.
\end{equation}
Consider the random variable $\mathcal X_{i,j}:=\mathds{1}\{\l\|H_{i,j}-\Sigma\r\| \leq \frac{1}{a\theta_\sigma}\}$, $a\geq 2$, and set $\theta_2 = \theta_\sigma$ in what follows. By Chebyshev's inequality,
\begin{equation*}
P(\mathcal{X}_{i,j}=0)\leq a^2\theta_\sigma^2 \tr \expect{(H_{i,j}-\Sigma)^2} \leq a^2\frac{2t}{k_0} r_H, 
\end{equation*}
where $r_H = \rk(\expect{(H_{i,j}-\Sigma)^2})$. Define the event
\[
\mathcal{E}:= \Big\{ \frac{1}{N}\sum_{i\neq j}(1-\mathcal{X}_{i,j}) \leq r_H \frac{2a^2t}{k_0}(1+\sqrt{\frac{3}{2a^2r_H}}) \Big\}
\]
By the finite difference inequality (see for example, \citet[Fact 4-6]{minsker2020robust}), 
\[
P\Big(\frac{1}{N}\sum_{i\neq j}(1-\mathcal{X}_{i,j})\geq r_H\frac{2a^2t}{k_0}(1+\tau)\Big) \leq e^{-\tau^2 2a^2tr_H/3},\quad 0<\tau<1.
\]
Setting $\tau=\sqrt{{3}/{(2a^2r_H)}}$ we get $P(\mathcal{E})\geq 1-e^{-t}$.
Therefore, for $k_0\geq 32a^2 tr_H$, we have that 
\[
\frac{1}{N}\sum_{i\neq j}(1-\mathcal{X}_{i,j})\leq\frac{1}{8}
\]
with probability $\geq1-e^{-t}$.
Note that on the event $\{\mathcal{X}_{i,j}=0\}$,
\[
\l\|\rho'(\theta_\sigma H_{i,j})\r\| \leq 1
\]
since $|\rho'(x)|\leq 1$ for any $x\in\mb{R}$. On the other hand, on the event $\{\mathcal{X}_{i,j}=1\}$, we have that $\l\|H_{i,j}-\Sigma\r\| \leq \frac{1}{a\theta_\sigma}$, hence
\[
\l\|H_{i,j}\r\| \leq \frac{1}{a\theta_\sigma} + \l\|\Sigma\r\| \leq \frac{1}{a\theta_\sigma} +\frac{1}{b\theta_\sigma}
\]
given that $k_0\geq {2b^2t^2\l\|\Sigma\r\|^2}/{\sigma^2}$. Therefore, by Theorem \ref{thm:operator holder} we have that
\begin{equation*}
\l\|\rho'(\theta_\sigma H_{i,j}) \r\| \leq 2c(1-\alpha)^{-1}\l\|\theta_\sigma H_{i,j}\r\|^\alpha \leq 2c(1-\alpha)^{-1} \l(\frac{1}{a} + \frac{1}{b} \r)^\alpha.
\end{equation*}
Setting $a,b$ large enough such that $2c(1-\alpha)^{-1} \l( \frac{1}{a} + \frac{1}{b} \r)^\alpha +\frac{1}{8} \leq \frac{1}{4}$, we have that
\begin{multline*}
\l\|\frac{1}{\theta_\sigma N}\sum_{i\neq j}\rho'(\theta_\sigma H_{i,j}) \r\| \leq \l\|\frac{1}{\theta_\sigma N}\sum_{i\neq j}\rho'(\theta_\sigma H_{i,j}) \mathcal{X}_{i,j}\r\| + \l\|\frac{1}{\theta_\sigma N}\sum_{i\neq j}\rho'(\theta_\sigma H_{i,j}) (1-\mathcal{X}_{i,j})\r\|\\
\leq \frac{1}{\theta_\sigma} 2c(1-\alpha)^{-1}(\frac{1}{a}+\frac{1}{b})^\alpha\frac{1}{N}\sum_{i\neq j} \mathcal{X}_{i,j} + \frac{1}{8\theta_\sigma} \leq \frac{1}{4}\frac{1}{\theta_\sigma}.
\end{multline*}
Therefore,
\begin{multline*}
L(V) - L(0) \geq \sup_{W\in \partial \l\|S\r\|_1|_{S=0}}\dotp{-\frac{1}{N\theta_2}\sum_{i\neq j}\rho'(\theta_2 H_{i,j})+\frac{\lambda_1}{2}W}{V} \\
=   \dotp{-\frac{1}{N\theta_2}\sum_{i\neq j}\rho'(\theta_2 H_{i,j})}{V} + \frac{\lambda_1}{2}\l\|V\r\|_1
\geq -\frac{1}{4\theta_\sigma}\l\|V\r\|_1 + \frac{\lambda_1}{2}\l\|V\r\|_1  > 0,
\end{multline*}
where we used the fact that $\partial\l\|S\r\||_{S=0} = \{W:\l\|W\r\|\leq 1\}$ and $\sup_{W:\l\|W\r\|\leq1}\dotp{W}{V} = \l\|V\r\|_1$. This is a contradiction to the fact that $V$ is a minimizer of the loss function $L(S)$, and hence we conclude that $\argmin_{S}L(S)=0$ with probability at least $1-e^{-t}$.
\endgroup
\end{proof}

%###########################
%\subsubsection{Proof of Theorem \ref{thm:heavy-tailed operator bound}}
%###########################
We now present the proof of Theorem \ref{thm:heavy-tailed operator bound}. 
\begin{proof}[Proof of Theorem \ref{thm:heavy-tailed operator bound}]
\begingroup
\allowdisplaybreaks
Recall the loss function
\begin{equation*}
L(S) = \frac{1}{\theta_2^2}\frac{1}{N}\sum_{i\neq j}\rho\big(\theta_2(H_{i,j}-S)\big)+\frac{\lambda_1}{2}\l\|S\r\|_1
\end{equation*}
and its gradient
\begin{equation*}
\nabla L(S) = -\frac{1}{N\theta_2}\sum_{i\neq j} \rho'\big(\theta_2(H_{i,j}-S)\big)+\frac{\lambda_1}{2}\partial \l\|S\r\|_1,
\end{equation*}
where $H_{i,j} = \wt{Y}_{i,j}\wt{Y}_{i,j}^T$, $\rho(\cdot) = \rho_1(\cdot)$ and $\theta_2 = {2}/{(\sqrt{N}\lambda_2)}$.
Consider the proximal gradient descent iteration:
\begin{enumerate}
\item $S^{0} := \expect{H} = \Sigma$.
\item For $t=1,2,\ldots$, do:
	\begin{itemize}
	\item $T^{t+1} := S^{t} + \frac{1}{N\theta_2}\sum_{i\neq j}\rho' \l(\theta_2(H_{i,j}-S^t)\r)$.
	\item $S^{t+1} := \argmin_S\{\frac{1}{2}\l\|S-T^{t+1}\r\|_F^2 +\frac{\lambda_1}{2}\l\|S\r\|_1\}$.
	\end{itemize}
\end{enumerate}
We will show that with an appropriate choice of $\theta_2$, $S^{t+1}$ does not escape a small neighborhood of $\Sigma$ with high probability, and the result will easily follow. First, the following lemma bounds $\l\|S^{t+1}-T^{t+1}\r\|$:
\begin{lemma}\label{lem:diff between S and T}
\[
\l\|S^{t+1}-T^{t+1}\r\| \leq \frac{\lambda_1}{2}.
\]
\end{lemma}
\begin{proof}
Repeating the reasoning for equation \eqref{appendix E.5.1:eq_1} in section \ref{sec:derivation of penalized huber} of the supplementary material, we can solve for $S^{t+1}$ explicitly:
\[
S^{t+1} = \argmin_S\l\{ \frac{1}{2}\l\|S-T^{t+1}\r\|_F^2 + \frac{\lambda_1}{2}\l\|S\r\|_1\r\} = \gamma_{\frac{\lambda_1}{2}}(T^{t+1}),
\]
where $\gamma_\lambda(u) = \sign(u)(|u| -\lambda)_+$ is the function that shrinks eigenvalues to 0. A direct calculation gives that 
\[
\l\|S^{t+1}-T^{t+1}\r\| = \l\|S^{t+1} -T^{t+1} \r\| = \l\| \gamma_{\frac{\lambda_1}{2}}(T^{t+1}) - T^{t+1}\r\| \leq \frac{\lambda_1}{2}.
\]
\end{proof}
Applying Lemma \ref{lem:diff between S and T} , we see that
\[
\l\|S^{t+1}-\Sigma\r\| \leq \l\|S^{t+1}-T^{t+1}\r\| + \l\|T^{t+1}-\Sigma\r\| \leq \frac{\lambda_1}{2} +  \l\|T^{t+1}-\Sigma\r\|.
\]
It remains to bound $\l\|T^{t+1}-\Sigma\r\|$. Note that
\begin{multline}
 \l\|T^{t+1}-\Sigma\r\| = \l\|S^t - \Sigma + \frac{1}{N\theta_2}\sum_{i\neq j}\rho' \l(\theta_2(H_{i,j}-S^t)\r) \r\|\\
 \leq \underbrace{\l\|\frac{1}{N\theta_2}\sum_{i\neq j}\l[\rho' \l(\theta_2(H_{i,j}-S^t)\r)-\rho'(\theta_2(H_{i,j}-\Sigma))\r]  + S^t - \Sigma \r\|}_{:=\mathrm{I}} 
 \\
 + \underbrace{\l\| \frac{1}{N\theta_2}\sum_{i\neq j}\rho'\l(\theta_2(H_{i,j}-\Sigma)\r)\r\|}_{:=\mathrm{II}}.
\end{multline}
We will bound terms $\mathrm{I}$ and $\mathrm{II}$ separately. Set $k_0 = \lfloor n/2 \rfloor$ and define
\begin{align*}
&Y_{i,j}(S;\theta) := \rho'(\theta(H_{i,j}-S)), \\
&W_{i_1,\ldots,i_n}(S;\theta) := \frac{1}{k_0}\l[Y_{i_1,i_2}(S;\theta) + \cdots + Y_{i_{2k_0-1},i_{2k_0}}(S;\theta) \r],
\end{align*}
where $(i_1,\ldots,i_n)\in\pi_n$ is a permutation. Fact 6 in \citet{minsker2020robust} implies that
\begin{align}
&\mathrm{I} = \l\| \frac{1}{\theta_2n!} \sum_{\pi_n}\l( W_{i_1,\ldots,i_n}(S^t;\theta_2) - W_{i_1,\ldots,i_n}(\Sigma;\theta_2) \r) + S^t - \Sigma \r\| , \\
&\mathrm{II} = \l\|\frac{1}{\theta_2n!}\sum_{\pi_n}W_{i_1,\ldots,i_n}(\Sigma;\theta_2)\r\|.
\end{align}
For a given $\sigma^2\geq\l\|\expect{(H_{i,j}-\Sigma)^2}\r\|$ and $\theta_\sigma := {\sigma}^{-1} \sqrt{{2t}/{k_0}}$, the following lemma provides a bound for the term $\mathrm{II}$:
\begin{lemma}\label{lemma:8.5}
Recall that $r_H = \rk(\expect{(H_{i,j}-\Sigma)^2})$. Given $t\geq 1$, we have that 
\[
\l\|\frac{1}{\theta_2n!}\sum_{\pi_n}W_{i_1,\ldots,i_n}(\Sigma;\theta_2)\r\| \leq \theta_2\sigma^2 + \frac{t}{\theta_2k_0}
\]
with probability at least $1-\frac{8}{3}r_He^{-t}$. When $\theta_2 = \theta_\sigma$, the upper bound takes the form $3\sigma\sqrt{{t}/{(2k_0)}}$.
\end{lemma}
\begin{proof}
It is easy to verify that for any $x\in\mb{R}$,
\[
-\log(1-x+x^2) \leq \rho'(x) \leq \log(1+x+x^2),
\]
and the rest of the proof follows from the argument in \citet[Section 5.5]{minsker2020robust}.
\end{proof}
To estimate the term $\mathrm{I}$, consider the random variable
\begin{equation*}
L_n(\delta) := \sup_{\l\|S-\Sigma\r\|\leq\delta} \l\| \frac{1}{\theta_\sigma n!} \sum_{\pi_n}\l( W_{i_1,\ldots,i_n}(S;\theta_\sigma) - W_{i_1,\ldots,i_n}(\Sigma;\theta_\sigma) \r) + S - \Sigma \r\|. 
\end{equation*}
\begin{lemma}\label{lemma:8.6}
Given $\alpha \in (0,1)$, we have that for all $\delta\leq (2\theta_\sigma)^{-1}$,
\begin{equation*}
L_n(\delta) \leq r_H\frac{13t}{k_0}\delta\l( 1+4c(1-\alpha)^{-1} \r)
\end{equation*}
with probability at least $1-e^{-t}$, where $c>0$ is an absolute constant specified in Theorem \ref{thm:operator holder}.
\end{lemma}
\begin{proof}
Define $\mathcal{X}_{i,j} := \mathds{1} \big\{ \l\|H_{i,j}-\Sigma\r\| \leq {(2\theta_\sigma)^{-1}}\big\}$ and consider the event $\mathcal{E}:= \big\{ \sum_{i\neq j}(1-\mathcal{X}_{i,j}) \leq {8t}Nr_H (1+\sqrt{\frac{3}{8r_H}})/k_0 \big\}$. \cite{minsker2020robust} proves that $P(\mathcal{E}) \geq 1-e^{-t}$. For $S$ with $\l\|S-\Sigma\r\|\leq\delta\leq{(2\theta_\sigma)^{-1}}$, we have that
\begin{multline*}
\frac{1}{\theta_\sigma n!} \sum_{\pi_n}\l( W_{i_1,\ldots,i_n}(S;\theta_\sigma) - W_{i_1,\ldots,i_n}(\Sigma;\theta_\sigma) \r) + S - \Sigma \\
= \frac{1}{N\theta_\sigma}\sum_{i\neq j}\l[\rho' \l(\theta_\sigma(H_{i,j}-S)\r)-\rho'(\theta_\sigma(H_{i,j}-\Sigma))\r]  + S - \Sigma  \\
= \Big(\frac{1}{N\theta_\sigma}\sum_{i\neq j}\l[\rho' \l(\theta_\sigma(H_{i,j}-S)\r)-\rho'(\theta_\sigma(H_{i,j}-\Sigma))\r] \mathcal{X}_{i,j} + S-\Sigma\Big) \\
+ \Big(\frac{1}{N\theta_\sigma}\sum_{i\neq j}\l[\rho' \l(\theta_\sigma(H_{i,j}-S)\r)-\rho'(\theta_\sigma(H_{i,j}-\Sigma))\r] (1-\mathcal{X}_{i,j})\Big).
\end{multline*}
We will separately control the two terms on the right-hand side of the equality above. First, note that when $\mathcal{X}_{i,j}=1$, we have that $\l\|H_{i,j}-\Sigma\r\| \leq {(2\theta_\sigma)^{-1}} \leq {(\theta_\sigma)^{-1}}$, and $\l\|H_{i,j}-S\r\| \leq \l\|H_{i,j}-\Sigma\r\| + \l\|\Sigma-S\r\| \leq {(\theta_\sigma)^{-1}}$. Therefore, on the event $\mathcal{E}$, 
\begin{multline}\label{8.47}
\l\| \frac{1}{N\theta_\sigma}\sum_{i\neq j}\l[\rho' \l(\theta_\sigma(H_{i,j}-S)\r)-\rho'(\theta_\sigma(H_{i,j}-\Sigma))\r] \mathcal{X}_{i,j} + S-\Sigma\r\| \\
= \l\| \frac{1}{N\theta_\sigma}\sum_{i\neq j}\l[ \theta_\sigma(H_{i,j}-S) - \theta_\sigma(H_{i,j}-\Sigma) \r]\mathcal{X}_{i,j} + S-\Sigma\r\| \\
= \l\| \frac{1}{N}\sum_{i\neq j}(S-\Sigma)(1-\mathcal{X}_{i,j}) \r\| \leq r_H\frac{8t}{k_0}(1+\sqrt{\frac{3}{8r_H}})\delta\leq r_H\frac{13t}{k_0}\delta.
\end{multline}
Next, recall that for any $\alpha\in(0,1)$, $|\rho'(x)-\rho'(y)|\leq 2|x-y|^\alpha$ for any $x,y\in\mb R$, so by Theorem \ref{thm:operator holder}, there exists a constant $c>0$ such that
\[
\l\|\rho'(A)-\rho'(B)\r\| \leq 2c(1-\alpha)^{-1}\l\|A-B\r\|^\alpha
\]
for any symmetric matrices $A$ and $B$. Therefore, on the event $\mathcal{E}$,
\begin{multline}\label{8.48}
\l\| \frac{1}{N\theta_\sigma}\sum_{i\neq j}\l[\rho' \l(\theta_\sigma(H_{i,j}-S)\r)-\rho'(\theta_\sigma(H_{i,j}-\Sigma))\r] (1-\mathcal{X}_{i,j})\r\| \\
\leq 2c(1-\alpha)^{-1} \l\|\theta_\sigma(\Sigma-S)\r\|^\alpha\cdot\frac{1}{N\theta_\sigma}\sum_{i\neq j}(1-\mathcal{X}_{i,j})\\
\leq 2c(1-\alpha)^{-1}(\frac{1}{2})^{\alpha-1}\delta\cdot r_H\frac{13t}{k_0} \leq 4c(1-\alpha)^{-1}\cdot r_H\frac{13t}{k_0}\delta.
\end{multline}
Combining \eqref{8.47}, \eqref{8.48} and $P(\mathcal{E}) \geq 1-e^{-t}$, we have that 
\begin{equation*}
L_n(\delta) \leq r_H\frac{13t}{k_0}\delta\l( 1+4c(1-\alpha)^{-1} \r)
\end{equation*}
with probability at least $1-e^{-t}$.
\end{proof}
Now we can bound $\l\|S^{t+1}-\Sigma\r\|$ as follows:
\newline For $t=0,1,\ldots$, define
\begin{align*}
&\delta_0 = 0, \\
&\delta_{t+1} = r_H\frac{13t}{k_0}\l(1 + 4c(1-\alpha)^{-1} \r)\delta_{t} + 5.75\sigma\sqrt{\frac{t}{k_0}} + \frac{\lambda_1}{2}.
\end{align*}
Choose $t,k$ such that $13r_H {t}\l(1 + 4c(1-\alpha)^{-1} \r) \leq {k_0}/{20}$ and $t\leq {k_0}/{520}$, we have that $5.75\sigma\sqrt{{t}/{k_0}}\leq {(40\theta_\sigma)^{-1}}$, hence
\[
\delta_{t+1} \leq \frac{1}{20}\delta_t + \frac{1}{40\theta_\sigma} + \frac{\lambda_1}{2} \leq \frac{1}{2\theta_\sigma}
\] 
given that $\delta_t\leq{(2\theta_\sigma)^{-1}}$ and $\lambda_1\leq{(2\theta_\sigma)^{-1}}$. Since $\l\|S^0-\Sigma\r\|=0\leq{(2\theta_\sigma)^{-1}}$, we have that for $t=0,1,\ldots$,
\begin{multline*}
\l\|S^{t+1}-\Sigma\r\| \leq \frac{\lambda_1}{2} + L_n(\delta_t) + \frac{3}{\sqrt{2}}\sigma\sqrt{\frac{t}{k_0}} 
\\
\leq \frac{\lambda_1}{2} + r_H\frac{13t}{k_0}\l(1+4c(1-\alpha)^{-1}\r)\delta_t + \frac{3}{\sqrt{2}}\sigma\frac{t}{k_0} \leq \delta_{t+1}
\end{multline*}
with probability at least $1-(\frac{8}{3}r_H+1)e^{-t}$.
Finally, for $\gamma := 13r_H t \l(1+4c(1-\alpha)^{-1}\r) \leq k_0/40$, it is easy to check that for $t=0,1,\ldots$,
\begin{multline*}
\delta_{t+1} = \gamma^{t+1}\delta_0 + \sum_{l=0}^{t}\gamma^t(\frac{\lambda_1}{2}+\frac{3}{\sqrt{2}}\sigma\sqrt{\frac{t}{k_0}}) 
\\
\leq \sum_{l\geq0}\frac{1}{40^l}(\frac{\lambda_1}{2}+\frac{3}{\sqrt{2}}\sigma\sqrt{\frac{t}{k_0}}) = \frac{20}{39}\lambda_1 + \frac{20\sqrt{2}}{13}\sigma\sqrt{\frac{t}{k_0}}.
\end{multline*}
By Theorem \ref{pgd convergence}, $S^{t}\rightarrow \wh{S}$ pointwise as $t\rightarrow\infty$, so the result follows.

To this end, we note that the proof above can be repeated with $\theta_2 < \theta_\sigma :=\sigma^{-1} \sqrt{{2t}/{k_0}}$, in which case Lemma \ref{lemma:large tuning} will be valid for $$k_0\geq \max \l\{44a^2tr_H,\frac{2b^2t^2\l\|\Sigma\r\|^2}{\sigma^2} \r\}.$$ Moreover, the upper bounds in Lemma \ref{lemma:8.5} and Lemma \ref{lemma:8.6} will become
\[
\sigma \sqrt{ \frac{2t}{k_0} } + \frac{t}{\theta_2k_0} 
\]
and 
\[
8\theta_2^2 \sigma^2 r_H \l( 1+4c(1-\alpha)^{-1} \r) \delta
\]
respectively. Consequently, we can deduce that whenever
\[
8\theta_2^2\sigma^2 r_H \l( 1+2c(1-\alpha)^{-1} \r) \leq \frac{1}{20},
\]
which is valid as long as 
\[
r_H\frac{t}{k_0} \Big( 1+2c(1-\alpha)^{-1} \Big) \leq \frac{1}{640},
\]
the following inequality holds with probability at least $1-(\frac{8}{3}r_H+1)e^{-t}$:
\[
\l\|\wh{S}_\lambda-\Sigma\r\| \leq \frac{20}{39}\lambda_1 + \frac{40}{39} \l[ \sigma \sqrt{\frac{2t}{k_0}} + \lambda_2 t \r].
\]
This completes the proof.

\endgroup
\end{proof}

%=======================================
% Section: Proofs for section 2.4.2
%=======================================
%\section{Proofs omitted in Section \ref{sec: Bound in Frobenius norm}}

%###########################
\subsection{Proof of Lemma \ref{lemma:upper bound of outliers}}
\label{section:proof of upper bound of outliers}
%###########################
In this section we prove that the fraction of outliers is small with high probability for heavy-tailed data.
%\begin{proof}[Proof of Lemma \ref{lemma:upper bound of outliers}]
\begingroup
\allowdisplaybreaks
Denote $k_0 = \lfloor n/2 \rfloor$, $\Sigma = \Sigma_Y$ and $\chi_{i,j} = \mathds{1}\l\{\l\|\wt{Y}_{i,j}\r\|_2\leq R \r\}$, which are valid in this proof only. Then we have that $|\wt{J}| = \sum_{i\neq j}(1-\chi_{i,j})$, and by Markov's inequality, 
\begin{equation}\label{8.8-eq1}
P(\chi_{i,j}=0) \leq \frac{\expect{\l\|\wt{Y}_{i,j}\r\|_2^2}}{R^2} = \frac{\tr(\Sigma)}{R^2}.
\end{equation}
By the finite difference inequality, we have that for $0<\tau<1$,
\[
P\l(\frac{1}{N}\sum_{i\neq j}(1-\chi_{i,j}) \geq (1+\tau)\frac{\tr(\Sigma)}{R^2}\r) \leq \exp \l\{ \frac{-\tau^2k_0\tr(\Sigma)}{3R^2} \r\}.
\]
Setting $\tau = R (3t) ^{1/2} \big( k_0\tr(\Sigma) \big) ^{-1/2}$ and assuming that $R (3t) ^{1/2} \big( k_0\tr(\Sigma) \big) ^{-1/2}<1$, we see that
\[
\epsilon = \frac{|\wt{J}|}{N}\leq \frac{\tr(\Sigma)}{R^2} + \frac{\sqrt{\tr(\Sigma)}}{R} \sqrt{\frac{3t}{k_0}}
\]
with probability at least $1-e^{-t}$. Note that when R is chosen as
\[
R = \l( \frac{\tr(\Sigma) \l\|\Sigma\r\| n}{\log \big(n\cdot\rk(\Sigma)\big)} \r)^{\frac14}, 
\]
the assumption $R (3t) ^{1/2} \big( k_0\tr(\Sigma) \big) ^{-1/2} <1$ is equivalent to
\[
\frac{\tr(\Sigma) \l\|\Sigma\r\| n}{\log \big(n\cdot\rk(\Sigma)\big)} < \l( \frac{k_0\tr(\Sigma)}{3t} \r)^2 \leq \l( \frac{n\tr(\Sigma)}{3t} \r)^2,
\]
which is valid as long as $n\log(n) \geq {9t^2}$. With this choice of $R$, we have that 
\begin{equation}\label{bound for fraction of outliers}
\epsilon \leq \frac{\tr(\Sigma)}{R^2} \l( 1 + R\sqrt{\frac{3t}{k_0\tr(\Sigma)}}  \r)  < \frac{2\tr(\Sigma)}{R^2} = 2 \sqrt{\frac{\rk(\Sigma) \cdot \log\big(n\cdot\rk(\Sigma) \big) }{n} }
\end{equation}
with probability at least $1-e^{-t}$.

Moreover, when $\wt{Y}_{i,j}$ satisfies $L_4-L_2$ norm equivalence with constant $K$, we can improve the bound in \eqref{8.8-eq1} to $K^4\tr(\Sigma)^2/R^4$. By finite difference inequality again, we have that for $0<\tau<1$,
\[
P\l(\frac{1}{N}\sum_{i\neq j}(1-\chi_{i,j}) \geq (1+\tau)K^4\frac{\tr(\Sigma)^2}{R^4}\r) \leq \exp \l\{ \frac{-\tau^2k_0K^4\tr(\Sigma)^2}{3R^4} \r\}.
\]
Assuming that $R^2 (3t)^{1/2} \big( k_0\tr(\Sigma)^2K^4 \big)^{-1/2} < 1$, or equivalently
\[
\frac{\tr(\Sigma) \l\|\Sigma\r\| n}{\log \big(n\cdot\rk(\Sigma)\big)} < \frac{n\tr(\Sigma)^2L^4}{3t},
\]
we can set $\tau = R^2 (3t)^{1/2} \big( k_0\tr(\Sigma)^2K^4 \big)^{-1/2} $ and derive the following improved bound:
\begin{equation}\label{improved bound for fraction of outliers}
\epsilon \leq 2K^4 \frac{\rk(\Sigma) \cdot \log\big( n\cdot\rk(\Sigma) \big)}{n},
\end{equation}
which holds with probability at least $1-e^{-t}$. Note that the assumption above is valid when $K^4\rk(\Sigma) \log \big(n\cdot \rk(\Sigma) \big) > 3t$, which requires the order of $t$ to be at most $\log(n)$.
\endgroup
%\end{proof}

%###########################
\subsection{Proof of Theorem \ref{thm:lasso u-stat improved}} 
\label{sec:proof of the improved frobenius bound}
%###########################
\begingroup
\allowdisplaybreaks
In this section we present the proof of Theorem \ref{thm:lasso u-stat improved}, which gives an improved error bound in the Frobenius norm for heavy-tailed data. The main idea making the improvement possible is the fact that for heavy-tailed data, the ``outliers'' $\wt{V}_{i,j}$ are nonzero if and only if the ``well-behaved'' term $\wt{Z}_{i,j}$ equals zero. %Due to this restriction, the error bound in the Frobenius norm will be strictly larger than the error bound in the operator norm. Moreover, note that the proportion of outliers will be bounded above by a small quantity with high probability, namely $\epsilon \leq c(L) {\frac{\rk(\Sigma_Y)\log\big( n\rk(\Sigma_Y)\big)}{n}}$ with probability at least $1-\frac{1}{n}$ (see Lemma \ref{lemma:upper bound of outliers}). 
We will repeat parts of the proof of Theorem \ref{thm:lasso u-stat} using this fact along with the inequality of Theorem \ref{thm:heavy-tailed operator bound} 
%argument in the original proof and implement the bound in the operator norm in Theorem \ref{thm:heavy-tailed operator bound} 
instead of the inequality $\l\|A\r\| \leq \l\|A\r\|_F$ to derive an improved bound.

We start with some notations, which are specific to this proof. Let $N=n(n-1)$, and $c(K)$ be a constant depending on $K$ only, which can vary from step to step. Consider the events
\[
\m{E}_1 = \Big \{ \epsilon \leq c(K){\frac{\rk(\Sigma_Y)\log\big( n \cdot \rk(\Sigma_Y)\big)}{n}} \Big\},
\]
\[
\m{E}_2 = \l\{  \l\| \wh{S} - \Sigma_Y \r\| \leq c(K)  \l\|\Sigma_Y\r\| \sqrt{ \frac{\rk(\Sigma_Y)  \log (n) ^3}{n} } \r\},
\]
and
\begin{multline*}
\m E = \Bigg\{ \lambda_1 \geq \frac{140\l\|\Sigma_Z\r\|}{\sqrt{n(n-1)}}\sqrt{\rk(\Sigma_Z)} + 4\l\| \frac{1}{n(n-1)}\sum_{i\neq j}\wt{Z}_{i,j}\wt{Z}_{i,j}^T-\Sigma_Z \r\|, \\
\lambda_2 \geq \frac{140\l\|\Sigma_Z\r\|}{n(n-1)} \sqrt{\rk(\Sigma_Z)} + 4\frac{1}{\sqrt{n(n-1)}}\max_{i\neq j}\l\|\wt{Z}_{i,j}\wt{Z}_{i,j}^T-\Sigma_Z\r\| \Bigg\}.
\end{multline*}
We will need to condition on these three events throughout the proof, so we will first estimate their probabilities.
\begin{enumerate}
%=====item====
	\item In the view of Lemma \ref{lemma:upper bound of outliers}, we have that $P\l(\m{E}_1\r) \geq 1 - \frac{1}{n}$. 
%=====item====
	\item For $\m{E}$, we need to choose $\lambda_1$ and $\lambda_2$ appropriately in order to guarantee that $\m{E}$ happens with high probability. Since $\l\|\wt{Z}_{i,j}\r\|_2 \leq R$ almost surely, we can invoke the following version of matrix Bernstein inequality, which is a corollary of Theorem 3.1 from \citet{minsker2017some},
\begin{theorem}\label{thm:Bernstein for covariance}
Let $Z_1,\ldots,Z_n$ be i.i.d. random vectors with $\expect{Z_{1}}=0$ and $\l\|Z_1\r\|_2\leq R$ almost surely. Denote $\Sigma_Z = \expect{Z_1Z_1^T}$ and $B=\expect{(Z_1Z_1^T)^2}$, then for $t\geq \big( 9n \l\| B \r\| \big) / \big(16R^4 \big)$,
\begin{equation*}
\l\|\frac{1}{n}\sum_{i}Z_iZ_i^T - \Sigma_Z\r\| \leq C \l( \sqrt{\frac{\l\|B\r\|\l(\log(\rk(B))+t\r)}{n}}\vee\frac{R^2\l(\log(\rk(B))+t\r)}{n}  \r)
\end{equation*}
with probability at least $1-e^{-t}$.
\end{theorem}
Following the same argument as in Section \ref{section: proof of thm:mean_bound}, we can derive the following corollary for the transformed data:
\begin{corollary}\label{cor:bounds for mean and max}
Let $\wt{Z}_{i,j}$ be defined as in \eqref{eq:truncation}, namely, $\wh{Z} = \mathds{1}\l\{\l\|\wt{Y}_{i,j}\r\|_2\leq R\r\}$. Then  
\begin{equation}\label{cor 4.1-eq1}
\l\|\frac{1}{N}\sum_{i\neq j}\wt{Z}_{i,j}\wt{Z}_{i,j}^T - \Sigma_Z\r\| \leq C \l( \sqrt{\frac{\l\|B\r\|\l(\log(\rk(B))+t\r)}{n}}\vee\frac{R^2\l(\log(\rk(B))+t\r)}{n}  \r)
\end{equation}
with probability at least $1-2e^{-t}$, where $B = \expect{(\wt{Z}_{i,j}\wt{Z}_{i,j}^T)^2}$.
\end{corollary}
It remains to estimate $\l\|B\r\|$ and $\rk(B)$. \citet{mendelson2020robust} showed that if $Y$ satisfies an $L_4-L_2$ norm equivalence with constant $K$, then
\[
c\l\|\Sigma_Z\r\|\tr(\Sigma_Z) \leq \l\|B\r\| \leq c(K)\l\|\Sigma_Y\r\|\tr(\Sigma_Y)
\]
and 
\[
\tr(B)\leq c(K) \tr(\Sigma_Y)^2.
\]
Combining these two bounds, we have that
\begin{equation}\label{bound for rk(B)}
\rk(B) = \frac{\tr(B)}{\l\|B\r\|} \leq c(K) \frac{ \tr(\Sigma_Y)^2 }{\l\|\Sigma_Z\r\|\tr(\Sigma_Z) }.
\end{equation}
On the other hand, we have the following lemma which guarantees that $\Sigma_Z$ is close to $\Sigma_Y$.
\begin{lemma}
\label{lemma:truncated covariance} 
Let $Y\in \mb{R}^{d}$ be a mean zero random vector satisfying the $L_4-L_2$ norm equivalence with constant $K$. Then
\begin{equation}\label{truncated error operator}
\l\|\Sigma_Z-\Sigma_Y\r\|\leq c(K)\frac{\l\|\Sigma_Y\r\|\tr(\Sigma_Y)}{R^2} = c(K)\frac{\l\|\Sigma_Y\r\|^2\rk(\Sigma_Y)}{R^2},
\end{equation}
\begin{equation}\label{truncated error trace}
|\tr(\Sigma_Z)-\tr(\Sigma_Y)| \leq c(K) \frac{\tr^2(\Sigma_Y)}{R^2} = c(K) \frac{\l\|\Sigma_Y\r\|^2\rk(\Sigma_Y)^2}{R^2},
\end{equation}
\begin{equation}\label{truncated error frobenius}
\l\|\Sigma_Z-\Sigma_Y\r\|_F \leq c(K) \frac{\l\|\Sigma_Y\r\|^{\frac{1}{2}}\tr(\Sigma_Y)^{\frac{3}{2}}}{R^2} = c(K) \frac{\l\|\Sigma_Y\r\|^2\rk(\Sigma_Y)^\frac{3}{2}}{R^2},
\end{equation}
where $\Sigma_Y = \expect{YY^T}$, $Z = Y \mathds{1}\{\l\|Y\r\|_2\leq R\}$, $\Sigma_Z=\expect{ZZ^T}$, and $c(K)$ is a constant depending only on $K$.
\end{lemma}
The proof of Lemma \ref{lemma:truncated covariance} is presented in section \ref{sec:proof or heavy-tailed Frobenius bound} of the supplementary material. %By Lemma \ref{lemma:truncated covariance} we have that 
In particular, it implies that both $\l\|\Sigma_Z\r\|$ and $\tr(\Sigma_Z)$ are equivalent up to a multiplicative constant factor to $\l\|\Sigma_Y\r\|$ and $\tr(\Sigma_Y)$ respectively, as long as $R\geq c(K)\sqrt{\tr(\Sigma_Y)}$. The condition is valid given that $n\geq c(K) \rk(\Sigma_Y)\Big[ \log\big(\rk(\Sigma_Y)\big)+\log(n) \Big]$, and hence by \eqref{bound for rk(B)},
\[
\rk(B)\leq c(K)\rk(\Sigma_Y).
\]
Combining the bounds on $\l\|B\r\|$ and $\rk(B)$ with Corollary \ref{cor:bounds for mean and max}, the choice of $R$ as
\[
R = \l( \frac{\tr(\Sigma_Y) \l\|\Sigma_Y\r\| n}{\log \big(n\cdot\rk(\Sigma_Y)\big)} \r)^{\frac14}, 
\]
and the choice of $t=\log(n)$, we deduce that
\begin{equation}
\l\|\frac{1}{n(n-1)}\sum_{i\neq j}\wt{Z}_{i,j}\wt{Z}_{i,j}^T - \Sigma_Z\r\| \leq c(K) \l\|\Sigma_Y\r\|  \sqrt{\frac{\rk(\Sigma_Y) \l[\log \big(n \cdot \rk(\Sigma_Y))\r]}{n}} 
\end{equation}
with probability at least $1-\frac{2}{n}$.
Similarly, applying Theorem \ref{thm:Bernstein for covariance} 
%Corollary \ref{cor:bounds for mean and max}
to each single point $\wt{Z}_{i,j}$ and proceeding in a similar way in Section \ref{section: proof of thm:max_bound}, we deduce that
\begin{equation}
\max_{i\neq j} \l\|\wt{Z}_{i,j}\wt{Z}_{i,j}^T - \Sigma_Z\r\| \leq c(K) \l\|\Sigma_Y\r\| \sqrt{ n \cdot \rk(\Sigma_Y) \l[\log\big(n \cdot \rk(\Sigma_Y) \big)\r]} 
\end{equation}
with probability at least $1-\frac{1}{n}$. It follows that with the choices of
\begin{equation}\label{improved lambda1}
\lambda_1 \geq c(K) \l\|\Sigma_Y\r\| \sqrt{\frac{\rk(\Sigma_Y) \l[\log \big(n \cdot \rk(\Sigma_Y))\r]}{n}}
\end{equation}
and
\begin{equation}\label{improved lambda2}
\lambda_2 \geq c(K) \l\|\Sigma_Y\r\| \sqrt{ \frac{ \rk(\Sigma_Y) \l[\log\big( n \cdot \rk(\Sigma_Y) \big)\r] }{n} },
\end{equation}
we have that $P(\m{E}) \geq 1 - \frac{3}{n}$.
%===item=====
	\item To estimate the probability of the event $\m{E}_2$, we first state a modified version of Theorem \ref{thm:heavy-tailed operator bound}.
\begin{remark}\label{remark:modified version of theorem 4.2}
Following the same argument as in the proof of Theorem \ref{thm:heavy-tailed operator bound}, we can show that for $A\geq1$, with the choice of $\lambda_1 \leq c(K)\l\|\Sigma_Y\r\| \sqrt{ n\log(n) \rk(\Sigma_Y)}$, $\lambda_2 = c(K) \l\|\Sigma_Y\r\| \sqrt{ \frac{\rk(\Sigma_Y)\log(n)}{An}}$ and under the assumptions that 
\[
r_H\frac{\log(n)}{n} \l( 1+2c(1-\alpha)^{-1} \r) \leq \frac{1}{1280},
\]
the following inequality holds with probability at least $1-(\frac{8}{3}r_H+1) \frac{1}{n^A}$:
\begin{equation}
\label{diff between s_hat and Sigma in op norm with lambda1}
\l\|\wh{S}_\lambda-\Sigma_Y\r\| \leq \frac{20}{39}\lambda_1 + \frac{40}{13} c(K) \l\|\Sigma_Y\r\| \sqrt{ \frac{ \rk(\Sigma_Y)A\log(n)^3}{n}}.
\end{equation}
\end{remark}
%This result will help us adjust parameters to the desired order as indicated in \eqref{improved lambda1} and \eqref{improved lambda2}. 
Applying Remark \ref{remark:modified version of theorem 4.2} with $\lambda_1 = c(K) \l\|\Sigma_Y\r\| \sqrt{\frac{\rk(\Sigma_Y) \l[\log \big(n \cdot \rk(\Sigma_Y))\r]}{n}} $, $\lambda_2 = c(K) \l\|\Sigma_Y\r\| \sqrt{ \frac{\rk(\Sigma_Y)\log(n)}{An}}$ and the assumption $n \geq c(K) \rk(\Sigma_Y) \big[ \log\big( n \cdot \rk(\Sigma_Y) \big) \big]$, we get that
\begin{equation}\label{diff between s_hat and Sigma in op norm}
\l\| \wh{S} - \Sigma_Y \r\| \leq c(K)  \l\|\Sigma_Y\r\| \sqrt{ \frac{\rk(\Sigma_Y) A  \log (n) ^3}{n} } 
\end{equation}
with probability at least $1-{(\frac{8}{3}r_H+1)}{n^{-A}}$. This confirms that $P(\m{E}_2) \geq 1-{(\frac{8}{3}r_H+1)}{n^{-A}}$. Note that the choices of $\lambda_1$ and $\lambda_2$ coincide with \eqref{improved lambda1} and \eqref{improved lambda2}.
%===end enumerate====
\end{enumerate}
For what follows, we will condition on the events $\m{E}$, $\m{E}_1$ and $\m{E}_2$.  Repeating parts of the argument in Section \ref{proof:lasso-ustat}, we can arrive at the inequality %\eqref{bd-2}:
\begin{multline}\label{eq-8.66 middle result}
\l\|\Sigma_Z-\wh{S}_\lambda\r\|_F^2 + \l\|S-\wh{S}_\lambda\r\|_F^2 \leq \l\|\Sigma_Z-S\r\|_F^2 + \frac{1}{8}\l\|\wh{S}_\lambda-S\r\|_F^2 + 2\lambda_1^2\rank(S)(\sqrt{2}+1)^2 \\
+ \frac{2}{\sqrt{N}}\sum_{i\neq j} \dotp{\wt{U}_{i,j}^*-\wh{U}_{i,j}}{\wh{S}_\lambda-S}.
\end{multline}
By Lemma \ref{lemma:truncated covariance} and choosing $R$ as  
\begin{equation}
\label{choice of R recall}
R = \l( \frac{\tr(\Sigma_Y) \l\|\Sigma_Y\r\| n}{\log \big(n\cdot\rk(\Sigma_Y)\big)} \r)^{\frac14}, 
\end{equation}
we have that 
\[
\l\| \Sigma_Z - \Sigma_Y \r\|_F \leq c(K) \l\|\Sigma_Y\r\| \frac{ \rk(\Sigma_Y) \sqrt{ \log\big(\rk(\Sigma_Y) \big) + \log(n) }}{\sqrt{n}}.
\]
Therefore, we can deduce from \eqref{eq-8.66 middle result} that
\begin{multline}\label{eq-8.66 middle result - improved}
\l\|\Sigma_Y-\wh{S}_\lambda\r\|_F^2 + \l\|S-\wh{S}_\lambda\r\|_F^2 \leq \l\|\Sigma_Z-S\r\|_F^2 + \frac{1}{8}\l\|\wh{S}_\lambda-S\r\|_F^2 + 2\lambda_1^2\rank(S)(\sqrt{2}+1)^2 \\
+ \frac{2}{\sqrt{N}}\sum_{i\neq j} \dotp{\wt{U}_{i,j}^*-\wh{U}_{i,j}}{\wh{S}_\lambda-S} + c(L) \l\|\Sigma_Y\r\|^2 \frac{ \rk(\Sigma_Y)^2 { \log\big( n \cdot \rk(\Sigma_Y) \big) }}{{n}}.
\end{multline}
It remains to bound the expression $$\frac{2}{\sqrt{N}}\sum_{i\neq j} \dotp{\wt{U}_{i,j}^*-\wh{U}_{i,j}}{\wh{S}_\lambda-S}.$$
First, note that 
\[
\sum_{(i,j)\notin\wt{J}} \l\|\wh{U}_{i,j}-\wt{U}_{i,j}^*\r\|_1 = \sum_{(i,j)\notin\wt{J}} \l\| P_{L_{i,j}^\perp} \wh{U}_{i,j} P_{L_{i,j}^\perp}\r\|_1
\]
and that
\[
\sum_{(i,j)\in\wt{J}} \l\|\wh{U}_{i,j}-\wt{U}_{i,j}^*\r\|_1 \leq \sum_{(i,j)\in\wt{J}} \l\|\m{P}_{L_{i,j}}(\wh{U}_{i,j}-\wt{U}_{i,j}^*)\r\|_1 + \sum_{(i,j)\in\wt{J}} \l\| P_{L_{i,j}^\perp} \wh{U}_{i,j} P_{L_{i,j}^\perp} \r\|_1.
\]
By Lemma \ref{lemma 3.2}, %along with the assumption $\lambda_1\leq \lambda_2$, 
we have that
\begin{multline*}
\sum_{i\neq j} \l\|\wh{U}_{i,j}-\wt{U}_{i,j}^*\r\|_1 \leq \sum_{(i,j)\in\wt{J}} \l\|\mathcal{P}_{L_{i,j}}(\wh{U}_{i,j}-\wt{U}_{i,j}^*)\r\|_1 + \sum_{i\neq j} \l\| P_{L_{i,j}^\perp} \wh{U}_{i,j} P_{L_{i,j}^\perp} \r\|_1\\
\leq \sum_{(i,j)\in\wt{J}} \l\|\m{P}_{L_{i,j}}(\wh{U}_{i,j}-\wt{U}_{i,j}^*)\r\|_1 + 3\l( \frac{\lambda_1}{\lambda_2}\l\|\m P_{L(k)}(\wh{S}-\Sigma(k)) \r\|_1 + \sum_{(i,j)\in\wt{J}} \l\|\m{P}_{L_{i,j}}(\wh{U}_{i,j}-\wt{U}_{i,j}^*)\r\|_1  \r)\\
\leq 4\sum_{(i,j)\in\wt{J}} \l\|\m{P}_{L_{i,j}}(\wh{U}_{i,j}-\wt{U}_{i,j}^*)\r\|_1 + 3 \frac{\lambda_1}{\lambda_2} \l\|\m P_{L(k)}(\wh{S}-\Sigma(k)) \r\|_1
\end{multline*}
Repeating the argument behind \eqref{3.3.2}, we have that
\begin{multline} 
\label{eq-8.67 middle result}
\frac{2}{\sqrt{N}}\sum_{i\neq j}\dotp{\wt{U}_{i,j}^*-\wh{U}_{i,j}}{\wh{S}-S} \leq \frac{2\l\|\wh{S}-S\r\|}{\sqrt{N}} \sum_{i\neq j} \l\|\wt{U}_{i,j}^*-\wh{U}_{i,j}\r\|_1 \\
\leq \frac{2\l\|\wh{S}-S\r\|}{\sqrt{N}} \l( 4\sum_{(i,j)\in\wt{J}} \l\|\m{P}_{L_{i,j}}(\wh{U}_{i,j}-\wt{U}_{i,j}^*)\r\|_1 + 3\frac{\lambda_1}{\lambda_2}\l\|\m P_{L(k)}(\wh{S}-\Sigma(k)) \r\|_1\r)\\
\leq \frac{2\l\|\wh{S}-S\r\|}{\sqrt{N}} \l(3\sqrt{2k} \frac{\lambda_1}{\lambda_2}\l\|\wh{S}-\Sigma_Z\r\|_F + 8\sum_{(i,j)\in\wt{J}} \l\|\wh{U}_{i,j}-\wt{U}_{i,j}\r\|_F  \r)\\
\leq 6\l\|\wh{S}-S\r\| \frac{\lambda_1}{\lambda_2} \sqrt{\frac{2k}{N}}\l\|\wh{S}-\Sigma_Z\r\|_F + 16\l\|\wh{S}-S\r\| \sqrt{\frac{|\wt{J}|}{N}}\sqrt{\sum_{(i,j)\in\wh{J}}\l\|\wh{U}_{i,j}-\wh{U}_{i,j}^*\r\|_F^2 }.
\end{multline}
We will estimate the two terms on the right-hand side of the above inequality one by one.
Note that we did not apply the crude bound $\l\|\wh{S} -S \r\| \leq \l\| \wh{S}-S \r\|_F$ since $\l\|\wh{S} -S \r\|$ is strictly smaller for the heavy tailed data due to the independence of the ``outliers''. By triangle inequality, %the definition of event $\m{E}_2$ (or equivalently, inequality \eqref{diff between s_hat and Sigma in op norm}), 
Lemma \ref{lemma:truncated covariance} and the choice of $R = \l( {\tr(\Sigma) \l\|\Sigma\r\| n}
\r) ^{1/4} \l( {\log \big(n\cdot\rk(\Sigma)\big)} \r)^{-1/4},$
%\[
%R = \l( \frac{\tr(\Sigma) \l\|\Sigma\r\| n}{\log \big(n\cdot\rk(\Sigma)\big)} \r)^{\frac14}, 
%\]
we have that on the event $\m{E}_2$,
\begin{multline}\label{eq-8.68 middle result}
\l\|\wh{S}_\lambda-S\r\| \leq \l\|\wh{S}_\lambda-\Sigma_Y\r\| + \l\|\Sigma_Y - \Sigma_Z\r\| +  \l\|\Sigma_Z-S\r\| \\
\leq c(L) \l( \l\|\Sigma_Y\r\| \sqrt{ \frac{\rk(\Sigma_Y) A \log (n) ^3}{n} }  + \frac{\l\|\Sigma_Y\r\|\tr(\Sigma_Y)}{R^2} \r)+  \l\|\Sigma_Z-S\r\| \\
\leq \underbrace{ c(L)  \l\|\Sigma_Y\r\| \sqrt{ \frac{\rk(\Sigma_Y) A \log (n) ^3}{n} }   }_{:=\mathrm{I}} +  \l\|\Sigma_Z-S\r\|.
\end{multline}
Note that the term $\mathrm{I}$ is of the order $\sqrt{{\rk(\Sigma_Y)}/{n}}$, up to the logarithmic factors. 
For what follows, we set $S = \Sigma_Y$, and \eqref{eq-8.68 middle result} implies that $\l\|\wh{S}_\lambda - S\r\| = \l\|\wh{S}_\lambda - \Sigma_Y \r\| \leq \mathrm{I}.$ To estimate $\sqrt{\sum_{\wt{J}}\l\| \wh{U}_{i,j}-\wt{U}_{i,j}^*\r\|_F^2 }$, we can apply inequality 
%\eqref{bound_case_1} which entails that
%\begin{multline*}
%\sqrt{\sum_{\wt{J}}\l\| \wh{U}_{i,j}-\wt{U}_{i,j}^*\r\|_F^2 } \leq 2\sqrt{2} \Bigg(\l\|\Sigma_Z-S\r\|_F + \sqrt{\rank(S)}\lambda_1(\sqrt{2}+1) 
%+  \lambda_2 \frac{(4/3+\sqrt{2})}{\sqrt{2}}\sqrt{|\wt{J}|} \Bigg),
%\end{multline*}
\eqref{ineq_2_9_temporary} which entails that
\begin{multline*}
\sqrt{\sum_{\wt{J}}\l\| \wh{U}_{i,j}-\wt{U}_{i,j}^*\r\|_F^2 } \leq 2\sqrt{2} \Bigg(\l\|\Sigma_Z - \Sigma_Y\r\|_F + \sqrt{\rank(S)}\lambda_1(\sqrt{2}+1) 
 \\
+  \lambda_2 \frac{(4/3+\sqrt{2})}{\sqrt{2}}\sqrt{|\wt{J}|} + \sqrt{2(6\sqrt{2} + 6)\frac{\lambda_1}{\lambda_2} } \l( {\frac{\rank(\Sigma_Y)}{N}} \r)^{\frac{1}{4}} \l\| \wh{S}_\lambda - \Sigma_Y \r\|_F  \Bigg),
\end{multline*}
%given that $k=\lfloor\frac{N\lambda_2^2}{1200\lambda_1^2}\rfloor$, $\rank(S)\leq\frac{n^2}{56000} \cdot \frac{\lambda_2^2}{\lambda_1^2}$, $|\wt{J}|\leq\frac{N}{6400}$. Note that the condition $\rank(S)\leq\frac{n^2}{56000} \cdot \frac{\lambda_2^2}{\lambda_1^2}$ can be reduced to $\rank(S)\leq c_1  \cdot \frac{n^2\lambda_2^2}{\lambda_1^2}$ for any $c_1 \leq \frac{1}{5980}$ according to \eqref{condition on rank of S}. For what follows, we assume that $\rank(\Sigma_Y)$ satisfies the same conditions as $\rank(S)$, namely, we assume that $$\rank(\Sigma_Y) \leq c_1 \frac{\lambda_2^2 n^2}{\lambda_1^2},$$ and set $S = \Sigma_Y$. This implies that $\l\|\wh{S}_\lambda - S\r\| = \l\|\wh{S}_\lambda - \Sigma_Y \r\| \leq \mathrm{I}$. 
given that $k=\lfloor\frac{N\lambda_2^2}{1200\lambda_1^2}\rfloor$ and $|\wt{J}|\leq {N}/{6400}$. For simplicity, we denote $B = \sqrt{2(6\sqrt{2} + 6){\lambda_1}/{\lambda_2} } \l( {{\rank(\Sigma_Y)}/ {N}} \r)^{{1}/1{4}}$.
Now we will estimate the two terms in \eqref{eq-8.67 middle result}:
\begin{itemize}
	\item First, 
	\begin{equation}
	6\l\|\wh{S}_\lambda- \Sigma_Y  \r\|   \frac{\lambda_1}{\lambda_2} \sqrt{\frac{2k}{N}}\l\|\wh{S}_\lambda-\Sigma_Z\r\|_F \\
	\leq 6 \cdot \mathrm{I} \cdot   \frac{\lambda_1}{\lambda_2}  \sqrt{\frac{2k}{N}} \l\|\wh{S}_\lambda-\Sigma_Z\r\|_F.
	\end{equation}
	This term is independent of the outliers, and a direct application of the inequality $2ab\leq a^2 + b^2$ gives that
	\begin{equation} \label{eq-8.70 middle result}
	6\l\|\wh{S}_\lambda- \Sigma_Y \r\|  \frac{\lambda_1}{\lambda_2} \sqrt{\frac{2k}{N}}\l\|\wh{S}_\lambda-\Sigma_Z\r\|_F \leq 3 \cdot  \frac{\lambda_1}{\lambda_2} \sqrt{\frac{2k}{N}} \l( \mathrm{I}^2 + \l\|\wh{S}_\lambda-\Sigma_Z\r\|_F^2 \ \r).
	\end{equation}
	\item Second, 
	\begin{multline}\label{eq-8.71 middle result}
	16\l\|\wh{S}_\lambda- \Sigma_Y \r\| \sqrt{\frac{|\wt{J}|}{N}}\sqrt{\sum_{(i,j)\in\wt{J}}\l\|\wh{U}_{i,j}-\wt{U}_{i,j}^*\r\|_F^2}\\
	\leq 16\l\|\wh{S}_\lambda- \Sigma_Y \r\|\sqrt{\frac{|\wt{J}|}{N}}\cdot2\sqrt{2} \Bigg(\l\|\Sigma_Z- \Sigma_Y \r\|_F + \sqrt{\rank(\Sigma_Y)}\lambda_1(\sqrt{2}+1) 
	\\
	+  \lambda_2 \frac{(4/3+\sqrt{2})}{\sqrt{2}}\sqrt{|\wt{J}|} + B \l\| \wh{S}_\lambda - \Sigma_Y \r\|_F \Bigg) 
	\\
	\leq 16\sqrt{2}\sqrt{\frac{|\wt{J}|}{N}} \l( \l\| \Sigma_Y-\wh{S}_\lambda\r\|^2+\l\|\Sigma_Z- \Sigma_Y \r\|_F^2 \r) + 8(\sqrt{2}+1)^2\frac{|\wt{J}|}{N} \l\|\Sigma_Y-\wh{S}_\lambda\r\|^2  
	\\
	+64\lambda_1^2\rank(\Sigma_Y)+ 32(4/3+\sqrt{2}) \cdot \mathrm{I} \cdot \lambda_2 \sqrt{ \frac{ |\wt{J}|^2}{N} } 
	\\
	+ 16\sqrt{2} \l( \sqrt{n} \cdot \mathrm{I}^2 \cdot B^2 + \frac{|\wt{J}|}{N} \cdot \frac{1}{\sqrt{n}} \l\| \wh{S}_\lambda - \Sigma_Y \r\|_F^2 \r).
	%\leq 16\sqrt{2}\sqrt{\frac{|\wt{J}|}{N}} \l\|\Sigma_Z-{S}\r\|_F^2  + \l( 16\sqrt{2}\sqrt{\frac{|\wt{J}|}{N}} + 8(\sqrt{2}+1)^2\frac{|\wt{J}|}{N}	+\frac{1}{4} \r)(2\mathrm{I}^2+2\l\|\Sigma_Z-S\r\|^2)\\
	%+ 64\lambda_1^2\rank(S) + (32(4/3+\sqrt{2}))^2\frac{|\wt{J}|^2}{N}\lambda_2^2
	\end{multline}
\end{itemize}
Combining (\ref{eq-8.66 middle result - improved}, \ref{eq-8.67 middle result}, \ref{eq-8.70 middle result}, \ref{eq-8.71 middle result}), and assuming that $6\cdot \frac{\lambda_1}{\lambda_2}\sqrt{\frac{2k}{N}} + 96\sqrt{2} \sqrt{\frac{|\wt{J}|}{N}} \leq \delta \leq \frac{3}{8}$ and $B^2 \leq (\delta \sqrt{A}) / (16 \sqrt{2n}) $, we deduce that
% we have that under the assumptions $6\cdot \frac{\lambda_1}{\lambda_2}\sqrt{\frac{2k}{N}} + 96\sqrt{2} \sqrt{\frac{|\wt{J}|}{N}} \leq \delta \leq \frac{3}{8}$, $\rank(\Sigma_Y) \leq c_1 \frac{\lambda_2^2 n^2}{\lambda_1^2}$ and $S=\Sigma_Y$,  
\begin{multline}\label{eq-8.72 middle result}
(1-\delta)\l\|\Sigma_Y-\wh{S}_\lambda\r\|_F^2 \leq %(1+\delta )\l\|\Sigma_Y-S\r\|_F^2 + 
c(\delta)76\lambda_1^2\rank(\Sigma_Y) + \delta \cdot \mathrm{I}^2  
+ 32(4/3+\sqrt{2}) \cdot \mathrm{I} \cdot \lambda_2 \epsilon n  \\
+  c(L,\delta) \l\|\Sigma_Y\r\|^2 \frac{ \rk(\Sigma_Y)^2 { \log\big( n \cdot \rk(\Sigma_Y) \big) }}{{n}},
\end{multline} 
where $\epsilon = {|\wt{J}|} / {N}$ is the proportion of outliers. 
Finally, we recall that the choices of $\lambda_1$ and $\lambda_2$ are
\begin{equation}\label{pf:choice of lambda1}
\lambda_1 = c(K) \l\|\Sigma_Y\r\| \sqrt{\frac{\rk(\Sigma_Y) \l[\log \big(n \cdot \rk(\Sigma_Y))\r]}{n}}
\end{equation}
and
\begin{equation}\label{pf:choice of lambda2}
\lambda_2 = c(K) \l\|\Sigma_Y\r\| \sqrt{ \frac{\rk(\Sigma_Y)\log(n)}{An}}.
\end{equation}
Also, recall the definition of $\mathrm{I}$ in \eqref{eq-8.68 middle result}:
\begin{equation}\label{pf:magnitude of I}
\mathrm{I} = c(K)  \l\|\Sigma_Y\r\| \sqrt{ \frac{\rk(\Sigma_Y)  A\log(n)^3 }{n} }.
\end{equation}
Combining the equations (\ref{pf:choice of lambda1}, \ref{pf:choice of lambda1}, \ref{pf:magnitude of I}) with \eqref{eq-8.72 middle result}, we derive that 
\begin{multline}\label{ineq:result}
\l\|\Sigma_Y-\wh{S}_\lambda\r\|_F^2 \\
\leq c(K,\delta)\l\|\Sigma_Y\r\|^2 \frac{\rk(\Sigma_Y) \log\big( n\cdot \rk(\Sigma_Y)\big)}{n} \rank(\Sigma_Y) + c(K,\delta) \l\|\Sigma_Y\r\|^2 { \frac{\rk(\Sigma_Y) A \log(n)^3 }{n} }\\
+ c(K,\delta) \epsilon \cdot n \cdot  \l\|\Sigma_Y\r\|^2 \frac{\rk(\Sigma_Y) \log(n)^2 }{n} + c(K,\delta) \l\|\Sigma_Y\r\|^2 { \frac{\rk(\Sigma_Y)^2  \log\big(n\cdot \rk(\Sigma_Y)\big)  }{n} } \\
\leq c(K,\delta)\l\|\Sigma_Y\r\|^2 \frac{\rk(\Sigma_Y) \log\big( n\cdot \rk(\Sigma_Y)\big)}{n} \rank(\Sigma_Y) 
+ c(K, \delta) \l\|\Sigma_Y\r\|^2\frac{\rk(\Sigma_Y)^2A \log(n)^3}{n} 
\end{multline}
under the assumptions that %$\rank(\Sigma_Y) \leq c_1n^2 \cdot \frac{\lambda_2^2}{\lambda_1^2}$ for some constant $c_1\leq\frac{1}{5980}$ 
$B^2 \leq (\delta \sqrt{A}) / (16 \sqrt{2n}) $ 
and $n\geq c(K) \rk(\Sigma_Y) \big[ \log\big(n \cdot \rk(\Sigma_Y) \big) \big]$, where the last step in \eqref{ineq:result} follows from Lemma \ref{lemma:upper bound of outliers}. Note that the assumption $B^2 \leq (\delta \sqrt{A}) / (16 \sqrt{2n})$ is valid as long as $\rank(\Sigma_Y) \leq c_1 \delta^2 \cdot n {A\lambda_2^2}{\lambda_1^{-2}}$ for any constant $c_1 \leq {\l( 4 (6\sqrt{2} + 6)^2 \r) ^{-1}}$. 
Finally, by the union bound over the events $\m{E}$, $\m{E}_1$ and $\m{E}_2$, inequality \eqref{ineq:result} will hold with probability at least $1 - { \l( \frac{8}{3}r_H + 1 \r) }{n^{-A}}-\frac{4}{n}$. 
To this end, note that the condition $\rank(\Sigma_Y) \leq c_1 \delta^2 \cdot n {A\lambda_2^2}{\lambda_1^{-2}}$ is equivalent to 
\[
\rank(\Sigma_Y) \leq c(K) \cdot {n} \cdot \frac{\log(n)}{\log\big(n \cdot \rk(\Sigma_Y) \big)}
\]
when $\lambda_1$, $\lambda_2$ are chosen as \eqref{pf:choice of lambda1} and \eqref{pf:choice of lambda2} respectively. The upper bound on $\rank(\Sigma_Y)$ is in the order of $n$ up to logarithmic factors. 
\endgroup

%=======================================
% Section: Proofs for Section 2.5
%=======================================
\section{Proofs ommitted from numerical experiments} % Section \ref{sec:numerical simulation}}
\label{sec:proofs for section 2.4}
	%=======================================
	% Subsection: Convergence analysis of PGD
	%=======================================
\subsection{Convergence analysis of the proximal gradient method (Theorem \ref{pgd convergence})}
\label{sec:convergence analysis of pgd}
\begingroup
\allowdisplaybreaks
In this section we present the convergence analysis of the proximal gradient method (with matrix variables). It is worth noting that our analysis follows the argument in \citet[Chapter 10]{beck2017first}. Recall that our loss function can be written in the form $L(S) = g(S) + h(S)$, where $h$ is convex, and $g$ is the average of N functions $g_{i,j}(S) = \tr \l(\rho_{\frac{\sqrt{N}\lambda_2}{2}}(\wt{Y}_{i,j}\wt{Y}_{i,j}^T - S) \r)$. Note that $\nabla g_{i,j}(S) = -\rho'_{\frac{\sqrt{N}\lambda_2}{2}}(\wt{Y}_{i,j}\wt{Y}_{i,j}^T - S)$, and using the fact from \citet[Lemma VII.5.5]{bhatia2013matrix} we have that $\nabla g_{i,j}(S)$ is Lipschitz in Frobenius norm with $L=1$, i.e.
\[
\l\|\nabla g_{i,j}(U) - \nabla g_{i,j}(V) \r\|_F \leq L\l\|U-V\r\|_F,
\]
hence $g(S)$ is also Lipschitz in Frobenius norm with $L=1$. We have the following matrix form of the descent lemma:
\begin{lemma}\label{descent lemma}
Assume that $g(S)$ is Lipschitz in Frobenius norm with constant $L>0$. Then
\[
g(S_2) \leq g(S_1) + \dotp{\nabla g(S_1)}{S_2 - S_1} + \frac{L}{2} \l\|S_2 -S_1 \r\|_F^2.
\]
\end{lemma}
\begin{proof}
First, denote $U_t = S_1 + t(S_2-S_1)$, we have that
\[
g(S_2) = g(S_1) + \int_{0}^{1}\dotp{\nabla g(U_t)}{S_2- S_1} dt,
\]
hence
\begin{multline*}
\big| g(S_2) - g(S_1) - \dotp{\nabla g(S_1)}{S_2-S_1} \big| = \big| \int_{0}^{1} \dotp{\nabla g(U_t) - \nabla g(S_1)}{S_2 - S_1}dt \big| \\
\leq \int_{0}^{1} \l\| \nabla g(U_t) - \nabla g(S_1) \r\|_F\l\|S_2 - S_1\r\|_F dt \leq \frac{L}{2}\l\|S_2-S_1\r\|_F^2.
\end{multline*}
\end{proof}
Now recall that the proximal gradient descent algorithm update is 
\[
S^{t+1} = \prox_{\alpha_t,h}(S^{t}-\alpha_t\nabla g(S^{t})).
\]
Set $G_\alpha(S) = \big[ S - \prox_{\alpha,h}(S-\alpha\nabla g(S))\big] /\alpha$, then $S^{t+1} = S^t - \alpha_t G_{\alpha_t}(S^t)$. The following lemma guarantees that the PGD makes progress at each step.
\begin{lemma}\label{PGD descent}
Assume that $0\leq \alpha_t \leq L$ for all $t=1,2,\ldots$, then for any symmetric matrix $U$,
\[
L(S^{t+1}) \leq L(U) + \dotp{G_{\alpha_t}(S^{t})}{S^t - U} - \frac{\alpha_t}{2} \l\|G_{\alpha_t}(S^t)\r\|_F^2.
\]
\end{lemma}
\begin{proof}
Since $g(\cdot)$ is convex, we have that for any symmetric matrix $U$,
\[
g(U) \geq g(S^t) + \dotp{\nabla g(S^t)}{U-S^t}.
\]
Combining this with Lemma \ref{descent lemma}, we have that
\begin{multline*}
g(S^{t+1}) \leq g(S^t) + \dotp{\nabla g(S^t)}{S^{t+1}-S^t} + \frac{\alpha_t}{2}\l\|G_{\alpha_t}(S^t)\r\|_F^2 \\
\leq g(U) - \dotp{\nabla g(S^t)}{U-S^t}  + \dotp{\nabla g(S^t)}{S^{t+1}-S^t}  +  \frac{\alpha_t}{2}\l\|G_{\alpha_t}(S^t)\r\|_F^2 \\
= g(U) + \dotp{\nabla g(S^t)}{S^{t+1}-U}  +  \frac{\alpha_t}{2}\l\|G_{\alpha_t}(S^t)\r\|_F^2 .
\end{multline*}
Since $h(\cdot)$ is convex, for any $V \in \partial h(S^{t+1})$,
\[
h(U) \geq h(S^{t+1}) + \dotp{V}{U-S^{t+1}}.
\]
Recall that
\[
S^{t+1} = \argmin_S\Big\{ h(S) + \frac{1}{2\alpha_t}\l\| S- (S^t - \alpha_t\nabla g(S^t))\r\|_F^2 \Big\}.
\]
By the optimality conditions,
\[
0\in \partial h(S^{t+1}) + \frac{1}{\alpha_t}(S^{t+1}-S^t + \alpha_t\nabla g(S^t)).
\]
Therefore,
\[
G_{\alpha_t}(S^t) - \nabla g(S^t) \in \partial h(S^{t+1}),
\]
and
\begin{multline*}
L(S^{t+1}) \leq g(U) + h(U) + \dotp{\nabla g(S^t)}{S^{t+1}-U} + \frac{\alpha_t}{2}\l\|G_{\alpha_t}(S^t)\r\|_F^2 \\
+ \dotp{G_{\alpha_t}(S^t) - \nabla g(S^t)}{S^{t+1}-U}\\
= L(U) + \dotp{G_{\alpha_t}(S^t)}{S^{t+1}-U} + \frac{\alpha_t}{2}\l\|G_{\alpha_t}(S^t)\r\|_F^2  \\
\leq L(U) + \dotp{G_{\alpha_t}(S^t)}{S^{t}-U} - \frac{\alpha_t}{2} \l\|G_{\alpha_t}(S^t)\r\|_F^2,
\end{multline*}
where in the last step we used the fact that $S^{t+1} = S^t - \alpha_t G_{\alpha_t}(S^t)$.
\end{proof}
Taking $U=S^t$ in Lemma \ref{PGD descent}, we have that
\[
L(S^{t+1}) \leq L(S^t) - \frac{\alpha_t}{2} \l\|G_{\alpha_t}(S^t)\r\|_F^2,
\]
i.e. the PGD method is making progress at each iteration. Taking $U=S^*$ in Lemma \ref{PGD descent}, where $S^*$ is the true minimizer of $L(S)$, we have that
\begin{multline*}
L(S^{t+1}) - L(S^*) \leq  \dotp{G_{\alpha_t}(S^t)}{S^{t}-S^*} - \frac{\alpha_t}{2} \l\|G_{\alpha_t}(S^t)\r\|_F^2 \\
= \frac{1}{2\alpha_t}\Big(\dotp{2\alpha_t G_{\alpha_t}(S^t)}{S^t-S^*} - \l\|\alpha_t G_{\alpha_t}(S^t)\r\|_F^2 \Big)\\
=\frac{1}{2\alpha_t}\Big(\l\|S^t - S^*\r\|_F^2 - \l\|\alpha_t G_{\alpha_t} -S^t + S^*\r\|_F^2 \Big)\\
= \frac{1}{2\alpha_t}\Big(\l\|S^t - S^*\r\|_F^2 - \l\|S^{t+1}- S^*\r\|_F^2 \Big).
\end{multline*}
Assuming that the step size is fixed (i.e. $\alpha_t = \alpha$) or diminishing (i.e. $\alpha_t\geq \alpha_{T+1} = \alpha)$, summing up both sides of the above inequality for $t=0,1,\ldots,T$, and recalling that $L(S^{t+1})\leq L(S^t)$, we have that
\[
(T+1)(L(S^{t+1}-L(S^*)) \leq \frac{1}{2\alpha}\Big(\l\|S^0 - S^*\r\|_F^2 - \l\|S^{t+1}-S^*\r\|_F^2\Big),
\]
hence
\[
L(S^{t+1}) - L^* \leq \frac{\l\|S^0 - S^*\r\|_F^2}{2\alpha(T+1)},
\]
as desired. Note that the convergence rate can be improved to $\mathcal{O}({1}/{T^2})$, see \citet{nesterov1983method,nesterov2003introductory}, \citet{tseng2008accelerated} for details.
\endgroup

	%=======================================
	% Subsection: Numerical methods of updating eigenvalues
	%=======================================
\subsection{Numerical method of updating eigenvalues (solving equation \eqref{compute_eigenvalue})}
\label{sec:eigenvalues}
\begingroup
\allowdisplaybreaks
In this section, we present the numerical method introduced by \citet{bunch1978rank} which computes the roots of $\omega_i(\mu)=0$ for $i=1,\ldots,k\leq d$, where $\omega_i(\mu)$ is defined as %\eqref{compute_eigenvalue}:
\begin{equation*}
\omega_i(\mu) = 1 + \sum_{j=1}^{k}\frac{\zeta_j^2}{\delta_j-\mu}
\end{equation*}
and $\delta_j = (d_j-d_i)/\rho$. Recall that the eigenvalues of $C=D+\rho zz^T$, denoted as $\wt{d}_1,\ldots,\wt{d}_k$, and the eigenvalues of $D$, denoted as $d_1,\ldots,d_k$, satisfy the identity $\wt{d}_i=d_i+\rho \mu_i$ with $\omega_i(\mu_i)=0$. Therefore, it remains to solve equations $\omega_i(\mu)=0$, $i=1,\ldots,k$. 
Fix $i\in\{1,\ldots,k\}$, and define
\[
\psi_i(t) = \sum_{j=1}^{i} \frac{\zeta_j^2}{\delta_j-t},\quad i=1,\ldots,k,
\]
and 
\[
\phi_i(t) = \l\{
			    \begin{array}{ll}
			    	0,\quad i=k, \\
			    	\sum_{j=i+1}^{k}\frac{\zeta_j^2}{\delta_j-t},\quad 1\leq i<k.
			    \end{array}
			    \r.
\]
It is clear that $\omega_i(t) = 1+\psi_i(t) +\phi_i(t)$. Without loss of generality, we shall assume that $\rho>0$; otherwise, we can replace $d_i$ by $-d_{k-i+1}$ and $\rho$ by $-\rho$. Also, we assume that $k>1$; otherwise, we have the trivial case $\mu_1 = \zeta_1^2$. We will deal with the case $i<k$ and $i=k$ separately.
\begin{enumerate}
\item Assume that $i\in\{1,\ldots,k-1\}$ is fixed. We are seeking $\mu_i$ such that $0<\mu_i<\min\{1-\sum_{j=1}^{i-1},\delta_{i+1}\}$ (by Theorem \ref{rank-1 update}) and 
\[
-\psi_i(\mu_i) = \phi_i(\mu_i)+1.
\]
Assume that we have an approximation $t_1$ to the root $\mu_i$ with $0<t_1<\mu_i$, and we want to get an updated approximation $t_2$. As suggested by \citet{bunch1978rank}, we shall consider the local approximation to the rational functions $\phi_i$ and $\psi_i$ at $t_1$, namely,
\begin{align}
&\frac{p_1}{q_1-t_1}=\psi_i(t_1),\quad\qquad \frac{p_1}{(q_1-t_1)^2}=\psi_i'(t_1),\\
&r_1+\frac{s_1}{\delta-t_1} =\phi_i(t_1),\quad \frac{s_1}{(\delta-t_1)^2} = \phi_i'(t_1).
\end{align}
where $\delta = \delta_{i+1}$. It can be easily verified that $p_1,q_1,r_1,s_1$ satisfies
\begin{align}\label{eq:update pqrs}
&p_1 = \psi_i(t_1)^2/\psi_i'(t_1),\quad q_1=t_1+\psi_i(t_1)/\psi_i'(t_1),\\
&r_1=\phi_i(t_1) - (\delta-t_1) \phi_i'(t_1),\quad s=(\delta-t_1)^2\phi_i'(t_1).
\end{align}
The updated approximation $t_2$ is then obtained by solving the following equation:
\begin{equation}\label{eq:update t2}
-\frac{p_1}{q_1-t_2} = 1 + r_1 + \frac{s_1}{\delta-t_2}.
\end{equation}
Direct computation shows that 
\[
t_2 = t_1 + 2b/(a+\sqrt{a^2-4b}),
\]
where
\begin{align*}
&a = \frac{(\delta-t_1)(1+\phi_i(t_1))+\psi_i(t_1)^2/\psi_i'(t_1)}{c} + \psi_i(t_1)/\psi_i'(t_1), \\
&b = \frac{ (\delta-t_1)w\psi_i(t_1) }{\psi_i'(t_1)c},  \\
&c = 1+\phi_i(t_1)-(\delta-t_1)\phi_i'(t_1),  \\
&w = 1 + \phi_i(t_1) + \psi_i(t_1).
\end{align*}
The following theorem shows that the update \eqref{eq:update t2} is guaranteed to converge to $\mu_i$:
\begin{theorem}[\citet{bunch1978rank}]\label{thm:quadratic convergence of eigenvalue computation}
Let $t_0\in(0,\mu_i)$ and $t_{j+1}$ be the solution of $-\frac{p_{j}}{q_{j}-t} = 1 + r_{j} + \frac{s_{j}}{\delta-t}$, $j\geq 0$, where $p_j, q_j, r_j, s_j$ are defined by \eqref{eq:update pqrs}. Then we have that $t_j < t_{j+1}<\mu_i$ and $\lim_{j\to\infty} = \mu_i$. Moreover, the rate of convergence is quadratic, meaning that for any $j$ sufficiently large, 
$|t_{j+1}-\mu_i| \leq C|t_j-\mu_i|^2$, where $C$ is an absolute constant independent of iteration. 
\end{theorem}
It remains to determine an initial guess $t_0$ such that $t_0\in(0,\mu_i)$. Recall that $\omega_i(\mu_i)=0$, which is equivalent to 
\[
1+\sum_{j=1,j\neq i,i+1}^{k}\frac{\zeta_j^2}{\delta_j-\mu_i} + \frac{\zeta_{i+1}^2}{\delta_{i+1}-\mu_i} = \frac{\zeta_{i}^2}{\mu_i}.
\]
Since $\mu_i < \delta_{i+1}$, we can define $t_0$ to be the positive solution of the equation
\[
1+\sum_{j=1,j\neq i,i+1}^{k}\frac{\zeta_j^2}{\delta_j-\delta_{i+1}} + \frac{\zeta_{i+1}^2}{\delta_{i+1}-t_0} = \frac{\zeta_{i}^2}{t_0}.
\]
By monotonicity, we have that $t_0\in(0,\mu_i)$, as desired.
\item Now we assume that $i=k$. In this case, $\phi_k(t)=0$ and we want to solve the equation $-\psi_k(t)=1$. Theorem \ref{thm:quadratic convergence of eigenvalue computation} is still valid, and the update \eqref{eq:update t2} can be simplified as
\begin{equation}\label{eq:update t2 bound case}
t_{j+1} = t_j + \l(\frac{1+\psi_k(t_j)}{\psi_k'(t_j)}\r)\psi_k(t_j).
\end{equation}
To choose $t_0\in(0,\mu_k)$, we again recall that $\omega_k(\mu_k)=0$, which is equivalent to
\[
1 - \frac{\zeta_k^2}{\mu_k} + \sum_{j=1}^{k-1}\frac{\zeta_j^2}{\delta_j-\mu_k} = 0.
\]
Since $\mu_k < 1$, we define $t_0$ to be the solution of 
\[
1 - \frac{\zeta_k^2}{t_0} + \sum_{j=1}^{k-1}\frac{\zeta_j^2}{\delta_j-1} = 0.
\]
By monotonicity, we have that $t_0<\mu_k$. Moreover, note that $\sum_{j=1}^{k-1}\zeta_j^2\leq\l\|z\r\|_2=1$ and $\delta_j < 0$, $\forall j=1,\ldots,k-1$, so $1+\sum_{j=1}^{k-1}\frac{\zeta_j^2}{\delta_j-1}>0$. Therefore, $t_0\in(0,\mu_k)$, as desired. 
\end{enumerate}
\endgroup

%=======================================
% Section: Auxiliary technical results
%=======================================
\section{Auxiliary technical results}
\label{sec:aux tech results for chapter 2}

	%=======================================
	% Subsection: Derivation of penalized Huber from double penalized estimator
	%=======================================
\subsection{Detailed derivation of the claim of Remark \ref{relation to penalized huber}}
\label{sec:derivation of penalized huber}
In this section we present the detailed derivation of Remark \ref{relation to penalized huber}. First, consider the function as follows 
\[
F(S,U_1,\cdots,U_n) := \frac{1}{2}\sum_{i=1}^{n}\l\|Y_iY_i^T - S - U_i \r\|_F^2 + \lambda_1 \l\|S\r\|_1 + \lambda_2\sum_{i=1}^{n}\l\|U_i\r\|_1.
\]
For a fixed matrix $S$, the matrix $Y_iY_i^T - S$ has a spectral decomposition
\[
Y_iY_i^T-S=\sum_{j=1}^{d}\lambda_j^{(i)}v_j^{(i)}v_j^{(i)T}, \text{ for all } i=1,\ldots,n,
\]
where $\lambda_j^{(i)}$ is the $j$-th eigenvalue of $Y_iY_i^T - S$ and $v_j^{(i)}$ is the corresponding eigenvector.
We claim that $F(S,\cdot)$ can be minimized by choosing
\begin{equation}\label{appendix E.5.1:eq_1}
\wt{U}_i = \sum_{j=1}^d \sign(\lambda_j^{(i)})\l(|\lambda_j^{(i)}| - \lambda_2 \r)_+ v_j^{(i)}v_j^{(i)T} = \gamma_{\lambda_2}(Y_iY_i^T-S),
\end{equation}
where $\gamma_{\lambda}(u) := \sign (u)\l( |u| - \lambda \r)_+, \ \forall u\in \mb R,\lambda\in \mb R^{+}$, and $(x)_+ := \max(x,0)$. 
Indeed, note that $(U_1,\ldots,U_n) \mapsto G_S(U_1,\ldots,U_n):=F(S,U_1,\ldots,U_n)$ is strictly convex, so a sufficient and necessary condition for $(\wt{U}_1,\ldots,\wt{U}_n)$ to be a point of minimum is 
\begin{equation*}
\textbf{0}\in \partial G_S(\wt{U}_1,\ldots,\wt{U}_n) = \Bigg( -\l(Y_1Y_1^T-S-U_1\r) + \lambda_2\wt{V}_1,
\ldots ,-\l(Y_nY_n^T-S-U_n\r) + \lambda_2\wt{V}_n\Bigg) , 
\end{equation*}
where $\wt{V}_i \in \partial\l\|\wt{U}_i\r\|,i=1,\ldots,n$.
By choosing $\wt{V}_i:=\sum_{j:|\lambda_j^{(i)}|>\lambda_2}\sign(\lambda_j^{(i)}) v_j^{(i)}v_j^{(i)T} + \sum_{j:|\lambda|_j^{(i)}\leq\lambda_2} \frac{\lambda_j^{(i)}}{\lambda_2}v_j^{(i)}v_j^{(i)T} \in \partial \l\|\wt{U}_i\r\|_1$, it is easy to verify that $\partial G_S(\wt{U}_1,\ldots,\wt{U}_n) = \textbf{0}$, hence $(\wt{U}_1,\ldots,\wt{U}_n)$ is the minimizer. Plugging in to $F(S,U_1,\ldots,U_n)$, we get that
{\allowdisplaybreaks
\begin{multline}\label{p-huber}
F(S,\wt{U}_1,\cdots,\wt{U}_n) = \frac{1}{2}\sum_{i=1}^{n}\l\|Y_iY_i^T - S - \wt{U}_i \r\|_F^2 + \lambda_1 \l\|S\r\|_1 + \lambda_2 \sum_{i=1}^{n}\l\|\wt{U}_i\r\|_1 \\
= \frac{1}{2}\sum_{i=1}^{n}\l\|\sum_{j=1}^{d}\l[ \lambda_j^{(i)} - \gamma_{\lambda_2}(\lambda_j^{(i)}) \r]v_j^{(i)}v_j^{(i)T}  \r\|_F^2 + \lambda_2\sum_{i=1}^n\sum_{j=1}^d \gamma_{\lambda_2}(\lambda_j^{(i)}) + \lambda_1\l\|S\r\|_1 \\
=\sum_{i=1}^n\l( \sum_{j:|\lambda_j^{(i)}|>\lambda_2}(\lambda_2|\lambda_j^{(i)}| - \frac{\lambda_2^2}{2}) + \sum_{j:|\lambda_j^{(i)}|\leq\lambda_2}\frac{\lambda_j^{(i)2}}{2} \r)+ \lambda_1\l\|S\r\|_1 \\
= \tr \l( \sum_{i=1}^n \rho_{\lambda_2}(Y_iY_i^T - S) \r) + \lambda_1\l\|S\r\|_1,
\end{multline}
}
where 
\[
\rho_\lambda(u) = \l\{
			    \begin{array}{ll}
			    	\frac{u^2}{2},\quad\l|u\r|\leq\lambda\\
			    	\lambda\l|u\r|-\frac{\lambda^2}{2},\quad\l|u\r|>\lambda
			    \end{array}
			    \r.
\]
is the Huber's loss function. Note that our loss function $L(S,\boldsymbol{U_{I_n^2}})$ can be expressed as
\begin{multline*}
L(S,\boldsymbol{U_{I_n^2}}) = \frac{1}{N}\sum_{i\neq j}\l\|\wt{Y}_{i,j}\wt{Y}_{i,j}^T-S-\sqrt{N}U_{i,j}\r\|_F^2+\lambda_1\l\|S\r\|_1 + \lambda_2 \sum_{i\neq j}\l\|U_{i,j}\r\|_1 \\
= \frac{2}{N}\Bigg[ \frac{1}{2}\sum_{i\neq j} \l\| \widetilde{Y}_{i,j} \widetilde{Y}_{i,j} ^T - S - \sqrt{N}U_{i,j} \r\|^2_\F  + \frac{\sqrt{N}\lambda_2}{2}\sum_{i\neq j} \l\|\sqrt{N}U_{i,j} \r\|_1 \Bigg]+ \lambda_1 \l\|S\r\|_1.
\end{multline*}
Therefore, \eqref{p-huber} implies that 
\begin{multline*}
\min_{S,\boldsymbol{U_{I_n^2}}}L(S,\boldsymbol{U_{I_n^2}}) = \min_{S}\min_{\boldsymbol{U_{I_n^2}}}L(S,\boldsymbol{U_{I_n^2}}) = \min_{S} L(S,\boldsymbol{\wt{U}_{I_n^2}})\\
= \min_{S} \l\{\frac{2}{N}\tr \l( \sum_{i\neq j} \rho_{\frac{\sqrt{N}\lambda_2}{2}}(\wt{Y}_{i,j}\wt{Y}_{i,j}^T - S) \r) + \lambda_1\l\|S\r\|_1  \r\},
\end{multline*}
where
\[
\wt{U}_{i,j} = \frac{1}{\sqrt{N}}\sum_{k=1}^d \sign(\lambda_k^{(i,j)})\l(|\lambda_k^{(i,j)}| - \frac{\sqrt{N}\lambda_2}{2} \r)_+ v_k^{(i,j)}v_k^{(i,j)T} \\
= \gamma_{\frac{\sqrt{N}\lambda_2}{2} }(\wt{Y}_{i,j}\wt{Y}_{i,j}^T-S)
\]
with $\wt{Y}_{i,j}\wt{Y}_{i,j}^T-S=\sum_{k=1}^{d}\lambda_k^{(i,j)}v_k^{(i,j)}v_k^{(i,j)T}, \text{ for all } i=1,\ldots,n$.

	%=======================================
	% Subsection: Holder continuity of rho'()
	%=======================================
%\subsubsection{H\"older continuity of $\rho'$}
%\label{sec:holder continuity}
%Recall that 
%\[
%\rho'(x) = \l\{
%			    \begin{array}{ll}
%			    	x,\quad\l|x\r|\leq1\\
%			    	\sign(x),\quad\l|x\r|>1
%			    \end{array}
%			    \r. \quad \forall x\in \mb R
%\]
%Assume that $0<\alpha\leq 1$.
%\begin{enumerate}[(i)]
%\item if $x,y>1$ or $x,y<1$, then $|\rho'(x)-\rho'(y)|=0\leq|x-y|^\alpha$.
%\item if $x,y\in[-1,1]$, then 
%\[
%\frac{|\rho'(x)-\rho'(y)|}{|x-y|^\alpha} = |x-y|^{1-\alpha} \l\{
%		            \begin{array}{ll}
%			    	\leq 1,\quad\text{if }|x-y|\leq1\\
%			    	\leq 2,\quad\text{if }1<|x-y|\leq 2
%			    \end{array}
%			    \r.
%\]
%\item if $x\in[-1,1]$ and $y>1$, then
%\[
%\frac{|\rho'(x)-\rho'(y)|}{|x-y|^\alpha} = \frac{|x-1|}{|x-y|^{\alpha}}
%\]
%\begin{itemize}
%\item	if $|x-y|<1$, then $\frac{|x-1|}{|x-y|^{\alpha}}\leq |x-y|^{1-\alpha}<1$.
%\item if $|x-y|\geq1$,then $\frac{|x-1|}{|x-y|^{\alpha}}\leq 2|x-y|^{-\alpha}\leq 2$.
%Similarly we can bound $\frac{|\rho'(x)-\rho'(y)|}{|x-y|^\alpha}$ on other intervals. 
%\end{itemize}
%\end{enumerate}
%Therefore, we have that $|\rho'(x)-\rho'(y)|\leq 2|x-y|^\alpha$ for any $x,y\in\mb{R}$.
%

	%=======================================
	% Section: Proof of the Lemma for truncated variance
	%=======================================
\subsection{Proof of Lemma \ref{lemma:truncated covariance}}
\label{sec:proof or heavy-tailed Frobenius bound}
%\begin{proof}
\begingroup
\allowdisplaybreaks
We denote $\Sigma_Y = \Sigma$, which is valid throughout this proof only. The proof of relations  (\ref{truncated error operator}, \ref{truncated error trace}) was presented in \citet[Lemma 2.1]{mendelson2020robust} with constants $c(K) = K^3$ and $c(K) = K^4$ respectively. For \eqref{truncated error frobenius}, assume that $\Sigma_Z-\Sigma$ has eigenvalues $\lambda_1\leq\cdots\leq \lambda_d$ with corresponding orthonormal eigenvector set $\{u_1,\ldots,u_d \}$. Define $T_1 = 0$ and $T_{j} = \sum_{l=1}^{j-1}\lambda_l, j=2,\ldots,d+1$. Then $\lambda_j = T_{j+1}-T_{j}$, and we have that
\begin{equation*}
\l\|\Sigma_Z-\Sigma\r\|_F^2 = \sum_{j=1}^{d}\lambda_j^2 = \sum_{j=1}^{d}\lambda_j(T_{j+1}-T_{j}).
\end{equation*}
Summation by parts implies that
\begin{multline*}
\sum_{j=1}^{d}\lambda_j(T_{j+1}-T_{j}) = (\lambda_d T_{d+1} - \lambda_1T_1) - \sum_{j=2}^{d}T_j(\lambda_j-\lambda_{j-1})\\
= \lambda_d T_{d+1}- \sum_{j=2}^{d}T_j(\lambda_j-\lambda_{j-1}) \leq |\lambda_d||T_{d+1}| + \l| \sum_{j=2}^{d}T_j(\lambda_j-\lambda_{j-1})\r|.
\end{multline*}
Since $\lambda_j - \lambda_{j-1}\geq0$, we have that
\[
\l| \sum_{j=2}^{d}T_j(\lambda_j-\lambda_{j-1})\r| \leq \max_{2\leq j\leq d}|T_j| \sum_{j=2}^{d}(\lambda_j-\lambda_{j-1}) = (\lambda_d - \lambda_1)\max_{2\leq j\leq d}|T_j|,
\]
hence 
\begin{equation}
\l\|\Sigma_Z-\Sigma\r\|_F^2  \leq |\lambda_d||T_{d+1}| + (\lambda_d - \lambda_1)\max_{2\leq j\leq d}|T_j| \leq 2\l\|\Sigma_Z-\Sigma\r\|\max_{2\leq j\leq d+1}|T_j|.
\end{equation}
It remains to bound $\max_{2\leq j\leq d+1}|T_j|$. Note that for $j=2,\ldots,d+1$,
\begin{multline*}
|T_j| = \l|\sum_{l=1}^{j-1}\lambda_l\r| = \l| \sum_{l=1}^{j-1}\dotp{(\Sigma_Z-\Sigma)u_l}{u_l} \r| = \l| \sum_{l=1}^{j-1}\dotp{\Sigma_Zu_l}{u_l} -\dotp{\Sigma u_l}{u_l} \r| \\
=  \l| \sum_{l=1}^{j-1}\expect{\dotp{Y}{u_l}^2\mathds{1}\{\l\|Y\r\|_2\leq R\}} -\expect{\dotp{Y}{u_l}^2} \r|
= \l| \sum_{l=1}^{j-1} \expect{\dotp{Y}{u_l}^2\mathds{1}\{\l\|Y\r\|_2 > R\}} \r|.
\end{multline*}
Applying Cauchy-Schwartz inequality and $L_4-L_2$ norm equivalence, we deduce that
\begin{multline}\label{Z_and_Y:eq_1}
|T_j| \leq \sum_{l=1}^{j-1} \expect{\dotp{Y}{u_l}^4}^{\frac{1}{2}} P(\l\|Y\r\|_2 > R)^{\frac{1}{2}} \leq \sum_{l=1}^{j-1} K^2 \expect{\dotp{Y}{u_l}^2} P(\l\|Y\r\|_2 > R)^{\frac{1}{2}} \\
\leq  K^2 P(\l\|Y\r\|_2 > R)^{\frac{1}{2}} \sum_{l=1}^{d} \expect{\dotp{Y}{u_l}^2} .
\end{multline}
Observe that $\{u_1,\ldots,u_d\}$ is an orthonormal set on $\mb R^{d}$, so Parseval's identity implies that
\begin{equation}\label{Z_and_Y:eq_2}
\sum_{l=1}^{d} \expect{\dotp{Y}{u_l}^2} = \expect{\l\|Y\r\|_2^2} = \tr(\expect{Y^TY}) = \expect{\tr(YY^T)} = \tr(\expect{YY^T}) = \tr(\Sigma).
\end{equation}
On the other hand, applying Cauchy-Schwartz inequality and $L_4-L_2$ norm equivalence again, we have that
\begin{multline*}
\expect{\l\|Y\r\|_2^4} = \expect{\Big(\sum_{j=1}^{d}\dotp{Y}{e_j}^2\Big)^2} =\expect{\sum_{i,j}\dotp{Y}{e_i}^2\dotp{Y}{e_j}^2} \\
\leq \sum_{i,j}\expect{\dotp{Y}{e_i}^4}^{\frac{1}{2}}\expect{\dotp{Y}{e_j}^4}^{\frac{1}{2}} \leq K^4 \sum_{i,j}\expect{\dotp{Y}{e_i}^2}\expect{\dotp{Y}{e_j}^2} 
= K^4 \sum_{i,j}\Sigma_{i,i}\Sigma_{j,j} =K^4\tr(\Sigma)^2.
\end{multline*}
Markov's inequality implies that
\begin{equation}\label{Z_and_Y:eq_3}
P(\l\|Y\r\|_2 > R)^{\frac{1}{2}} \leq \l( \frac{\expect{\l\|Y\r\|_2^4}}{R^4}\r)^{\frac{1}{2}} \leq K^2\frac{\tr(\Sigma)}{R^2}.
\end{equation}
Combining (\ref{Z_and_Y:eq_1}, \ref{Z_and_Y:eq_2}, \ref{Z_and_Y:eq_3}) together, we have that
\[
|T_j| \leq K^2 \cdot K^2\frac{\tr(\Sigma)}{R^2} \cdot \tr(\Sigma) = K^4\frac{\tr(\Sigma)^2}{R^2}
\]
for $j=2,\ldots,d+1$. Therefore, 
\begin{equation*}
\l\|\Sigma_Z-\Sigma\r\|_F^2 \leq 2\l\|\Sigma_Z-\Sigma\r\|\max_{2\leq j\leq d+1}|T_j| 
\leq 2\cdot K^3\frac{\l\|\Sigma\r\| \tr(\Sigma)}{R^2} \cdot K^4\frac{\tr(\Sigma)^2}{R^2} = 2K^7\frac{\l\|\Sigma\r\|\tr(\Sigma)^3}{R^4},
\end{equation*}
hence
\begin{equation*}
\l\|\Sigma_Z-\Sigma\r\|_F \leq \sqrt{2}K^{\frac{7}{2}} \frac{\l\|\Sigma\r\|^{\frac{1}{2}}\tr(\Sigma)^{\frac{3}{2}}}{R^2} = c(K) \frac{\l\|\Sigma\r\|^2\rk(\Sigma)^\frac{3}{2}}{R^2},
\end{equation*}
as desired.
\endgroup
%\end{proof}

%\endgroup

\end{appendix}

\end{document}